\pdfoutput=1
\documentclass[11pt,a4paper]{amsart}
\usepackage[left=25truemm,right=25truemm,top=30truemm,bottom=30truemm]{geometry}
\usepackage{amsmath, amsfonts, amssymb, amsthm, amscd, mathtools}
\usepackage{ascmac}
\usepackage{comment}
\usepackage{pb-diagram}
\usepackage{bm}
\usepackage{xcolor}
\usepackage{enumitem}  
\usepackage{hyperref}
\usepackage{mathrsfs}
\usepackage{physics}
\usepackage{url}
\providecommand{\doi}[1]{\href{https://doi.org/#1}{\nolinkurl{https://doi.org/#1}}}

\numberwithin{equation}{section}

\newtheorem{Thm}{Theorem}[section]
\newtheorem{Lem}[Thm]{Lemma}
\newtheorem{Prop}[Thm]{Proposition}
\newtheorem{Cor}[Thm]{Corollary}
\theoremstyle{definition}
\newtheorem{Def}[Thm]{Definition}
\newtheorem{Rem}[Thm]{Remark}
\newtheorem{Ex}[Thm]{Example}
\newtheorem{problem}{Problem}
\newtheorem{Question}{Question}

\newlist{romanenum}{enumerate}{2}
\setlist[romanenum,1]{label=(\roman*)}
\setlist[romanenum,2]{label=(\alph*)}

\newcommand{\RR}{\mathbb R}

\newcommand{\RRH}{\mathbb{R}\mathrm{H}}

\newcommand{\rest}{{\mathop{\mathrm{rest}}}}
\newcommand{\Sym}{\mathrm{Sym}}
\newcommand{\No}{\mathcal{N}_0}
\newcommand{\Ad}{\mathrm{Ad}}

\newcommand{\Aut}{\mathrm{Aut}}
\newcommand{\Stat}{\mathrm{Stat}}
\newcommand{\LStat}{\mathcal{L}\mathrm{Stat}}
\newcommand{\MLStat}{\mathcal{ML}\mathrm{Stat}}

\newcommand{\MHSStat}{\mathcal{M}_{\mathop{\mathbf{HS}}}\mathrm{Stat}}
\newcommand{\MStatStat}{\mathcal{M}_{\mathop{\mathbf{Stat}}}\mathrm{Stat}}
\newcommand{\Isom}{\mathrm{Isom}}
\newcommand{\Diff}{\mathrm{Diff}}

\newcommand{\Sect}{\mathrm{Sect}}

\newcommand{\HS}{\mathop{\mathbf{HS}}}
\newcommand{\HMFD}{\mathop{\mathbf{HMFD}}}
\newcommand{\STAT}{\mathop{\mathbf{Stat}}}
\newcommand{\HRiem}{\mathop{\mathbf{HRiem}}}
\newcommand{\Inn}{\mathrm{Inn}}

\newcommand{\Core}{\mathrm{Core}}
\newcommand{\Homog}{\mathrm{Homog}}

\title[Invariant statistical connections on the multivariate centered Gaussian model and their moduli spaces]{Invariant statistical connections on the multivariate centered Gaussian model and their moduli spaces}
\author{Hideyuki Ishi, Hikozo Kobayashi, Takayuki Okuda}
    \subjclass[2020]{
    Primary 53C30,  
    Secondary 53B12, 53C15, 53A15}
    \keywords{statistical manifold; statistical connection; Hessian manifold; multivariate normal distribution; moduli space}
    \thanks{
    The first author was supported by JSPS KAKENHI, Grant Number JP24KK0059. 
    The second author was supported by JST SPRING, Grant Number JPMJSP2132.
    The third author was supported by JSPS KAKENHI, Grant Number JP24K06714.
    This work was partly supported by MEXT Promotion of Distinctive Joint Research Center Program JPMXP0723833165 and Osaka Metropolitan University Strategic Research Promotion Project (Development of International Research Hubs).
    }
\address[H.~Ishi]{%
    Department of Mathematics, Graduate School of Science, 
    Osaka Metropolitan University, 
    3-3-138, Sugimoto, Sumiyoshi-ku, Osaka, 558-8585, Japan
        }
\email{hideyuki-ishi@omu.ac.jp}
\address[H.~Kobayashi]{%
    Mathematics Program, Graduate School of Advanced Science and Engineering, Hiroshima University, 
    1-3-1 Kagamiyama, Higashi-Hiroshima City, Hiroshima, 739-8526, Japan
        }
\email{hikozo-kobayashi@hiroshima-u.ac.jp}
\address[T.~Okuda]{%
    Mathematics Program, Graduate School of Advanced Science and Engineering, Hiroshima University, 
    1-3-1 Kagamiyama, Higashi-Hiroshima City, Hiroshima, 739-8526, Japan
        }
\email{okudatak@hiroshima-u.ac.jp}
\begin{document}

\begin{abstract}  
We study invariant statistical connections on the space $\mathcal{N}_0^n$ of zero-mean multivariate normal distributions (the multivariate centered Gaussian model) equipped with the Fisher metric $g^F$. We introduce moduli spaces of invariant statistical connections on homogeneous Riemannian manifolds via two natural equivalence relations arising from a categorical viewpoint, and apply this framework to $(\mathcal{N}_0^n, g^F)$. We explicitly determine the $GL(n,\mathbb{R})$-invariant and $\mathrm{Isom}(\mathcal{N}_0^n, g^F)$-invariant statistical connections, with particular emphasis on the dually flat case, and describe the corresponding moduli spaces.
\end{abstract}

\maketitle

\tableofcontents

\section{Introduction}\label{sec:Intro}

In differential geometry, the study of geometric structures invariant under group actions on spaces is a fundamental topic. 
In particular, homogeneous spaces and Lie groups serve as natural starting points for such investigations, and many foundational results have been obtained in this setting.

The aim of this paper is to address the following question.

\begin{Question}
Let $(M,g)$ be a Riemannian manifold endowed with a symmetry, namely an isometric smooth action of a Lie group $G$ on $(M,g)$.
How does this symmetry constrain statistical connections on $(M, g)$?
\end{Question}

A \emph{statistical connection} on $(M,g)$ is a torsion-free affine connection $\nabla$ on $M$ such that $\nabla g$ is totally symmetric.
Statistical connections, together with the associated statistical structures $(g,\nabla)$, arise in information geometry, as well as in submanifold theory, affine hypersurface theory, and in Hessian geometry, 
and they constitute an important class of geometric structures (see, for example, \cite{Amari-Nagaoka_2000, Ivanov_1995, Kurose_1994, Matsuzoe_2010, Shima}).
More recently, statistical structures (and statistical connections) have also appeared in the context of geometric analysis; see, for example,~[Fujitani~(2025), arXiv:2507.00385].

In these fields, statistical connections satisfying additional properties are often of particular interest.
For instance, in information geometry, flat statistical connections play a central role (see, e.g.,~\cite{Amari_1985, Amari-Nagaoka_2000, Ay_2017}); a statistical structure $(g,\nabla)$ with $\nabla$ flat is called a \emph{dually flat structure} (or \emph{Hessian structure}).
Other examples of statistical connections with special properties include those satisfying conjugate symmetry, constant curvature, or Einstein-type conditions with respect to the metric $g$.
In various contexts, such statistical connections have been extensively studied.

In this paper, we focus on the case where $M$ is a homogeneous Riemannian manifold and consider the following problem.

\begin{problem}\label{prob:Stat_determine}
Let $G$ be a Lie group acting transitively and isometrically on a Riemannian manifold $(M,g)$.  
For a given property $\mathrm{P}$ of $G$-invariant statistical connections on $(M,g)$ (for example, the dually flat property, etc.), determine the set
\[
\Stat^{G}_{\mathrm{P}}(M,g)
\]
of all $G$-invariant statistical connections on $(M,g)$ satisfying $\mathrm{P}$.
\end{problem}

In this paper, we mainly consider the cases where the property $\mathrm{P}$ is given by the \emph{dually flat property} $\mathrm{DF}$ or \emph{conjugate symmetry} $\mathrm{CS}$ (see Definition~\ref{def:CS-stat} for the definition of conjugate symmetric statistical structure).

Problem~\ref{prob:Stat_determine} provides a useful framework for understanding distinguished classes of statistical structures from the perspective of symmetry and for constructing and characterizing important concrete examples.

As a motivating example, we recall below important previous results related to Problem~\ref{prob:Stat_determine} (see Section~\ref{subsec:N-review} for details).

\begin{Thm}[\cite{FIK} for $n = 1$ and \cite{KO} for $n \ge 2$]
Let $\mathcal{N}^n$ denote the space of $n$-dimensional normal distributions, and let $g^F$ be the Fisher metric on $\mathcal{N}^n$. 
Then the following equalities hold:
\begin{align*}
    \Stat^{\mathrm{Aff}^s(n,\mathbb{R})}_{\mathrm{CS}}(\mathcal{N}^n,g^F) 
    &= \{ \nabla^{A(\alpha)} \mid \alpha \in \mathbb{R} \}, \\ 
    \Stat^{\mathrm{Aff}^s(n,\mathbb{R})}_{\mathrm{DF}}(\mathcal{N}^n,g^F)
    &=
    \{\nabla^{A(+1)},\,\nabla^{A(-1)}\}.
\end{align*}
Here, $\nabla^{A(\alpha)}$ denotes the Amari--Chentsov $\alpha$-connection on $\mathcal{N}^n$.
\end{Thm}

In the present paper, we investigate Problem~\ref{prob:Stat_determine} for the family of zero-mean multivariate normal distributions
\[
    \mathcal{N}^n_0 := \{ N(x;\Sigma,0) \mid \Sigma \in \Sym^+(n, \RR)\},
\]
which is identified with the space $\Sym^+(n,\mathbb{R})$ of $n \times n$ positive definite symmetric matrices and the Fisher metric $g^F$ on $\No^n$.

We consider two Lie groups acting transitively and isometrically on the Riemannian manifold $(\mathcal{N}^n_0,g^F)$, namely
\[
GL(n,\mathbb{R}) \quad \text{and} \quad \Isom(\mathcal{N}^n_0,g^F).
\]
We focus on three cases for the property $\mathrm{P}$: the unrestricted case, 
the conjugate symmetry~$\mathrm{CS}$ 
and the dually flat property~$\mathrm{DF}$.
One of our main results is the explicit determination of the following sets:
\[
\Stat^{GL(n,\mathbb{R})}(\mathcal{N}^n_0,g^F), \quad 
\Stat^{GL(n,\mathbb{R})}_{\mathrm{CS}}(\mathcal{N}^n_0,g^F), \quad 
\Stat^{GL(n,\mathbb{R})}_{\mathrm{DF}}(\mathcal{N}^n_0,g^F),
\]
\[
\Stat^{\Isom(\mathcal{N}^n_0,g^F)}(\mathcal{N}^n_0,g^F), \quad
\Stat^{\Isom(\mathcal{N}^n_0,g^F)}_{\mathrm{CS}}(\mathcal{N}^n_0,g^F), \quad
\Stat^{\Isom(\mathcal{N}^n_0,g^F)}_{\mathrm{DF}}(\mathcal{N}^n_0,g^F)
\]
(see Theorems~\ref{thm:Main-thm-G=GL+} and \ref{thm:Main-thm-G=Isom} for details).
Note that for the space $(\mathcal{N}_0, g^F)$, every $GL(n,\mathbb{R})$-invariant statistical connection is automatically conjugate symmetric (see Proposition~\ref{prop:GL+-stat_CS} for details). 
That is, the following equalities hold:
\begin{align*}
    \Stat^{GL(n,\mathbb{R})}(\mathcal{N}^n_0,g^F) &=
    \Stat^{GL(n,\mathbb{R})}_{\mathrm{CS}}(\mathcal{N}^n_0,g^F),\\
    \Stat^{\Isom(\mathcal{N}^n_0,g^F)}(\mathcal{N}^n_0,g^F) &=
    \Stat^{\Isom(\mathcal{N}^n_0,g^F)}_{\mathrm{CS}}(\mathcal{N}^n_0,g^F).
\end{align*}

As an illustration of our main results, we explicitly determine all $GL(n,\mathbb{R})$-invariant dually flat statistical connections on $(\mathcal{N}^n_0, g^F)$.

\begin{Thm}[cf.~Theorem~\ref{thm:Main-thm-G=GL+}~\ref{thm:Main-thm-G=GL_Hesse}]\label{thm:Stat-DF-N0_intro}
The following equality holds:
\[
    \Stat^{GL(n, \RR)}_{\mathrm{DF}}(\mathcal{N}_0^n, g^F)
    =
    \begin{cases}
    \{\nabla^{A(+1)},\ \nabla^{A(-1)},\ \nabla',\ (\nabla')^\ast \} & (n \ge 3),\\[2mm]
    \{\nabla^{A(+1)},\ \nabla^{A(-1)}\} & (n = 2),\\[2mm]
    \{\nabla^{A(\alpha)} \mid \alpha \in \RR\} & (n = 1).
    \end{cases}
\]
Here, $\nabla^{A(\alpha)}$ denotes the Amari--Chentsov $\alpha$-connection on $\mathcal{N}_0^n$.
Moreover, the affine connection $\nabla'$ is defined by Equation~\eqref{eq:def-C'} and is different from $\nabla^{A(\alpha)}$ for any $\alpha \in \RR$.
\end{Thm}

We note that Theorem~\ref{thm:Stat-DF-N0_intro} has an interesting consequence in the case $n \geq 3$. 
Namely, besides the Amari--Chentsov connection $\nabla^{A(+1)}$ and its dual $\nabla^{A(-1)}$, 
there exist $GL(n,\RR)$-invariant dually flat statistical connections on $(\No^n,g^F)$, namely $\nabla'$ and its dual $(\nabla')^*$, which are different from them.

As a supplementary remark, we record the canonical divergences of these Hessian structures in the sense of~\cite{Shima}.
It is well known that the canonical divergence of the Hessian structure $(g^F, \nabla^{A(-1)})$ is the Kullback--Leibler divergence (KL-divergence) on $\No$ given by
\[
    D_{\mathrm{KL}}(\Sigma_1, \Sigma_2)
        = \frac{1}{2} \left(
        \tr(\Sigma_1 \Sigma_2^{-1})
        -
        \log\det(\Sigma_1 \Sigma_2^{-1})
        -
        n \right)
\]
(cf.~\cite{Amari_2014, Shima-Hao_2000}). 
On the other hand, the canonical divergence corresponding to $(g^F, (\nabla')^\ast )$ is given by
\[
D'(\Sigma_1, \Sigma_2)
    = \frac{1}{2}  \left(
    \det(\Sigma_1\Sigma_2^{-1})^{-2/n}\,
    \tr(\Sigma_1\Sigma_2^{-1})
    -
    \log\det(\Sigma_1^{-1}\Sigma_2)
    -
    n \right)
\]
(see Proposition~\ref{prop:Amari-open-prob} for details).

These observations naturally lead to the following question.

\begin{Question}\label{q:DF-connections_diff}
Are the following four dually flat statistical connections
\[
    \nabla^{A(+1)}, \quad \nabla^{A(-1)}, \quad \nabla', \quad (\nabla')^\ast
\]
``essentially'' different?
Likewise, are divergences $D_{\mathrm{KL}}$ and $D'$ ``essentially'' different?
\end{Question}

Motivated by this question, we study moduli spaces of invariant statistical connections on homogeneous Riemannian manifolds. 
The point is that, although the connections appearing in Theorem~\ref{thm:Stat-DF-N0_intro} are distinct as affine connections, they may become equivalent once one passes to a suitable notion of geometric equivalence. 
Our aim is to formulate such equivalence relations in a natural way and to determine the resulting moduli spaces.

Let $(M,g)$ be a homogeneous Riemannian manifold endowed with a
transitive isometric action of a Lie group $G$.
Given two $G$-invariant statistical connections $\nabla$ and $\nabla'$ on
$(M,g)$, a natural way to compare them is to ask whether there exists a
statistical isomorphism
\[
    f : (M,g,\nabla) \longrightarrow (M,g,\nabla')
\]
satisfying an additional compatibility condition.
Different choices of this compatibility condition give rise to
different equivalence relations on the space of invariant statistical
connections.

Several such notions may be considered.
One may require $f$ to be strictly $G$-equivariant.
This is the standard viewpoint when studying maps between $G$-manifolds
for a fixed Lie group $G$.
However, in the case of homogeneous spaces, equivariant maps are well known to be highly rigid: once the image of a single point is fixed, the equivariance condition typically determines the map uniquely.
As a consequence, the equivalence relation induced by strict
$G$-equivariance turns out to be too restrictive for our purposes and
does not lead to a meaningful reduction of the space of invariant
statistical connections (see also Remark \ref{Remark:strictequivariant} for details).

In this paper, we consider two types of additional conditions on a statistical isomorphism $f$ in order to define moduli spaces.

The first approach is to impose no additional condition on $f$. In this case, the moduli space classifies $G$-invariant statistical connections on $(M,g)$ satisfying the property $\mathrm{P}$ purely as statistical manifolds. We denote this moduli space by
\[
\MStatStat^G_{\mathrm{P}}(M,g),
\]
and refer to it as the \emph{$\STAT$-moduli space}.

The second approach is to impose a weakened form of $G$-equivariance. More precisely, a statistical isomorphism
\[
f : (M,g,\nabla) \longrightarrow (M,g,\nabla')
\]
is regarded as admissible if there exists an automorphism $\phi$ of $G$ such that
\[
f(h \cdot x) = \phi(h) \cdot f(x)
\qquad (h \in G,\; x \in M).
\]
This condition may be viewed as a relaxed notion of equivariance and expresses compatibility with the homogeneous structure up to automorphisms of the group. The moduli space defined in this way is denoted by
\[
\MHSStat^G_{\mathrm{P}}(M,g),
\]
and we refer to it as the \emph{$\HS$-moduli space}.
There is a natural surjective map
\[
\MHSStat^G_{\mathrm{P}}(M,g) \longrightarrow \MStatStat^G_{\mathrm{P}}(M,g).
\]

From a conceptual point of view, this construction fits naturally into the theory of homogeneous spaces. Indeed, if $M$ and $M'$ are homogeneous spaces of Lie groups $G$ and $G'$, respectively, it is natural to regard a map $f : M \to M'$ as compatible with the homogeneous structures whenever there exists a homomorphism $\phi : G \to G'$ satisfying
\[
f(h \cdot x) = \phi(h) \cdot f(x).
\]
This viewpoint leads to a natural category of homogeneous spaces and structure-compatible maps. In the present setting, we apply this idea to the case $G = G'$, allowing nontrivial automorphisms of $G$.

An important advantage of this equivalence relation is that it provides a richer class of admissible transformations than strict equivariance, while still retaining sufficient rigidity to make the resulting moduli spaces tractable. In fact, the equivalence relation can be described in terms of an action of the automorphism group of $G$ (see Section \ref{subsec:moduli-inv-stat-conn} for details), which makes it significantly more manageable than the equivalence relation defined by arbitrary statistical isomorphisms.

In particular, when $G$ acts simply transitively on $M$, we may identify $M$ with $G$ by fixing a base point. In this situation, $\Stat^G_{\mathrm{P}}(M,g)$ coincides with the set of left-invariant statistical connections on $(G,g)$ satisfying the property $\mathrm{P}$, and the moduli space $\MHSStat^G_{\mathrm{P}}(M,g)$ can be identified with the quotient of $\Stat^G_{\mathrm{P}}(G,g)$ by the natural action of the automorphism group $\Aut(G,g) := \{ \phi \in \Aut(G) \mid \phi \text{ preserves } g \}$ (see Section \ref{subsection:simplytransitive} for details). 
Thus, the $\HS$-moduli space introduced in this paper is based on the same philosophy as the moduli spaces of left-invariant geometric structures on Lie groups studied by Kodama--Takahara--Tamaru \cite{KTT} and related works \cite{
CCR_2025, Cosgaya-Reggiani_2022, Dekic-Babic-Vukmirovic_2022, Hashinaga_2014, Hashinaga-Tamaru_2017, HTT_2016, 
KONDO, KT_2023, KOTT, 
Luis_2022, Luis-T,
Sukilovic-Vukmirovic_2023, Sukilovic-et.al_2025, 
Taketomi_2015, Taketomi_2018, Taketomi-Tamaru_2018, Scott_2017, Torres_2024} (see Section~\ref{subsec:moduli-left-inv-geom-str} for a review of a selection of these studies).
In particular, \cite{KOOT_2025} studies moduli spaces of left-invariant statistical structures on Lie groups, and the $\HS$-moduli space introduced in the present paper follows this perspective.
Furthermore, in Inoguchi~\cite{Inoguchi_2026}, the set of left-invariant statistical connections and the moduli space that are compatible with various geometric structures on certain Lie groups are treated in detail from a more geometric perspective (similar to the $\STAT$-moduli space in this paper).

In this paper, we formalize these notions of equivalence and define the corresponding moduli spaces $\MStatStat^G_{\mathrm{P}}(M,g)$ and $\MHSStat^G_{\mathrm{P}}(M,g)$ of invariant statistical connections satisfying the property $\mathrm{P}$. 
We then apply this framework to several concrete examples, with particular emphasis on invariant statistical connections on $(\mathcal{N}_0^n, g^F)$.
In particular, we focus on the moduli spaces of $GL(n,\RR)$-invariant dually flat statistical connections on $(\mathcal{N}_0^n, g^F)$ and determine them completely. One of the main results of this paper is the following description.

\begin{Thm}[cf.~Theorem~\ref{thm:Main-thm-G=Isom}~\ref{thm:Main-thm-G=GL_MHesse} and Corollary~\ref{cor:HS_Stat-moduli_coincides_N0}]\label{thm:Stat-DF-moduli-N0_intro}
The following equality holds.
\[
   \MStatStat_{\mathrm{DF}}^{GL(n, \RR)}(\mathcal{N}_0^n, g^F) =  \MHSStat_{\mathrm{DF}}^{GL(n, \RR)}(\mathcal{N}_0^n, g^F)
        \simeq
        \begin{cases}
        \{*\} & (n \ge 2),\\[2mm]
        \RR_{\ge 0} & (n = 1).
        \end{cases}
\]
\end{Thm}

In particular, the above theorem shows that all $GL(n, \mathbb{R})$-invariant dually flat statistical
connections on $(\mathcal{N}^n_0, g^F)$ become equivalent under these notions of equivalence when $n \geq 2$.
As a further consequence, one verifies that the divergences $D'$ and $D_{\mathrm{KL}}$
are related by an isometry of $(\mathcal{N}^n_0, g^F)$ that is compatible with the $GL(n,\mathbb{R})$-action (see Proposition~\ref{prop:Amari-open-prob} for details).
Moreover, when $n\geq2$, the canonical divergences associated with
$GL(n,\RR)$-invariant dually flat statistical connections on
$(\mathcal N_0^n,g^F)$
are all related in the same way.

We note that the main theorems of the present paper,
as well as all the propositions in Section~\ref{subsec:main-results}, remain valid if one replaces \(\mathcal{N}_0^n\) and \(g^F\) by
\(\Sym^+(n,\RR)\) and the
\(GL(n,\RR)\)-invariant Riemannian metric \(g^{(r)}\) on
\(\Sym^+(n,\RR)\), respectively, where \(g^{(r)}\) is defined by
\[
    g^{(r)}_{\Sigma}(X,Y)
    =
    r \cdot \tr(\Sigma^{-1} X \Sigma^{-1} Y)
    \qquad
    (r>0,\ 
    X,Y \in T_{\Sigma}\Sym^+(n,\RR)
    \simeq \Sym(n,\RR)).
\]
In particular, when \(n \geq 3\), the set
\[
    \Stat_{\mathrm{DF}}^{GL(n,\RR)}
    (\Sym^+(n,\RR), g^{(r)})
\]
consists of four elements. Among them, two coincide with the Amari--Chentsov \((+1)\)-connection \(\nabla^{A(+1)}\) and the Amari--Chentsov \((-1)\)-connection \(\nabla^{A(-1)}\) induced from \(\mathcal{N}_0^n\).

Furthermore, let \(\mathcal{W}_\nu^p\) denote the family of Wishart distributions on \(\Sym^+(p,\RR)\) with degrees of freedom \(\nu \ge p\).
Then, its parameter space is identified with \(\Sym^+(p,\RR)\), and the Fisher metric \(g^{F\mathchar`-\mathcal{W}_\nu}\) on \(\mathcal{W}_\nu^p\) is given by
\[
    g^{F\mathchar`-\mathcal{W}_\nu}_{\Sigma}(X,Y)
    =
    \frac{\nu}{2}
    \tr(\Sigma^{-1} X \Sigma^{-1} Y)
\]
(cf.~\cite{Ayadi_2023}).
Therefore, the results of the present paper also apply to such statistical models, including the Wishart model \(\mathcal{W}_\nu^p\) with fixed degrees of freedom \(\nu \ge p\).

This paper is organized as follows: 
In Section~\ref{sec:pre}, we first review the basic notions of statistical manifolds in Section~\ref{subsec:stat-mfd}, and then study in detail the Riemannian manifold $(\No^n, g^F)$ in Section~\ref{subsec:N0}. 
We also review several known formulations of the moduli spaces of left-invariant geometric structures on $G$ in Section~\ref{subsec:moduli-left-inv-geom-str}. 
In Section~\ref{sec:Cat_Stat}, we introduce the category $\STAT$ of statistical manifolds in Section~\ref{subsec:Cat_Stat}, and define the moduli space $\MStatStat(M,g)$ of statistical connections on a Riemannian manifold $(M,g)$ in Section~\ref{subsection:DefStatModuli}. 
We also introduce the moduli space $\MStatStat^G(M,g)$ in this subsection. 
In Section~\ref{sec:mor-HS}, we study equivariant morphisms between homogeneous manifolds, introduce the categories of homogeneous manifolds $\HMFD$ and homogeneous Riemannian manifolds $\HRiem$ (Section~\ref{section:EM}),
and investigate their automorphism groups (Section~\ref{subsec_Autgrp}). 
These results provide a framework for describing equivalence relations compatible with the homogeneous structure, which will be used in the definition of moduli spaces $\MHSStat^G(M,g)$ in Section~\ref{sec:HS-moduli}. 
In Section~\ref{sec:HS-moduli}, we introduce the category $\HS$ of homogeneous statistical manifolds and define the moduli spaces $\MHSStat^G(M,g)$. 
In particular, in Section~\ref{subsec:Stat-HS-moduli_coincide}, we provide sufficient conditions under which the two moduli spaces $\MStatStat^G(M,g)$ and $\MHSStat^G(M,g)$ coincide, and we also discuss the simply transitive case in Section~\ref{subsection:simplytransitive}. 
Section~\ref{sec:main-results} is devoted to our main results.
In Section~\ref{subsec:Pre-for-mainresults}, we introduce notation used in the statements of the main theorems.
In Section~\ref{subsec:main-results}, we present our main theorems (Theorem~\ref{thm:Main-thm-G=GL+} and Theorem~\ref{thm:Main-thm-G=Isom}).
Finally, in Section~\ref{subsec:proof-of-main-thm} and Section~\ref{subsec:proof-of-main-thm-Isom}, we give their proofs.

\section{Preliminaries}\label{sec:pre}
Throughout this paper, $n$ denotes an integer greater than or equal to $1$.
Also, for a matrix $A$, we denote its transpose by ${}^t A$.

\subsection{Statistical manifolds}\label{subsec:stat-mfd}
\begin{Def}
Let $(M,g)$ be a Riemannian manifold.
A torsion-free affine connection $\nabla$ on $M$ is called a \emph{statistical connection}
if the covariant derivative $\nabla g$ is totally symmetric, that is,
\[
(\nabla_X g)(Y,Z) = (\nabla_Y g)(X,Z)
\quad\text{for all } X,Y,Z \in \Gamma(TM).
\]
A triple $(M,g,\nabla)$ is called a \emph{statistical manifold} and a pair $(g, \nabla)$ is called a \emph{statistical structure} on $M$.
\end{Def}

In this paper, we denote by $\nabla^{g}$ the Levi--Civita connection of the Riemannian metric $g$.
It is clear that, for any Riemannian manifold $(M, g)$, the pair $(g, \nabla^{g})$ defines a statistical structure on $M$.
The pair $(g, \nabla^{g})$ is called the \emph{trivial statistical structure}. In the context of information geometry, a particularly important class of statistical structures is the dually flat structure, defined as follows.

\begin{Def}
A statistical structure $(g, \nabla)$ on a manifold $M$ is called \emph{dually flat} if $\nabla$ is flat, that is, if its curvature tensor $R^{\nabla}$, defined by
\[
    R^{\nabla}(X,Y)Z
    = \nabla_X \nabla_Y Z
    - \nabla_Y \nabla_X Z
    - \nabla_{[X,Y]} Z,
\]
vanishes.
In this case, $(g,\nabla)$ is called a \emph{dually flat structure} on $M$,
and $(M,g,\nabla)$ is called a \emph{dually flat statistical manifold}
(or simply a \emph{dually flat manifold}).
Moreover, $\nabla$ is called a \emph{dually flat statistical connection} on the Riemannian manifold $(M,g)$.
\end{Def}

A dually flat structure is equivalent to a Hessian structure (see~\cite{Shima}).

\begin{Def}
Let $(M, g)$ be a Riemannian manifold and $\nabla$ be an affine connection on $M$. 
The affine connection $\nabla^{\ast}$ defined by the following equation is called the \emph{dual connection} of \( \nabla \) with respect to \( g \).
\[
    Xg(Y,Z) = g(\nabla_X Y, Z) + g(Y, \nabla^{\ast}_X Z).
\]
\end{Def}

For a Riemannian metric $g$ and an affine connection $\nabla$ on a manifold $M$, the pair $(g, \nabla)$ is a statistical structure if and only if $(g, \nabla^{\ast})$ is a statistical structure. 
Furthermore, the statistical structure $(g, \nabla)$ is dually flat if and only if the statistical structure $(g, \nabla^{\ast})$ is dually flat (see~\cite{L-SM}).

Throughout this paper, for a Riemannian manifold $(M,g)$,
we denote by $\Stat(M,g)$ the set of all statistical connections on $(M,g)$.
For $\nabla \in \Stat(M,g)$, the symmetric $(0,3)$-tensor field
$C^\nabla$, also denoted by $C^{(g,\nabla)}$, is defined by
\[
    C^\nabla(X,Y,Z)
    =
    C^{(g,\nabla)}(X,Y,Z)
    :=
    -2\,g\bigl(\nabla_XY-\nabla_X^gY,\,Z\bigr)
    \qquad
    (X,Y,Z\in\Gamma(TM)).
\]
This tensor is called the \emph{cubic form} associated with
the statistical manifold $(M,g,\nabla)$.

Let $\Gamma(S^3T^\ast M)$ denote the space of symmetric $(0,3)$-tensor fields on $M$.
Then the following result is classical:

\begin{Thm}[{\cite{L-SM}}]\label{theorem:nablaC}
Let $(M,g)$ be a Riemannian manifold.
The correspondence
\[
\Stat(M,g) \longrightarrow \Gamma(S^3T^\ast M),
\qquad
\nabla \longmapsto C^{\nabla}
\]
is a bijection.
\end{Thm}

Note that the Levi-Civita connection \(\nabla^g\) corresponds to the origin \(0 \in \Gamma(S^3 T^\ast M)\).
Moreover, for \(\nabla \in \mathrm{Stat}(M, g)\), the relation \(C^{\nabla^{\ast}} = -\,C^\nabla\) holds.

Throughout this paper, for a Riemannian manifold $(M, g)$ and a symmetric
$(0,3)$-tensor field $C$ on $M$, we denote by $\nabla^{(g, C)}$ the statistical connection on $(M, g)$ obtained from $C$ via the one-to-one correspondence in Theorem~\ref{theorem:nablaC}.
Moreover, we also refer to a pair consisting of a Riemannian metric \( g \) and a symmetric \((0,3)\)-tensor field \( C \) on a manifold \( M \), as well as to the triple \( (M, g, C) \), as a statistical structure and a statistical manifold, respectively.

As a class of statistical structures broader than dually flat structures,
there is the class of so-called \emph{conjugate symmetric} statistical structures.

\begin{Def}[\cite{L-SM}]\label{def:CS-stat}
Let $(g, C)$ be a statistical structure on a manifold $M$.
$(g, C)$ is called \textit{conjugate symmetric} if the $(0, 4)$-tensor field $\nabla^g C$ is totally symmetric.
\end{Def}

For a statistical structure $(g, \nabla)$ on a manifold $M$, being conjugate symmetric is equivalent to the condition $R^{\nabla} = R^{\nabla^{\ast}}$ (see~\cite{L-SM}). 
Moreover, every dually flat structure is conjugate symmetric (see~\cite{L-SM}).
The trivial statistical structure $(g, \nabla^{g})$ is an example of a conjugate symmetric statistical structure. 
In Section~\ref{subsec:N0}, we present examples of nontrivial conjugate symmetric statistical structures.

For conjugate symmetric statistical structures, one can define a notion of
``sectional curvature'' as follows.
Let $(g, C)$ be a conjugate symmetric statistical structure on an
$n$-dimensional manifold $M$ with $n \ge 2$.
For linearly independent vectors $v, w \in T_p M$, we define
\begin{equation}\label{eq:sectional-curvature_for-CS}
    \mathrm{Sect}_{g}^C (v, w)
   := \frac{ g_p\big( R(v, w) w,\, v \big) }
           { g_p(v, v)\, g_p(w, w) - g_p(v, w)^2 }.    
\end{equation}
Here, $R$ denotes the curvature tensor of $\nabla^{(g, C)}$.
The value $\mathrm{Sect}_{g}^C (v, w)$ is independent of the choice of basis of the $2$-dimensional plane $\Pi := \operatorname{span}\{v, w\}$ (see~\cite{L-SM}).

\begin{Def}\label{def:CS-sect_basis}
Let $(M,g,C)$ be an $n$-dimensional conjugate symmetric statistical manifold with $n \ge 2$, and let $p \in M$.
Let $\mathcal{B} = \{e_1, \dots, e_n\}$ be an orthonormal basis of $T_p M$ with respect to $g_p$.
For distinct elements $e_i$ and $e_j$ in $\mathcal{B}$, we define
\[
    \Sect^C_g(e_i ,e_j)_{\mathcal{B}}
    := g_p\big( R^g(e_i, e_j)e_j,\, e_i \big)
       + \frac{1}{4}\,[C^{i}, C^{j}]_{ij},
\]
where $R^g$ denotes the curvature tensor of $\nabla^g$.
The matrix $C^i$ is the $n \times n$ matrix whose $(j,k)$-entry is
$C(e_i, e_j, e_k)$, and $[C^{i}, C^{j}]_{ij}$ denotes the $(i,j)$-entry of the commutator $[C^{i}, C^{j}] = C^{i} C^{j} - C^{j} C^{i}$.  
Namely,
\[
    [C^{i}, C^{j}]_{ij}
    = \sum_{k=1}^n C(e_i, e_i, e_k)\, C(e_j, e_j, e_k)
      - C(e_i, e_j, e_k)^2.
\]
\end{Def}

\begin{Prop}\label{prop:curvature-formula}
Under the same setting as in Definition~\ref{def:CS-sect_basis},
the following identity holds:
\[
    \Sect^C_g(e_i ,e_j)_{\mathcal{B}}
    = \Sect^C_g(e_i ,e_j)
    \qquad (i \ne j).
\]
\end{Prop}

\begin{proof}
By \cite[Section 2.2]{Opozda_2019}, the curvature tensor $R$ of $\nabla^{(g, C)}$ can be written as follows.
\[
    R(X,Y)Z = R^{g}(X,Y)Z + [K_X, K_Y](Z) \quad (X,Y,Z \in \Gamma(TM)).
\]
Here, $[K_X, K_Y](Z) := K(X, K(Y,Z)) - K(Y, K(X,Z))$ and $K$ is defined by
\[
    -2g(K(X,Y), Z) = C(X,Y,Z). 
\]
Moreover, we have
\begin{align*}
    g_p([K_{ e_{i} }, K_{ e_{j} } ] e_{j}, e_{i} ) &= g_p (K (e_i,  K ({e_j}, e_j), e_i) - g_p (K({e_j}, K({e_i}, e_j), e_i) \\
    &= -\frac{1}{2} C(e_i, K({e_j}, e_j), e_i) + \frac{1}{2} C(e_j, K({e_i}, e_j), e_i)\\
    &= \frac{1}{4} \cdot \sum_{k = 1}^m  \left\{ C(e_i, C(e_j, e_j, e_k) e_k, e_i ) - C(e_j, C(e_i, e_j, e_k)e_k, e_i) \right\} \\
    &= \frac{1}{4} \cdot \sum_{k = 1}^m  \left\{ C(e_i, e_i, e_k) C(e_j, e_j, e_k) -  C(e_i, e_j, e_k)^2 \right\} \\
    &= \frac{1}{4} \cdot [C^i, C^j]_{ij}.
\end{align*}
This completes the proof.
\end{proof}

\begin{Rem}
Opozda~\cite{Opozda_2015} introduced the notion of sectional $\nabla$-curvature for general statistical structures $(g, \nabla)$. 
The same concept has also been discussed in \cite{Furuhata-Hasegawa_2016} and \cite{IO-2024}, although under different terminology.
When the statistical structure is conjugate symmetric, $\Sect_g^{C}$ coincides with the sectional $\nabla$-curvature (see also~\cite{IO-2024} for details).
\end{Rem}

\subsection{The space of zero-mean multivariate normal distributions}\label{subsec:N0}
Let $\mathcal{N}_0$ denote the space of zero-mean multivariate normal
distributions, namely
\[
    \mathcal{N}_0 = \mathcal{N}_0^n :=
    \left\{
        p(x;\Sigma)
        = \frac{1}{\sqrt{(2\pi)^n \det\Sigma}}
          \exp\!\left( -\frac{1}{2} {}^{t}x \Sigma^{-1} x \right)
        \ \middle|\ 
        \Sigma \in \Sym^+(n,\RR)
    \right\}
    \quad (x \in \RR^n),
\]
where $\Sym^+(n,\RR)$ denotes the space of positive definite symmetric
matrices of size $n$.

On the manifold $\Theta_0 := \Sym^+(n, \RR)$, one can define a Riemannian metric, called the
\emph{Fisher metric on $\No$}, and a symmetric $(0,3)$-tensor field, called the
\emph{Amari--Chentsov $\alpha$-tensor field on $\No$}, as follows:
\begin{align*}
    g^F_{\hat{\theta}}(X, Y)
        &:= \int_{\RR^n}
            \bigl(X \log p(x;\theta)\bigr)
            \bigl(Y \log p(x;\theta)\bigr)
            p(x;\hat{\theta})\, dx, \\
    C^{A(\alpha)}_{\hat{\theta}}(X, Y, Z)
        &:= \alpha
            \int_{\RR^n}
            \bigl(X \log p(x;\theta)\bigr)
            \bigl(Y \log p(x;\theta)\bigr)
            \bigl(Z \log p(x;\theta)\bigr)
            p(x;\hat{\theta})\, dx,
            \qquad \alpha \in \RR.
\end{align*}
Here, $\hat{\theta} \in \Theta_0$ and $X, Y, Z \in T_{\hat{\theta}} \Theta_0$.
Moreover, $dx$ denotes the Lebesgue measure on $\RR^n$.
The affine connection $\nabla^{(g^F, C^{A(\alpha)})}$ is called the
\emph{Amari--Chentsov $\alpha$-connection on $\No$} 
and is denoted
by $\nabla^{A(\alpha)}$.

The space $\mathcal{N}_0$ can be identified with the parameter space
$\Theta_0 = \Sym^+(n,\RR)$ through the correspondence
\[
    \Sym^+(n,\RR) \ni \Sigma \longmapsto p(x;\Sigma) \in \mathcal{N}_0.
\]
Accordingly, throughout this paper, we identify $\mathcal{N}_0$ with the
open submanifold $\Sym^+(n,\RR) = \Theta_0$ of the vector space $\Sym(n,\RR)$ of symmetric matrices.

The Lie group \( GL(n, \mathbb{R}) \) acts smoothly and transitively on \( \mathcal{N}_0 = \Sym^+(n, \RR) \) as follows:
\begin{equation}\label{eq:def-GL-action-N0}
    A \cdot \Sigma = A \Sigma {}^{t}A \qquad (A \in GL(n, \RR), ~ \Sigma \in \Sym^+(n, \RR)).
\end{equation}
The isotropy subgroup of $GL(n, \RR)$ at the identity matrix $I_n\in\mathcal N_0$
is the orthogonal group $O(n)$, that is,
\[
    \left\{A \in GL(n, \RR) ~\middle|~ A \cdot I_n = I_n \right\} = O(n)    
\]
holds.
Thus, $\mathcal{N}_0$ can be identified with the coset manifold $GL(n, \RR)/O(n)$ as $GL(n, \RR)$-homogeneous manifold via the $GL(n, \RR)$-equivariant map $GL(n, \RR)/O(n) \to \No,~ hO(n) \mapsto h {}^{t}h$. 
Moreover, the subgroup
\[
    GL^+(n,\RR) := \{\, A \in GL(n,\RR) \mid \det A > 0 \,\}
\]
also acts transitively on $\mathcal{N}_0$, and hence $\mathcal{N}_0 \simeq GL^+(n,\RR)/SO(n)$ as well.

It is well known that both $g^F$ and $C^{A(\alpha)}$ are $GL(n,\RR)$-invariant on $\mathcal{N}_0$ (see, for example, \cite[Corollary~3.6]{Ay_2015}).
Therefore, in particular, they are also $GL^+(n, \RR)$-invariant.

\begin{Prop}[\cite{Amari_1985, L-SM}]\label{prop:Fisher-Amari_CS-DF}
The statistical structure \((g^F, C^{A(\alpha)})\) is conjugate symmetric for any \(\alpha \in \mathbb{R}\).
Moreover, it is dually flat only when \(\alpha = \pm 1\).
\end{Prop}

Moreover, by \cite[Theorem~2]{Kobayashi-Okuda_2025}, the following holds.

\begin{Prop}\label{prop:GL+-stat_CS}
Every $GL^+(n,\RR)$-invariant statistical structure on $\mathcal{N}_0$
is conjugate symmetric.
\end{Prop}

Note that any $GL(n,\RR)$-invariant statistical structure and any
$\Isom(\No, g^F)$-invariant statistical structure on $\No$ are
$GL^+(n,\RR)$-invariant.

\begin{Prop}[\cite{Skovgaard_1984}]\label{prop:Fisher-on-N0}
Under the identification 
\begin{equation}\label{eq:O(n)-equiv-linear-isom}
    \Sym(n, \RR) \ni X \longmapsto 
    \left.\frac{d}{dt}\right|_{t=0}(I_n + tX) \in T_{I_n}\mathcal{N}_0,
\end{equation}
The Fisher metric $g^F_{I_n}$ can be written as follows:
\[
    g^F_{I_n} (X, Y) = \frac{1}{2} \tr (XY).
\]
\end{Prop}

In \cite{Dolcetti-Pertici_2019}, the isometry group of the \( GL(n, \mathbb{R}) \)-invariant Riemannian metric on \( \Sym^+(n, \mathbb{R}) \) given at the point \( I_n \) by \( g_{I_n} (X, Y) = \tr(XY) \) is investigated.  
From \cite{Dolcetti-Pertici_2019}, we obtain the following.

\begin{Prop}\label{prop:full-isometry_N0}
For \(A \in GL(n,\mathbb{R})\), define a diffeomorphism
\[
\Gamma(A) : \mathcal{N}_0 \longrightarrow \mathcal{N}_0,
\qquad
\Sigma \longmapsto A \Sigma {}^{t}A .
\]
Define diffeomorphisms \(\sigma_1, \sigma_2 : \mathcal{N}_0 \to \mathcal{N}_0\) by
\begin{equation}
\sigma_1(\Sigma) := \Sigma^{-1},
\qquad
\sigma_2(\Sigma) := (\det\Sigma)^{-2/n}\Sigma .
\label{eq:def_sigma_1-2}
\end{equation}
Then \(\sigma_1\) and \(\sigma_2\) commute. 
Let \(\langle \sigma_1,\sigma_2 \rangle\) denote the finite abelian subgroup of
\(\mathrm{Diff}(\mathcal{N}_0)\) generated by \(\sigma_1\) and \(\sigma_2\).
This subgroup normalizes \(\Gamma(GL(n,\mathbb{R}))\). In fact,
\begin{equation}\label{eq:H_normalizer_Gamma}
    \sigma_1 \circ \Gamma(A) \circ \sigma_1
    = \Gamma({}^{t}A^{-1}),
    \quad
    \sigma_2 \circ \Gamma(A) \circ \sigma_2
    = \Gamma( (\det A)^{-2/n}A)
    \quad (A \in GL(n,\mathbb{R})).    
\end{equation}
If \(n \ge 3\), the map
\[
\Gamma(GL(n,\mathbb{R})) \rtimes \langle \sigma_1,\sigma_2\rangle
\longrightarrow
\mathrm{Isom}(\mathcal{N}_0,g^F),
\qquad
(\Gamma(A),h) \longmapsto \Gamma(A)\circ h
\]
is a group isomorphism.
If \(n=1\) or \(2\), the map
\[
\Gamma(GL(n,\mathbb{R})) \rtimes \langle \sigma_1\rangle
\longrightarrow
\mathrm{Isom}(\mathcal{N}_0,g^F),
\qquad
(\Gamma(A),h) \longmapsto \Gamma(A)\circ h
\]
is a group isomorphism.
\end{Prop}

\begin{proof}
It is clear that \(\sigma_1\) and \(\sigma_2\) commute.
Moreover, Equation~\eqref{eq:H_normalizer_Gamma} follows from a direct computation.
Let us prove the two group isomorphisms.
By \cite[Theorem 4.2]{Dolcetti-Pertici_2019}, the following map $\Psi$ is surjective:
\[
   \Psi : \Gamma(GL(n, \RR)) \rtimes \langle \sigma_1, \sigma_2 \rangle \longrightarrow \mathrm{Isom}(\mathcal{N}_0, g^F), \quad (\Gamma(A), h) \longmapsto \Gamma(A) \circ h.
\]
Moreover, one can see that $\Psi$ is a group homomorphism.
If $n \geq 3$, $\Gamma(GL(n, \RR)) \cap H = \{ \mathrm{id} \}$ (see~\cite[4.5 Remark]{Dolcetti-Pertici_2019}).
Therefore, $\Psi$ is an isomorphism.
If $n = 2$, then by \cite[4.5 Remark]{Dolcetti-Pertici_2019},
\[
    \Psi' : \Gamma(GL(n, \RR)) \rtimes \langle \sigma_1 \rangle \longrightarrow \mathrm{Isom}(\mathcal{N}_0^n, g^F), \quad (\Gamma(A), h) \longmapsto \Gamma(A) \circ h
\]
is an isomorphism.
For the case $n = 1$, it follows from Proposition~\ref{prop:Fisher-on-N0} that the Riemannian manifold $(\mathcal{N}_0, g^F)$ coincides with the $1$-dimensional Riemannian manifold $(\RR_{>0}, \tfrac{1}{2x^2} dx^2)$.
As is well known, the full isometry group of the Riemannian manifold $(\RR_{>0}, \tfrac{1}{x^2} dx^2) \simeq (\RR, dx^2)$ is isomorphic to $\RR \rtimes O(1)$.
Therefore, in the case $n = 1$, the map $\Psi'$ is also a group isomorphism.
\end{proof}

From Proposition~\ref{prop:full-isometry_N0}, we obtain the following, which will be applied in Section~\ref{subsec:proof-of-main-thm}.

\begin{Cor}\label{cor:full-isometry_N0_In}
Let $\Isom(\No, g^F)_{I_n}$ be the isotropy subgroup of $\Isom(\No, g^F)$ at $I_n \in \No$, that is,
\[
    \Isom(\No, g^F)_{I_n} := \left\{ f \in \Isom(\No, g^F) ~\middle|~ f(I_n) = I_n \right\}.
\]
For \( n \geq 3 \), the following map is a group isomorphism:
\[
\Gamma(O(n)) \rtimes \langle \sigma_1, \sigma_2 \rangle \longrightarrow 
\mathrm{Isom}(\mathcal{N}_0, g^F)_{I_n}, \quad (\Gamma(A), h) \longmapsto \Gamma(A) \circ h.
\]
For $n = 2$ or $1$, the following map is a group isomorphism:
\[
\Gamma(O(n)) \rtimes \langle \sigma_1 \rangle \longrightarrow 
\mathrm{Isom}(\mathcal{N}_0, g^F)_{I_n}, \quad (\Gamma(A), h) \longmapsto \Gamma(A) \circ h.
\]
\end{Cor}

The following proposition will also be used in Section~\ref{subsec:proof-of-main-thm}.

\begin{Prop}\label{prop:diff_sigma_1-2}
Through the linear isomorphism~\eqref{eq:O(n)-equiv-linear-isom}, the differentials of $\sigma_1$ and $\sigma_2$ at $I_n$, denoted by $(d\sigma_1)_{I_n}$ and $(d\sigma_2)_{I_n}$, can be expressed as follows on $\Sym(n, \RR)$.
\begin{align*}
    (d\sigma_1)_{I_n} &: \Sym(n, \RR) \longrightarrow \Sym(n, \RR), \quad X \longmapsto -X, \\
    (d\sigma_2)_{I_n} &: \Sym(n, \RR) \longrightarrow \Sym(n, \RR), \quad X \longmapsto X - \tfrac{2}{n} \tr(X) I_n. 
\end{align*}
\end{Prop}

\begin{Rem}
It is well known that $(\No, g^F)$ is a Riemannian (globally) symmetric space
in the sense of Kobayashi--Nomizu~\cite[Chapter~XI, Section~6]{Kobayashi-Nomizu_II}.
In particular, the symmetry at $I_n \in \No$ is given by $\sigma_1$.
We next explain the geometric meaning of $\sigma_2$.
Let $\Sym^{+}_{1}(n,\mathbb{R})$ denote the space of positive definite
symmetric matrices with determinant $1$.
Consider the diffeomorphism
\[
    \Sym^{+}(n,\mathbb{R})
    \xrightarrow{\;\sim\;}
    \Sym^{+}_{1}(n,\mathbb{R}) \times \mathbb{R}_{>0},
    \qquad
    \Sigma
    \longmapsto
    \left(
    (\det\Sigma)^{-1/n}\Sigma,\,
    (\det\Sigma)^{1/n}
    \right)
\]
and denote it by $\pi \times p$.
Let $\tau_{\mathbb{R}_{>0}} : \mathbb{R}_{>0} \to \mathbb{R}_{>0}$ be the
involutive diffeomorphism defined by $\tau_{\mathbb{R}_{>0}}(r)=r^{-1}$.
Then $\sigma_{2}$ can be written as
\[
\sigma_{2}
=
(\pi \times p)^{-1}
\circ
(\mathrm{id} \times \tau_{\mathbb{R}_{>0}})
\circ
(\pi \times p).
\]
In other words, $\sigma_{2}$ is the isometry induced by the decomposition of the Riemannian symmetric space
\[
    GL^{+}(n,\mathbb{R})/SO(n) \simeq SL(n,\mathbb{R})/SO(n) \times \mathbb{R}_{>0}
\]
(see, for example,~\cite{Dolcetti-Pertici_2019}).
\end{Rem}

At the end of this subsection, we present the formula for the curvature tensor of the Fisher metric $g^F$ at $I_n$. 
This fact will also be used in Section~\ref{subsec:proof-of-main-thm}.

\begin{Lem}[\cite{Skovgaard_1984}]\label{prop:Fisher-curvature}
Under the identification 
\[
    \Sym(n, \RR) \ni X \longmapsto 
    \left.\frac{d}{dt}\right|_{t=0}(I_n + tX) \in T_{I_n}\mathcal{N}_0,
\]
the curvature tensor $R^F$ of $g^F$ at $I_n$ can be computed as follows:
\[
    g^F_{I_n}(R^F_{I_n}(X,Y)Z ,W) = \frac{1}{4} \left( \tr(YXZW) - \tr(XYZW) \right).
\]
\end{Lem}

\begin{Cor}\label{cor:Fisher-flatness}
The Riemannian manifold $(\mathcal{N}_0^n, g^F)$ is flat only in the case $n = 1$.
\end{Cor}

\subsection{The space of multivariate normal distributions}\label{subsec:N-review}
In this subsection, we review preceding works related to Problem~\ref{prob:Stat_determine} (cf.~\cite{FIK, KO}), which have served as one of the motivations for the present study.
For more detailed discussions, refer to the original works~\cite{FIK, KO}.

Let $\mathcal{N}^n = \{ N(x; \Sigma, \mu) \}$ denote the family of $n$-dimensional normal distributions.
$\mathcal{N}^n$ can be identified with the parameter space $\Sym^+(n,\mathbb{R}) \times \mathbb{R}^n$, where $\Sym^+(n,\mathbb{R})$ denotes the space of $n \times n$ positive definite symmetric matrices. 
Let $g^F$ denote the Fisher metric on $\mathcal{N}^n$.
Then the pair $(\mathcal{N}^n, g^F)$ forms a Riemannian manifold.

Let $UT^+(n,\mathbb{R})$ be the group of upper triangular matrices with positive diagonal entries, and define the Lie group
\[
    \mathrm{Aff}^s(n,\mathbb{R}) := UT^+(n,\mathbb{R}) \ltimes \mathbb{R}^n.
\]
This group acts on $\mathcal{N}^n$ by
\begin{equation}\label{eq:group-action_N_e}
(A,b) \cdot (\Sigma,\mu) := (A \Sigma {}^tA,\, A\mu + b),
\end{equation}
for $(A,b) \in \mathrm{Aff}^s(n,\mathbb{R})$ and $(\Sigma,\mu) \in \mathcal{N}^n$. 
This smooth action is isometric with respect to the Fisher metric $g^F$ and is simply transitive on $\mathcal{N}^n$.

If the property~$\mathrm{P}$ is conjugate symmetry~$\mathrm{CS}$, then the space $\Stat^{\mathrm{Aff}^s(n,\mathbb{R})}_{\mathrm{CS}}(\mathcal{N}^n,g^F)$ of
$\mathrm{Aff}^s(n,\mathbb{R})$-invariant statistical connections on
$(\mathcal{N}^n, g^F)$ satisfying~$\mathrm{CS}$ is given by
\[
    \Stat^{\mathrm{Aff}^s(n,\mathbb{R})}_{\mathrm{CS}}(\mathcal{N}^n,g^F)
    =
    \{\nabla^{A(\alpha)}\mid \alpha\in\mathbb R\}.
\]
Here,
$\nabla^{A(\alpha)}$ denotes the Amari--Chentsov $\alpha$-connection on $\mathcal N^n$.
The case \(n=1\) was treated by Furuhata--Inoguchi--Kobayashi~\cite{FIK},
while the case \(n\ge2\) was established by Kobayashi--Ohno~\cite{KO}.
Furthermore, these results imply that, if $\mathrm{P}$ is taken to be the dually flat property~$\mathrm{DF}$, then for any $n \geq 1$,
\[
    \Stat^{\mathrm{Aff}^s(n,\mathbb{R})}_{\mathrm{DF}}(\mathcal{N}^n,g^F)
    =
    \{\nabla^{A(+1)},\,\nabla^{A(-1)}\}.
\]
Indeed, the $\alpha$-connection $\nabla^{A(\alpha)}$ is flat if and only if
$\alpha=\pm1$ (see, for example,~\cite{Burbea_1986}).

\begin{Rem}
The $n$-dimensional affine group $\mathrm{Aff}(n,\RR) := GL(n,\RR) \ltimes \RR^n$ also acts on the space $\mathcal{N}^n$ in the form~\eqref{eq:group-action_N_e}. 
Then, the Fisher metric $g^F$ and the Amari--Chentsov $\alpha$-connection $\nabla^{A(\alpha)}$ are invariant under the $\mathrm{Aff}(n,\RR)$-action.
This follows from the fact that this action is induced by transformations of the sample space $\RR^n$, together with the invariance of the Amari--Chentsov structure (in the sense of~\cite{Ay_2015}) under such transformations (see~\cite[Corollary~3.6]{Ay_2015}). 
Consequently, the following equalities also hold:
\begin{align*}
    \Stat^{\mathrm{Aff}(n,\mathbb{R})}_{\mathrm{CS}}(\mathcal{N}^n,g^F) 
    &= \{ \nabla^{A(\alpha)} \mid \alpha \in \mathbb{R} \}, \\ 
    \Stat^{\mathrm{Aff}(n,\mathbb{R})}_{\mathrm{DF}}(\mathcal{N}^n,g^F)
    &=
    \{\nabla^{A(+1)},\,\nabla^{A(-1)}\}.
\end{align*}
\end{Rem}

\subsection{Moduli spaces of left-invariant geometric structures on Lie groups}\label{subsec:moduli-left-inv-geom-str}
In this subsection, let $G$ be a Lie group.
The Lie group $G$ acts simply transitively on the manifold $G$ by left multiplication.
In this subsection, we review several known formulations of the moduli spaces of geometric structures on $G$ that are invariant under this action.

In what follows, we denote by $\mathcal{T}^{(s,k)}(G)$ the space of all $(s,k)$-tensor fields on $G$, 
and by $\mathcal{T}^{(s,k)}(G)^G$ the vector space consisting of all left-invariant ones. 
We denote by $\mathfrak{g}$ the Lie algebra of $G$, and by $T^{(s,k)}(\mathfrak{g})$ the $(s,k)$-tensor space of the vector space $\mathfrak{g}$.
Then $\mathcal{T}^{(s,k)}(G)^G$ and $T^{(s,k)}(\mathfrak{g})$ are linearly isomorphic. 
In particular, $\mathcal{T}^{(s,k)}(G)^G$ is a finite-dimensional vector space and hence carries a unique natural Hausdorff topology. 
The group $\Aut(G)$ of automorphisms of $G$ and the multiplicative group $\RR_{>0}$ act on $\mathcal{T}^{(s,k)}(G)^G$ as follows:
\[
    \phi \cdot T = (\phi^{-1})^\ast T, \quad r \cdot T = rT 
    \quad (\phi \in \Aut(G),~ r \in \RR_{>0},~ T \in \mathcal{T}^{(s,k)}(G)^G).
\]
These actions commute with each other, and hence define an action of the direct product group $\RR_{>0} \times \Aut(G)$.

\begin{Def}
\begin{description}
    \item[(Pseudo-)Riemannian metrics (\cite{KTT, KOTT})] 
    Let $n := \dim G$, and let $p,q \geq 0$ be nonnegative integers such that $p + q = n$. 
    We denote by $\mathfrak{M}_{(p,q)}(G)$ the set of all left-invariant pseudo-Riemannian metrics on $G$ of signature $(p,q)$. 
    Then $\mathfrak{M}_{(p,q)}(G)$ is a subset of the topological vector space $\mathcal{T}^{(0,2)}(G)^G$, 
    and we regard it as a topological space with the relative topology. 
    Moreover, $\mathfrak{M}_{(p,q)}(G)$ is stable under the action of $\RR_{>0} \times \Aut(G)$. 
    The quotient space
    \[
        \mathfrak{PM}_{(p,q)}(G) := (\RR_{>0} \times \Aut(G)) \backslash \mathfrak{M}_{(p,q)}(G)
    \]
    is called the \emph{moduli space of left-invariant pseudo-Riemannian metrics of signature $(p,q)$ on $G$}.
    \item[Statistical structures (\cite{KOOT_2025})] 
    Let $n := \dim G$. 
    We denote by $\LStat(G)$ the set of all left-invariant statistical structures on $G$. 
    We write $\Gamma(S^3 T^\ast G)^G \subset \mathcal{T}^{(0,3)}(G)^G$ for the finite-dimensional vector space consisting of all left-invariant symmetric $(0,3)$-tensor fields on $G$.
    By Theorem~\ref{theorem:nablaC}, $\LStat(G)$ can be identified with the product space $\mathfrak{M}_{(n,0)}(G) \times \Gamma(S^3 T^\ast G)^G$ via the map
    \[
        \LStat(G) \longrightarrow \mathfrak{M}_{(n,0)}(G) \times \Gamma(S^3 T^\ast G)^G, 
        \qquad (g, \nabla) \longmapsto (g, C^{(g, \nabla)}).
    \]
    In particular, under this identification, $\LStat(G)$ can be regarded as a trivial vector bundle over $\mathfrak{M}_{(n,0)}(G)$ with fiber $\Gamma(S^3 T^\ast G)^G$.
    The product space $\mathfrak{M}_{(n,0)}(G) \times \Gamma(S^3 T^\ast G)^G$ is stable under the diagonal action of $\RR_{>0} \times \Aut(G)$. 
    Via the above identification, this induces an action of $\RR_{>0} \times \Aut(G)$ on $\LStat(G)$ given by
    \[
        (r, \phi) \cdot (g, \nabla) 
        = \bigl(r (\phi^{-1})^\ast g,\, (\phi^{-1})^\ast \nabla\bigr).
    \]
    The quotient space
    \[
        \MLStat(G) := (\RR_{>0} \times \Aut(G)) \backslash \LStat(G)
    \]
    is called the \emph{moduli space of left-invariant statistical structures on $G$}.
\end{description}
\end{Def}

\begin{Rem}
As a definition of the moduli space of left-invariant pseudo-Riemannian metrics, we have introduced
\[
    \mathfrak{PM}_{(p,q)}(G) := (\RR_{>0} \times \Aut(G)) \backslash \mathfrak{M}_{(p,q)}(G).
\]
However, in the original papers \cite{KTT, KOTT}, it is defined by
\[
    \mathfrak{PM}_{(p,q)}(G) := (\RR^\times \times \Aut(G)) \backslash \mathfrak{M}_{(p,q)}(G).
\]
Here, the group $\RR^\times$ acts on $\mathcal{T}^{(0,2)}(G)^G$ by
\[
    c \cdot \bigl(T(\cdot, \cdot)\bigr) := T(c^{-1}\cdot, c^{-1}\cdot) = c^{-2} T(\cdot, \cdot).
\]
Since the orbit decomposition of $\mathfrak{M}_{(p,q)}(G)$ under this action coincides with that under the action of $\RR_{>0}$ described above, the quotient spaces are the same in either case.
\end{Rem}

\begin{Rem}
Moduli spaces of left-invariant geometric structures such as symplectic, complex, and K\"{a}hler structures can be defined in a similar manner. 
For instance, in Moscoso--Tamaru~\cite{Luis-T}, the authors introduce the moduli space of non-degenerate left-invariant $2$-forms on a Lie group $G$, and by explicitly determining this space, they obtain a classification of symplectic structures within it. 
Moreover, related approaches to moduli spaces of left-invariant symplectic and complex structures on Lie groups have also been studied in other works (see, for example, [Kato (2025), arXiv:2510.14610] and Salamon~\cite{Salamon_2001}).
\end{Rem}

\section{Category of statistical manifolds and $\STAT$-moduli spaces of statistical connections}\label{sec:Cat_Stat}

\subsection{The category of statistical manifolds and statistical immersions}\label{subsec:Cat_Stat}
We begin by fixing our terminology for statistical immersions between statistical manifolds.

\begin{Def}[cf.~\cite{Furuhata_2009} and {\cite[Section~4.5.1]{Ay_2017}}]\label{Definition:DAI}
Let $(M,g,\nabla)$ and $(M',g',\nabla')$ be statistical manifolds.
A smooth map $f : M \rightarrow M'$ is called a \emph{statistical immersion} if the following conditions hold:
\begin{enumerate}[label=(\arabic*)]
    \item $f$ is isometric, i.e.\ $f^\ast g' = g$ (in particular, $f$ is an immersion).
    \item For any vector fields $X,Y,Z$ on $M$, one has
    \[
        g(\nabla_X Y, Z)
        =
        g'\bigl((f^\ast \nabla')_X f_\ast Y,\, f_\ast Z\bigr).
    \]
    Here $f^\ast \nabla'$ denotes the connection induced by $f$ on the pullback bundle $f^\ast TM' \to M$.
\end{enumerate}
\end{Def}

\begin{Prop}
Let $C := C^\nabla$ and $C' := C^{\nabla'}$ denote the cubic forms associated with
$(M,g,\nabla)$ and $(M',g',\nabla')$, respectively
(see Section~\ref{subsec:stat-mfd} for the definition of the cubic form of a statistical manifold).
For an isometric map $f : M \to M'$, the following conditions are equivalent:
\begin{enumerate}[label=(\roman*)]
    \item $f$ satisfies condition {\rm(2)} in Definition~\ref{Definition:DAI}.
    \item $f^\ast C' = C$.
\end{enumerate}
\end{Prop}

It is straightforward to verify that statistical manifolds and statistical immersions form a category.
In this paper, we denote this category by $\mathrm{Stat}_{\mathrm{StatImm}}$, or simply by $\STAT$.
A statistical immersion $f$ between $(M,g,\nabla)$ and $(M',g',\nabla')$ is an isomorphism in $\STAT$
if and only if $f : M \to M'$ is a diffeomorphism.

Recall that for a category $\mathcal{D}$, a full subcategory $\mathcal{C} \subset \mathcal{D}$
is called \emph{replete} if whenever two objects $M$ and $M'$ of $\mathcal{D}$ are isomorphic in $\mathcal{D}$
and $M$ belongs to $\mathcal{C}$, then $M'$ also belongs to $\mathcal{C}$.

\begin{Def}\label{Definition:class}
In this paper, any full replete subcategory $\mathcal{C}$ of $\STAT$ will be called
a \emph{class of statistical manifolds}.
An object $(M,g,\nabla)$ of $\mathcal{C}$ will be referred to as
a statistical manifold of class $\mathcal{C}$.
\end{Def}

For instance, the category consisting of dually flat statistical manifolds and statistical immersions
forms a class of statistical manifolds.

\subsection{Definition of $\STAT$-moduli spaces of statistical connections}\label{subsection:DefStatModuli}

Let $(M,g)$ be a Riemannian manifold.
We denote by $\Stat(M,g)$ the set of all statistical connections on $(M,g)$.
For each $\nabla \in \Stat(M,g)$, the triple $(M,g,\nabla)$ defines
a statistical manifold, which is therefore an object of the category $\STAT$.

We introduce an equivalence relation $\sim_{\STAT}$ on $\Stat(M,g)$ as follows.
For $\nabla,\nabla' \in \Stat(M,g)$, we write
\[
\nabla \sim_{\STAT} \nabla'
\]
if there exists a statistical isomorphism (an isomorphism in the category $\STAT$)
\[
f : (M,g,\nabla) \longrightarrow (M,g,\nabla').
\]

This defines an equivalence relation on $\Stat(M,g)$.
The quotient set
\[
\MStatStat(M,g) := \Stat(M,g) / {\sim_{\STAT}}
\]
will be called the \emph{$\STAT$-moduli space} of statistical connections on $(M,g)$.
For $\nabla \in \Stat(M,g)$, we denote by $[\nabla]_{\STAT}$ the equivalence class of $\nabla$.

This equivalence relation may also be interpreted as the orbit decomposition
associated with the action of the isometry group of $(M,g)$.
Indeed, if we denote the isometry group by $\Isom(M,g)$,
then $\Isom(M,g)$ acts on $\Stat(M,g)$ by 
\[
    f \cdot \nabla := (f^{-1})^* \nabla
\]
for $f \in \Isom(M,g)$ and $\nabla \in \Stat(M,g)$ (see Definition \ref{Definition:DAI} for the notation $(f^{-1})^* \nabla$).
It is immediate that the orbit space of this action coincides with
$\MStatStat(M,g)$.
In particular, we may write
\[
\MStatStat(M,g)
=
\Isom(M,g) \backslash \Stat(M,g).
\]

More generally, for a subset $\mathcal{S} \subset \Stat(M,g)$,
we define a subset of $\MStatStat(M,g)$ by
\[
\mathcal{M}_{\STAT}\mathcal{S}
:=
\{\, [\nabla]_{\STAT} \mid \nabla \in \mathcal{S} \,\}
\subset
\MStatStat(M,g).
\]

If $\mathcal{S}$ is stable under the action of $\Isom(M,g)$,
then $\mathcal{M}_{\STAT}\mathcal{S}$ coincides with the orbit space
$\Isom(M,g) \backslash \mathcal{S}$.

For example, let $\mathcal{C}$ be a class of statistical manifolds
(see Definition~\ref{Definition:class}), and define
\[
\Stat_{\mathcal{C}}(M,g)
:=
\{
\nabla \in \Stat(M,g)
\mid
(M,g,\nabla) \in \mathcal{C}
\}
\subset
\Stat(M,g).
\]
Then $\Stat_{\mathcal{C}}(M,g)$ is stable under the action of $\Isom(M,g)$,
and hence
\[
\MStatStat_{\mathcal{C}}(M,g)
=
\Isom(M,g) \backslash \Stat_{\mathcal{C}}(M,g).
\]
In later sections, we will also consider the situation where
a Lie group $G$ acts isometrically on $(M,g)$. 
In this case, we denote by $\Stat^G(M,g)$ the subset of $\Stat(M,g)$
consisting of $G$-invariant statistical connections. 
Moreover, for a class $\mathcal{C}$ of statistical manifolds, we denote by $\Stat^G_{\mathcal{C}}(M,g)$ the subset of $\Stat^G(M,g)$ consisting of those belonging to $\mathcal{C}$.
In general, $\Stat^G(M,g)$ and $\Stat^G_{\mathcal{C}}(M,g)$ are not stable under the action of $\Isom(M,g)$. 
Consequently, $\MStatStat^G(M,g)$ and $\MStatStat^G_{\mathcal{C}}(M,g)$ are merely quotient sets of $\Stat^G(M,g)$ 
and should not be regarded as orbit spaces arising from a group action.

\section{Equivariant morphisms between homogeneous manifolds}\label{sec:mor-HS}

\subsection{Equivariant morphisms}\label{section:EM}
Throughout this paper, a triple $(M,G,\sigma)$ is called a \emph{homogeneous manifold}
if $M$ is a smooth manifold, $G$ is a Lie group, and
$\sigma : G \times M \to M$ is a smooth transitive action of $G$ on $M$.
For $h \in G$ and $x \in M$, we write
$h \cdot_\sigma x = \sigma_h(x) = \sigma(h,x)$.

Let $(M,G,\sigma)$ and $(M',G',\sigma')$ be homogeneous manifolds.
A pair $(\phi,f)$ is called an \emph{equivariant morphism}
from $(M,G,\sigma)$ to $(M',G',\sigma')$
if $\phi : G \to G'$ is a Lie group homomorphism and
$f : M \to M'$ is a smooth map satisfying
\[
f(h \cdot_\sigma x) = \phi(h) \cdot_{\sigma'} f(x)
\]
for all $h \in G$ and $x \in M$.

\begin{Rem}
A typical example arises when the acting group is fixed.
Let $G$ be a Lie group, and let $M$ and $M'$ be manifolds equipped with
smooth transitive $G$-actions $\sigma$ and $\sigma'$, respectively.
A smooth map $f : M \to M'$ is said to be \emph{$G$-equivariant} if
\[
f(h \cdot_\sigma x) = h \cdot_{\sigma'} f(x)
\]
for all $h \in G$ and $x \in M$.
In this case, the pair $(\mathrm{id}_G,f)$ defines an equivariant morphism
from $(M,G,\sigma)$ to $(M',G,\sigma')$.
In general, however, an equivariant morphism between
$(M,G,\sigma)$ and $(M',G,\sigma')$ need not be of this form.
\end{Rem}

Homogeneous manifolds together with equivariant morphisms form a category,
denoted by $\HMFD$,
with composition given by
\[
(\phi_2,f_2) \circ (\phi_1,f_1)
:= (\phi_2 \circ \phi_1,\, f_2 \circ f_1).
\]
The identity morphism on $(M,G,\sigma)$ is $(\mathrm{id}_G,\mathrm{id}_M)$.
An equivariant morphism $(\phi,f)$ is an isomorphism in $\HMFD$
if and only if $\phi : G \to G'$ is a Lie group isomorphism
and $f : M \to M'$ is a diffeomorphism.

Throughout this paper, a quadruple $(M,g,G,\sigma)$ is called a
\emph{homogeneous Riemannian manifold}
if $(M,G,\sigma)$ is a homogeneous manifold and
$g$ is a Riemannian metric on $M$ invariant under the $G$-action $\sigma$.

\begin{Ex}\label{ex:N0_homog-Riem}
The quadruple
$(\mathcal{N}_0, g^F, GL(n,\mathbb{R}), \Gamma)$
is a homogeneous Riemannian manifold,
where the action $\Gamma$ is given by
\[
    \Gamma(A,\Sigma) := A \Sigma {}^{t}A .
\]
(See Section~\ref{subsec:N0} for details.)
\end{Ex}

Let $(M,g,G,\sigma)$ and $(M',g',G',\sigma')$ be homogeneous Riemannian manifolds.
An equivariant morphism $(\phi,f)$
from $(M,G,\sigma)$ to $(M',G',\sigma')$
is called an \emph{equivariant isometric morphism} between $(M,g,G,\sigma)$ and $(M',g',G',\sigma')$
if the smooth map $f : M \to M'$ is an isometry with respect to $g$ and $g'$, that is,
$f^* g' = g$.
We note that for such $f$, since $g$ is non-degenerate, $f$ should be an immersion from $M$ to $M'$.

Homogeneous Riemannian manifolds together with equivariant isometric morphisms
form a category, denoted by $\HRiem$,
with composition and identity morphisms defined in the same manner as in $\HMFD$.

\subsection{Automorphism groups in the categories $\HMFD$ and $\HRiem$}\label{subsec_Autgrp}
Let $(M,g,G,\sigma)$ be a homogeneous Riemannian manifold. 
The automorphism group of $(M,g,G,\sigma)$ in the category $\HRiem$, which is denoted by $\Aut_{\HRiem}(M,g,G,\sigma)$, will play a key role in this paper.
In this subsection, we study the group $\Aut_{\HRiem}(M,g,G,\sigma)$.

First, let us fix a homogeneous manifold $(M,G,\sigma)$ and study the automorphism group $\Aut_{\HMFD}(M,G,\sigma)$ in the category $\HMFD$ as follows:
For each $l \in G$, we write 
\[
\mathrm{Inn}_l : G \longrightarrow G, \qquad h \longmapsto l h l^{-1}.
\]
for the conjugation on $G$ by $l$. 
Then the pair $\theta_l := (\mathrm{Inn}_l,\sigma_l)$ is an automorphism on $(M,G,\sigma)$ in the category $\HMFD$. 
Furthermore, 
\[
\theta : G \longrightarrow \Aut_{\HMFD}(M,G,\sigma), \qquad l \longmapsto \theta_l
\]
is a group homomorphism.
The subgroup $\theta(G)$ of $\Aut_{\HMFD}(M,G,\sigma)$ is normal, in fact, for each $l \in G$ and $(\phi,f) \in \Aut_{\HMFD}(M,G,\sigma)$, 
one can see that 
\begin{align*}
    (\phi,f) \circ \theta_l \circ (\phi,f)^{-1} = (\phi \circ \mathrm{Inn}_l \circ \phi^{-1},f \circ \sigma_l \circ f^{-1}) = (\mathrm{Inn}_{\phi(l)}, \sigma_{\phi(l)}) = \theta_{\phi(l)}.
\end{align*}

We shall fix a point $p_0 \in M$.
Let us define the subgroup $\Aut_{\HMFD}(M,G,\sigma)_{p_0}$ of $\Aut_{\HMFD}(M,G,\sigma)$ by 
\[
\Aut_{\HMFD}(M,G,\sigma)_{p_0} := \{ (\phi,f) \in \Aut_{\HMFD}(M,G,\sigma) \mid f(p_0) = p_0 \}.
\]
The semidirect product group of $\Aut_{\HMFD}(M,G,\sigma)_{p_0}$ and $\theta(G)$ will be denoted by $\Aut_{\HMFD}(M,G,\sigma)_{p_0} \ltimes \theta(G)$, 
that is, we define the group structure on the direct product set $\Aut_{\HMFD}(M,G,\sigma)_{p_0} \times \theta(G)$ by putting 
\[
(\alpha,\beta) \cdot (\alpha',\beta') = (\alpha \circ \alpha', \beta \circ (\alpha \circ  \beta' \circ \alpha^{-1}))
\]
for $(\alpha,\beta), (\alpha',\beta') \in \Aut_{\HMFD}(M,G,\sigma)_{p_0} \times \theta(G)$.

The isotropy subgroup of $G$ at $p_0$ will be denoted by $K := \{ k \in G \mid k \cdot_\sigma p_0 = p_0 \}$.
Then the following holds:

\begin{Prop}\label{Proposition:AutHMFD_p0}
The group homomorphism 
\[
\Psi : \Aut_{\HMFD}(M,G,\sigma)_{p_0} \ltimes \theta(G) \longrightarrow \Aut_{\HMFD}(M,G,\sigma), \qquad (\alpha,\beta) \longmapsto \beta \circ \alpha
\]
is surjective, and the kernel is given by 
\[
\ker \Psi = \{ (\theta_{k^{-1}},\theta_k) \mid k \in K \}.
\]
\end{Prop}

\begin{proof}
First, we shall prove that $\Psi$ is surjective.
Take any $(\phi,f) \in \Aut_{\HMFD}(M,G,\sigma)$. 
Our goal is to find $l \in G$ such that $(\Inn_{l},\sigma_l)^{-1} \circ (\phi,f) \in \Aut_{\HMFD}(M,G,\sigma)_{p_0}$.
Put $q_0 := f(p_0)$.
Since the $G$-action on $M$ is transitive, 
one can find an element $l \in G$ such that $\sigma_l(p_0) = q_0$.
For such an $l \in G$, 
we have $\sigma_{l^{-1}}(f(p_0)) = \sigma_{l^{-1}}(q_0) = p_0$, and hence $(\Inn_{l},\sigma_l)^{-1} \circ (\phi,f) \in \Aut_{\HMFD}(M,G,\sigma)_{p_0}$.

Next, let us prove that $\ker \Psi = \{ (\Inn_{k^{-1}},\sigma_{k^{-1}}),(\Inn_k,\sigma_k)) \mid k \in K \}$.
The inclusion $\ker \Psi \supset \{ (\theta_{k^{-1}}, \theta_k) \mid k \in K \}$ is trivial.
Take any $(\alpha,\beta) \in  \ker \Psi$.
Then $\alpha = \beta^{-1}$.
We write $\beta = \theta_l$ for some $l \in G$.
Then $\alpha = \beta^{-1} = \theta_{l^{-1}} = (\Inn_{l^{-1}},\sigma_{l^{-1}})$.
Since $\alpha \in \Aut_{\HMFD}(M,G,\sigma)_{p_0}$, we have $l^{-1} \in K$ and $l \in K$.
This proves the converse inclusion $\ker \Psi \subset \{ (\theta_{k^{-1}}, \theta_k) \mid k \in K \}$
\end{proof}

Let us also study $\Aut_{\HMFD}(M,G,\sigma)_{p_0}$.
We define the subgroup $\Aut(G,K)$ of the automorphism group $\Aut(G)$ of the Lie group $G$ by 
\[
\Aut(G,K) := \{ \phi \in \Aut(G) \mid \phi(K) = K \}.
\]
For each $\phi \in \Aut(G,K)$, we shall define 
\begin{equation}\label{eq:def-f_phi}
    f_{\phi} : M \longrightarrow M, 
    \qquad
    x = h_x \cdot_{\sigma} p_0
    \longmapsto
    \phi(h_x)\cdot_{\sigma} p_0.
\end{equation}
One can see that $f_{\phi}$ is well-defined, smooth and $(\phi,f_\phi) \in \Aut_{\HMFD}(M,G,\sigma)_{p_0}$.

\begin{Thm} The correspondence $(\phi,f) \mapsto \phi$ gives a group isomorphism between $\Aut_{\HMFD}(M,G,\sigma)_{p_0}$ and $\Aut(G,K)$, 
and the inverse map is given as 
\[
\Aut(G,K) \longrightarrow \Aut_{\HMFD}(M,G,\sigma)_{p_0}, \qquad \phi \longmapsto (\phi,f_{\phi}).
\]
\end{Thm}

\begin{proof}
First, let us prove that $\phi \in \Aut(G,K)$ for each $(\phi,f) \in \Aut_{\HMFD}(M,G,\sigma)_{p_0}$.
Take any $k \in K$.
Then $p_0 = f(p_0) = f(k \cdot_{\sigma} p_0) = \phi(k) \cdot_{\sigma} f(p_0) = \phi(k) \cdot_{\sigma} p_0$, and hence $\phi(k) \in K$.
This proves $\phi(K) \subset K$.
Since $(\phi^{-1},f^{-1}) \in \Aut_{\HMFD}(M,G,\sigma)_{p_0}$, we also have $\phi^{-1}(K) \subset K$, 
and thus $\phi(K) = K$.
We obtain a group homomorphism 
\[
\Aut_{\HMFD}(M,G,\sigma)_{p_0} \longrightarrow \Aut(G,K), \qquad (\phi,f) \longmapsto \phi.
\]
It remains to show that $f = f_{\phi}$ for each $(\phi,f) \in \Aut_{\HMFD}(M,G,\sigma)_{p_0}$.
Take any $x = h \cdot_{\sigma} p_0$.
Then 
\[
f(x) = f(h \cdot_{\sigma} p_0) = \phi(h) f(p_0) = \phi(h) \cdot_{\sigma} p_0 = f_{\phi}(x).
\]
This completes the proof.
\end{proof}

Under the group isomorphism $\Aut_{\HMFD}(M,G,\sigma)_{p_0} \simeq \Aut(G,K)$, 
we obtain the semidirect product group $\Aut(G,K) \ltimes \theta(G)$ and the surjective group homomorphism 
\[
\Aut(G,K) \ltimes \theta(G) \longrightarrow \Aut_{\HMFD}(M,G,\sigma), \qquad (\phi,\theta_{l}) \longmapsto (\Inn_{l} \circ \phi, \sigma_l \circ f_\phi), 
\]
with its kernel is given by 
\[
\{ (\Inn_{k^{-1}},\theta_k) \mid k \in K \}, 
\]
where the group structure on $\Aut(G,K) \ltimes \theta(G)$ can be written as 
\[
(\phi,\theta_{l}) \cdot (\phi',\theta_{l'}) = (\phi \circ \phi', \theta_l \circ \theta_{\phi(l')})
\]
for $(\phi,\theta_l), (\phi',\theta_{l'}) \in \Aut(G,K) \ltimes \theta(G)$.

We now turn to the study of automorphism groups in the category $\HRiem$.
Before proceeding further, we record the following auxiliary result,
which will be used later.
\begin{Prop}\label{Proposition:Aut_preserves_invtensor}
Let $T$ be a $G$-invariant tensor field on the smooth manifold $M$.
Then, for each $(\phi,f) \in \Aut_{\HMFD}(M,G,\sigma)$,
the pullback tensor $f^* T$ is also $G$-invariant.
\end{Prop}

\begin{proof}
Let us write $\mathcal{T}(M)$ for the set of all tensor fields on $M$.
Then the correspondence $((\phi,f),T) \mapsto (f^{-1})^\ast T$ defines an $\Aut_{\HMFD}(M,G,\sigma)$-action on $\mathcal{T}(M)$.
Note that a tensor field $T \in \mathcal{T}(M)$ is $G$-invariant if and only if it is invariant by $\theta(G) \subset \Aut_{\HMFD}(M, G, \sigma)$.
Since $\theta(G)$ is a normal subgroup of $\Aut_{\HMFD}(M,G,\sigma)$, 
for each $\theta(G)$-invariant tensor field $T$ and each $(\phi,f) \in \Aut_{\HMFD}(M,G,\sigma)$, 
the tensor field $f^* T$ should be $\theta(G)$-invariant.
\end{proof}

Let us fix a $G$-invariant Riemannian metric $g$ on $M$ and consider the homogeneous Riemannian manifold $(M,g,G,\sigma)$.
We note that 
\[
\Aut_{\HRiem}(M,g,G,\sigma) = \{ (\phi,f) \in \Aut_{\HMFD}(M,G,\sigma) \mid f^*g = g \}.
\]

Since the metric $g$ is $G$-invariant,
it follows that $\theta_l \in \Aut_{\HRiem}(M,g,G,\sigma)$
for each $l \in G$.
Then by putting 
\[
\Aut_{\HRiem}(M,g,G,\sigma)_{p_0} := \{ (\phi,f) \in \Aut_{\HRiem}(M,g,G,\sigma) \mid f(p_0) = p_0 \}, 
\]
we obtain the semidirect product group $\Aut_{\HRiem}(M,g,G,\sigma)_{p_0} \ltimes \theta(G)$.
Similarly to Proposition \ref{Proposition:AutHMFD_p0}, we also obtain the following:

\begin{Prop}
The group homomorphism
\[
\Aut_{\HRiem}(M,g,G,\sigma)_{p_0} \ltimes \theta(G) \longrightarrow \Aut_{\HRiem}(M,g,G,\sigma), \qquad (\alpha,\beta) \longmapsto \beta \circ \alpha
\]
is surjective and the kernel is given by 
\[
\{ (\theta_{k^{-1}}, \theta_k) \mid k \in K \}.
\]
\end{Prop}

Recall that $\Aut_{\HMFD}(M,G,\sigma)_{p_0}$ is isomorphic to $\Aut(G,K)$.
Let us study the subgroup of $\Aut(G,K)$ corresponding to the subgroup $\Aut_{\HRiem}(M,g,G,\sigma)_{p_0}$ of $\Aut_{\HMFD}(M,G,\sigma)_{p_0}$.

We write $\mathfrak{g}$ and $\mathfrak{k}$ for the Lie algebras of $G$ and $K$.
Then the quotient vector space $\mathfrak{g}/{\mathfrak{k}}$ can be identified with the tangent space $T_{p_0}M$ via the map 
\[
\Upsilon : \mathfrak{g}/\mathfrak{k} \longrightarrow T_{p_0}M,
\qquad
[X]_{\mathfrak{k}} \longmapsto
\left.\frac{d}{dt}\right|_{t=0}
\bigl(\exp(-tX)\cdot_{\sigma} p_0\bigr), 
\]
where we write $[X]_{\mathfrak{k}} := X + \mathfrak{k} \in \mathfrak{g}/{\mathfrak{k}}$ for each $X \in \mathfrak{g}$.
In particular, the metric tensor $g$ on $M$ induces an inner product $g_0$ on $\mathfrak{g}/{\mathfrak{k}} \simeq T_{p_0} M$.
Let us define a linear representation of the group $\Aut(G,K)$ on $\mathfrak{g}/{\mathfrak{k}}$ by putting 
\begin{equation}\label{eq:def-phi_gk}
    \phi_{\mathfrak{g}/\mathfrak{k}}
    :
    \mathfrak{g}/\mathfrak{k}
    \longrightarrow
    \mathfrak{g}/\mathfrak{k},
    \qquad
    [X]_{\mathfrak{k}}
    \longmapsto
    [\phi_\ast(X)]_{\mathfrak{k}}.
\end{equation}
for each $\phi \in \Aut(G,K)$, 
where $\phi_* : \mathfrak{g} \rightarrow \mathfrak{g}$ denotes the automorphism of the Lie algebra $\mathfrak{g}$ induced by $\phi$. 

\begin{Prop}\label{proposition:fphistarphigk}
Under the identifications $\Aut(G,K) \simeq \Aut_{\HMFD}(M,G,\sigma)_{p_0}$ and $\mathfrak{g}/{\mathfrak{k}} \simeq T_{p_0}M$, 
the linear isomorphism 
$
\phi_{\mathfrak{g}/{\mathfrak{k}}} : \mathfrak{g}/{\mathfrak{k}} \rightarrow \mathfrak{g}/{\mathfrak{k}}
$
corresponds to the derivation 
\[
(f_{\phi})_* : T_{p_0} M \longrightarrow T_{p_0} M
\]
of $f_{\phi}$ at $p_0$.    
\end{Prop}

\begin{proof}
Let us fix $\phi \in \Aut(G,K)$.
Our goal is to show the commutativity 
$(f_{\phi})_* \circ \Upsilon = \Upsilon \circ \phi_{\mathfrak{g}/{\mathfrak{k}}}$.
Fix any $[X]_{\mathfrak{k}} \in \mathfrak{g}/{\mathfrak{k}}$.
Then 
\begin{align*}
    ((f_{\phi})_* \circ \Upsilon)([X]_{\mathfrak{k}}) 
        &= \left.\frac{d}{dt} \right|_{t=0}f_{\phi}(\exp(-t X) \cdot_\sigma p_0) \\
        &= \left.\frac{d}{dt} \right|_{t=0} \phi(\exp(-tX)) \cdot_{\sigma} p_0 \\
        &= \left.\frac{d}{dt} \right|_{t=0} (\exp(-t \phi_*(X))) \cdot_{\sigma} p_0 \\
        &= (\Upsilon \circ \phi_{\mathfrak{g}/{\mathfrak{k}}})([X]_{\mathfrak{k}}).
\end{align*}
\end{proof}

We define the subgroup of $\Aut(G,K,g_0)$ of $\Aut(G,K)$ by 
\[
\Aut(G,K,g_0) := \{ \phi \in \Aut(G,K) \mid \phi_{\mathfrak{g}/\mathfrak{k}} \text{ preserves } g_0 \}.
\]

\begin{Prop}\label{Proposition:AutGKg0AutHRiemp0}
The following statements hold:
\begin{enumerate}[label=(\arabic*)]
    \item For each $\phi \in \Aut(G,K)$, the two conditions below are equivalent:
    \begin{enumerate}[label=(\alph*)]
        \item \label{item:AutGKg0AutHRiemp0:f*g} $f_\phi^* g = g$. 
        \item \label{item:AutGKg0AutHRiemp0:phigk} $\phi_{\mathfrak{g}/\mathfrak{k}}$ preserves the inner product $g_0$.
    \end{enumerate}
    \item Under the identification $\Aut_{\HMFD}(M,G,\sigma)_{p_0} \simeq \Aut(G,K)$, 
    the subgroup $\Aut_{\HRiem}(M,g,G,\sigma)_{p_0}$ corresponds to $\Aut(G,K,g_0)$.
\end{enumerate}
\end{Prop}

\begin{proof}
The second claim follows immediately from the first claim.
Therefore, we only need to show the first claim.
Let us fix $\phi \in \Aut(G,K)$.
By Proposition \ref{Proposition:Aut_preserves_invtensor}, 
$f_{\phi}^* g$ is also $G$-invariant.
Thus the condition \ref{item:AutGKg0AutHRiemp0:f*g} is equivalent to the condition $(f_{\phi}^* g)_{p_0} = g_{p_0}$ on $T_{p_0}M$.
Since via the identification $T_{p_0}M \simeq \mathfrak{g}/{\mathfrak{k}}$, 
the derivation $(f_{\phi})_* : T_{p_0}M \rightarrow T_{p_0}M$ corresponds to $\phi_{\mathfrak{g}/{\mathfrak{k}}} : \mathfrak{g}/{\mathfrak{k}} \rightarrow \mathfrak{g}/{\mathfrak{k}}$ (by Proposition \ref{proposition:fphistarphigk}), 
thus the condition $(f_{\phi}^* g)_{p_0} = g_{p_0}$ is equivalent to the condition \ref{item:AutGKg0AutHRiemp0:phigk}.
\end{proof}

Via the group isomorphism
$\Aut_{\HRiem}(M,g,G,\sigma)_{p_0} \simeq \Aut(G,K,g_0)$,
the automorphism group $\Aut_{\HRiem}(M,g,G,\sigma)$
admits a description in terms of the semidirect product
$\Aut(G,K,g_0) \ltimes \theta(G)$.
More precisely, we obtain the following theorem:

\begin{Thm}\label{theorem:AutGKthetaontoAutHRiem}
The map 
\[
\Aut(G,K,g_0) \ltimes \theta(G)
\longrightarrow \Aut_{\HRiem}(M,g,G,\sigma),
\qquad
(\phi,\theta_l) \longmapsto (\Inn_l \circ \phi, \sigma_l \circ f_\phi)
\]
gives a surjective group homomorphism 
whose kernel is given by
\[
\{ (\Inn_{k^{-1}}, \theta_k) \mid k \in K \}.
\]    
\end{Thm}

\subsection{Effective cases}
In this subsection, 
we give a characterization of the group $\Aut_{\HRiem}(M,g,G,\sigma)$ as a normalizer of $G$ in $\Isom(M,g)$,
for the cases where the $G$-action on $M$ is effective and $M$ has only finitely many connected components.

\begin{Prop}\label{Proposition:EffectiveNormalizer}
Let $(M,G,\sigma)$ be a homogeneous manifold.
We shall consider the group homomorphism 
\[
\Xi : \Aut_{\HMFD}(M,G,\sigma) \longrightarrow \Diff(M), \qquad (\phi,f) \longmapsto f.
\]
\begin{enumerate}[label=(\arabic*)]
    \item \label{item:prop:normalize:HMFD} 
The subgroup $\Xi(\Aut_{\HMFD}(M,G,\sigma))$ normalize the subgroup $\sigma(G)$ in $\Diff(M)$.
    \item \label{item:prop:effective:HMFD} 
Assume that the $G$-action $\sigma$ on $M$ is effective.
Then the group homomorphism $\Xi$ is injective.
Furthermore, $\Aut_{\HMFD}(M,G,\sigma)$ normalizes $G$ if we consider $\Aut_{\HMFD}(M,G,\sigma)$ and $G$ as subgroups of $\Diff(M)$.
\item \label{item:prop:effective:HRiem} Let $(M,g,G,\sigma)$ be a homogeneous Riemannian manifold such that the $G$-action $\sigma$ on $M$ is effective and $M$ has only finitely many connected components.
    We consider $\Aut_{\HRiem}(M,g,G,\sigma)$ as a subgroup of $\Isom(M,g)$ via the injective group homomorphism $\Xi$ defined above.
    Then the subgroup $\Aut_{\HRiem}(M,g,G,\sigma)$ of $\Isom(M,g)$ coincides with the normalizer of the subgroup $G$ in $\Isom(M,g)$.
\end{enumerate}
\end{Prop}

\begin{proof}
First, let us prove that \ref{item:prop:normalize:HMFD}.
By the arguments in Section~\ref{subsec_Autgrp}, 
\[
\theta(G) = \{ \theta_l= (\Inn_l,\sigma_l) \mid l \in G \}
\]
is a normal subgroup in $\Aut_{\HMFD}(M,G,\sigma)$.
Thus the image $\sigma(G) = \Xi(\theta(G))$ is a normal subgroup of $\Xi(\Aut_{\HMFD}(M,G,\sigma))$.

Next, we shall prove \ref{item:prop:effective:HMFD}.
In order to prove the injectivity of the map 
\[
\Xi : \Aut_{\HMFD}(M,G,\sigma) \longrightarrow \Diff(M), \qquad (\phi,f) \longmapsto f, 
\]
take $(\phi_1,f),(\phi_2,f) \in \Aut_{\HMFD}(M,G,\sigma)$.
We shall prove $\phi_1 = \phi_2$.
Fix $h \in G$ and $x \in M$.
Since the $G$-action $\sigma$ on $M$ is effective, 
we only need to show that $\phi_1(h) \cdot_\sigma x = \phi_2(h) \cdot_\sigma x$.
Recall that $f \in \Diff(M)$, and thus $f$ has the inverse map.
We put $y := f^{-1}(x) \in M$.
Then 
\begin{align*}
\phi_1(h) \cdot_\sigma x 
    &= \phi_1(h) \cdot_\sigma f(y) \\
    &= f(h \cdot_{\sigma} y) \\
    &= \phi_2(h) \cdot_\sigma f(y) \\
    &= \phi_2(h) \cdot_\sigma x.
\end{align*}
This completes the injectivity of the map. 
The last half claim of 
\ref{item:prop:effective:HMFD} 
follows from the claim \ref{item:prop:normalize:HMFD} proved above.

Finally, let us show \ref{item:prop:effective:HRiem}.
Let us consider $G$ and $\Aut_{\HRiem}(M,g,G,\sigma)$ as subgroups of $\Isom(M,g)$.
We denote by $N_{\Isom(M)}(G)$ the normalizer of $G$ in $\Isom(M,g)$.
Then by the claim \ref{item:prop:effective:HMFD} proved above, we have $\Aut_{\HRiem}(M,g,G,\sigma) \subset N_{\Isom(M)}(G)$.
We only need to show the converse inclusion.
Take any $f \in N_{\Isom(M)}(G)$.
Then our goal is to fine $\phi_f \in \Aut(G)$ such that $(\phi_f,f)$ is an equivariant endomorphism on $(M,G,\sigma)$.
Let us define the map $\phi_f$ by 
\[
\phi_f : G \longrightarrow G, \qquad h \longmapsto f \circ h \circ f^{-1}.
\]
Then $\phi_f$ is well-defined (since $f \in N_{\Isom(M)}(G)$), a group automorphism on $G$ and satisfying 
\[
f(h \cdot_\sigma x) = \phi_f(h) \cdot_\sigma f(x)
\]
for any $h \in G$ and any $x \in M$.
It remains to show that the bijective map $\phi_f : G \rightarrow G$ is a diffeomorphism.
Since $M$ has only finitely many connected components, the isometry group $\Isom(M,g)$ is a Lie group with respect to the compact-open topology (cf.~\cite[Theorem 3.4 in Chapter VI]{Kobayashi-Nomizu_I}), and the map 
\[
    \sigma : G \longrightarrow \Isom(M,g), \qquad h \longmapsto \sigma_h
\]
is a Lie group homomorphism.
In fact, since the action of $G$ on $M$ is smooth, it is continuous. 
As $M$ is locally compact, the induced homomorphism
\[
    \sigma: G \longrightarrow C(M,M), \qquad a \longmapsto \sigma_a,
\]
is continuous with respect to the compact-open topology.
Since $\sigma_a\in \Isom(M,g)$ for all $a\in G$ and $\Isom(M,g)$ is a Lie group
with respect to the compact-open topology, the map
\[
    \sigma : G \longrightarrow \Isom(M,g)
\]
is a continuous homomorphism of Lie groups. 
Hence, it is smooth (cf.~\cite[Theorem 3.39]{Warner_1983}).
By the injectivity of $G \rightarrow \Isom(M,g)$, one sees that $G \simeq \sigma(G)$ is a weakly embedded submanifold of $\Isom(M,g)$ (see \cite[Proposition 7.17 and Theorem 19.25]{Lee_Manifold_2013}).
Therefore, the map $\phi_f : G \rightarrow G$ is smooth since 
\[
\sigma \circ \phi_f : G \longrightarrow \Isom(M,g), \qquad h \longmapsto f \circ h \circ f^{-1}
\]
is smooth.
Similarly, one also sees that $(\phi_f)^{-1}$ is smooth.
This completes the proof of \ref{item:prop:effective:HRiem}.
\end{proof}

\section{A category of homogeneous statistical manifolds and $\HS$-moduli spaces of invariant statistical connections}\label{sec:HS-moduli}

\subsection{A category of homogeneous statistical manifolds and equivariant statistical immersions}\label{section:homogStat}
Throughout this paper, 
a data $(M,g,\nabla,G,\sigma)$ is said to be a \emph{homogeneous statistical manifold} 
if $(M,g,\nabla)$ is a statistical manifold, 
$G$ is a Lie group, 
$\sigma$ is a smooth transitive $G$-action on $M$ such that $g$ and $\nabla$ are both $G$-invariant.

\begin{Ex}\label{ex:N0_homog-stat}
For each $\alpha \in \RR$, the quintuple
$(\mathcal{N}_0, g^F, \nabla^{A(\alpha)}, GL(n,\mathbb{R}), \Gamma)$ is a homogeneous statistical manifold.
(See Section~\ref{subsec:N0} and Example~\ref{ex:N0_homog-Riem} for details.)
\end{Ex}

\begin{Def}
Let $(M,g,\nabla,G,\sigma)$ and $(M',g',\nabla',G',\sigma')$ be homogeneous statistical manifolds.
In this paper, a pair $(\phi,f)$ is called an \emph{equivariant statistical immersion} between $(M,g,\nabla,G,\sigma)$ and $(M',g',\nabla',G',\sigma')$ if
$\phi:G\to G'$ is a Lie group homomorphism 
and $f:M\to M'$ is a statistical immersion between $(M,g,\nabla)$ and $(M',g',\nabla')$ (in the sense of Definition~\ref{Definition:DAI})
such that $(\phi,f)$ is an equivariant morphism
between homogeneous manifolds
$(M,G,\sigma)$ and $(M',G',\sigma')$
in the sense of Section~\ref{section:EM}; namely,
\[
f(h\cdot_\sigma x)
=
\phi(h)\cdot_{\sigma'}f(x)
\]
for each $h\in G$ and $x\in M$.
\end{Def}

One can see that homogeneous statistical manifolds and equivariant statistical immersions form a category 
denoted by $\HS = \mathrm{HomogStat}_{\mathrm{EquivStat}}$,
with composition given by
\[
(\phi_2,f_2) \circ (\phi_1,f_1)
:= (\phi_2 \circ \phi_1,\, f_2 \circ f_1).
\]
The identity on $(M,g,\nabla,G,\sigma)$ is $(\mathrm{id}_G,\mathrm{id}_M)$.
An equivariant statistical immersion $(\phi,f)$ is an isomorphism in $\HS$
if and only if $\phi : G \to G'$ is a Lie group isomorphism
and $f : M \to M'$ is a diffeomorphism.

\begin{Def}
Let $\mathcal{C}$ be a class of statistical manifolds in the sense of Definition \ref{Definition:class}.
A homogeneous statistical manifold $(M,g,\nabla,G,\sigma)$ is said to be \emph{of class $\mathcal{C}$} if $(M,g,\nabla)$ is an object of $\mathcal{C}$.
Then homogeneous statistical manifolds of class $\mathcal{C}$ and equivariant statistical immersions form a full replete subcategory of $\HS$, 
which will be denoted by $\HS_{\mathcal{C}}$.    
\end{Def}

\subsection{The space of invariant statistical connections}
Let 
$(M,g,G,\sigma)$ be a homogeneous Riemannian manifold in the sense of Section \ref{section:EM}, 
that is, 
$(M,g)$ is a Riemannian manifold, 
$G$ is a Lie group, 
and $\sigma$ is a smooth transitive isometric $G$-action on $(M,g)$.

As introduced in Section~\ref{subsection:DefStatModuli}, 
we define the subset $\Stat^G(M,g)$ of $\Stat(M,g)$ by
\[
\Stat^G(M,g)
:=
\{\, \nabla \in \Stat(M,g) \mid \nabla \text{ is } G\text{-invariant} \,\}.
\]

Note that for each $\nabla \in \Stat^G(M,g)$, the system $(M,g,\nabla,G,\sigma)$ form a homogeneous statistical manifold in the sense of Section \ref{section:homogStat}.

Let us set
\[
\Gamma(S^3 T^\ast M)^G
:=
\{ C\in\Gamma(S^3T^\ast M)\mid C \text{ is } G\text{-invariant}\}.
\]
Recall that $\Gamma(S^3T^\ast M)$ denotes the space of symmetric $(0,3)$-tensor fields on $M$.

By Theorem~\ref{theorem:nablaC},
we have a natural bijection
\[
\Stat(M,g) \;\longrightarrow\; \Gamma(S^3 T^\ast M),
\]
which restricts to a bijection
\[
\Stat^G(M,g) \;\longrightarrow\; \Gamma(S^3 T^\ast M)^G.
\]

Moreover, it is well known that the space $\Gamma(S^3 T^\ast M)^G$ is finite-dimensional (see, for instance, Proposition~\ref{proposition:StatG-S3-isom} below). 
Since every finite-dimensional real vector space carries a unique Hausdorff topology compatible with its vector space structure, the space $\Gamma(S^3 T^\ast M)^G$ admits a canonical Hausdorff vector space topology.
We therefore endow $\Stat^G(M,g)$ with the Hausdorff topology transported from $\Gamma(S^3 T^\ast M)^G$ via the above bijection.

Fix a base point $p_0\in M$, and let 
\[
K :=G_{p_0}=\{ k \in G\mid k \cdot_\sigma p_0 = p_0\}
\]
be the isotropy subgroup at $p_0$. 
As is well known, the homogeneous space $M$ can be identified with the coset manifold $G/K$ via
\[
G/K \longrightarrow M,\qquad hK \longmapsto h \cdot_\sigma p_0.
\]
Let $\mathfrak{g}$ and $\mathfrak{k}$ denote the Lie algebras of $G$ and $K$, respectively. 
Then the tangent space $T_{p_0}M$ can be identified with the quotient vector space $\mathfrak{g}/\mathfrak{k}$ by
\[
\mathfrak{g}/\mathfrak{k} \longrightarrow T_{p_0}M,\qquad
X+\mathfrak{k} \longmapsto
\left.\frac{d}{dt}\right|_{t=0}\exp(-tX)\cdot_\sigma p_0.
\]
Recall that $T_{p_0}M$ is a $K$-representation, called the isotropy representation defined by
\[
    k \cdot v := (d \sigma_k)_{p_0} (v) = (d ( k \cdot_{\sigma} - ) )_{p_0} (v) 
\]
where $k \in K$ and $v \in T_{p_0} M$.
Since the adjoint representation of $G$ on $\mathfrak{g}$ induces a representation of $K$ on $\mathfrak{g}/\mathfrak{k}$, 
the above correspondence yields an isomorphism of $K$-representations
\[
T_{p_0}M \simeq \mathfrak{g}/\mathfrak{k}.
\]
By a slight abuse of notation, we will also denote the induced $K$--representation on $\mathfrak{g}/\mathfrak{k}$ by $\Ad$.

Let $S^3(T^\ast_{p_0}M)^K \simeq S^3((\mathfrak{g}/\mathfrak{k})^\ast)^K$ denote the spaces of all $K$--invariant symmetric $3$--tensors on the vector spaces $T^\ast_{p_0}M \simeq (\mathfrak{g}/\mathfrak{k})^\ast$. 
Then, by the general theory of invariant sections of equivariant vector bundles over homogeneous spaces, the correspondence
\[ 
\Gamma(S^3 T^\ast M)^G \longrightarrow S^3(T^\ast_{p_0}M)^K, \qquad C \longmapsto C_{p_0}
\]
defines a homeomorphism.
In particular, we obtain the following identifications.

\begin{Prop}\label{proposition:StatG-S3-isom}
There are natural isomorphisms
\[
\Stat^G(M,g)
\simeq \Gamma(S^3T^\ast M)^G
\simeq S^3(T^\ast_{p_0} M)^K
\simeq S^3((\mathfrak{g}/\mathfrak{k})^\ast)^K.
\]
\end{Prop}

\subsection{Definition of $\HS$-moduli spaces of invariant statistical connections}\label{subsec:moduli-inv-stat-conn}
As in the same setting of the previous subsection, 
let $(M,g,G,\sigma)$ be a homogeneous Riemannian manifold in the sense of Section \ref{section:EM}.
We introduce an equivalence relation on the space $\Stat^G(M,g)$ in terms of the category $\HS$ defined in Section~\ref{section:homogStat}.

\begin{Def}
Let $\nabla, \nabla' \in \Stat^G(M,g)$.
We write $\nabla \sim_{\HS} \nabla'$ if the homogeneous statistical manifolds
$(M,g,\nabla,G,\sigma)$ and $(M,g,\nabla',G,\sigma)$
are isomorphic in the category $\HS$ defined in Section~\ref{section:homogStat}.
Then $\sim_{\HS}$ defines an equivalence relation on $\Stat^G(M,g)$.
\end{Def}

For $\nabla \in \Stat^G(M,g)$, we denote by $[\nabla]_{\HS}$ the equivalence class of $\nabla$.

\begin{Def}\label{def:MH-Stat^G}
The \emph{$\HS$-moduli space} of invariant statistical connections on $(M,g,G,\sigma)$ is defined as the quotient space of $\Stat^G(M,g)$ by $\sim_{\HS}$. Such the space will be denoted by $\MHSStat^G(M,g)$. 
\end{Def}

Since each isomorphism in $\HS$ is an isomorphism in $\STAT$, 
we have canonical surjection $\MHSStat^G(M,g)$ onto $\MStatStat^G(M,g)$.

The following proposition characterizes the relation $\nabla \sim_{\HS} \nabla'$ 
for $\nabla, \nabla' \in \Stat^G(M,g)$ in terms of a group action.

\begin{Prop}\label{proposition:sim_AutHRiem}
Let $\nabla, \nabla' \in \Stat^G(M,g)$.
Then $\nabla \sim_{\HS} \nabla'$ if and only if there exists $(\phi,f) \in \Aut_{\HRiem}(M,g,G,\sigma)$ such that $f^* \nabla' = \nabla$ (see Definition \ref{Definition:DAI} for the notation $f^* \nabla'$), 
or equivalently, 
$f^*C' = C$, 
where we put $C := C^{\nabla}$ and $C' := C^{\nabla'}$ for the cubic forms associated with $(M,g,\nabla)$ and $(M,g,\nabla')$, respectively (see Section~\ref{subsec:stat-mfd} for the definition of cubic forms associated with statistical manifolds).
\end{Prop}

\begin{proof}
By the definition of the category $\HS$,
$\nabla \sim_{\HS} \nabla'$ if and only if 
there exists a Lie group automorphism $\phi : G \rightarrow G$ and a diffeomorphism $f : M \rightarrow M$ such that $f$ is a statistical immersion between $(M,g,\nabla)$ and $(M,g,\nabla')$ and $(\phi,f)$ is an equivariant endomorphism on $(M,G,\sigma)$.
Such $(\phi,f)$ is nothing but an automorphism on $(M,g,G,\sigma)$ in the category $\HRiem$ with $f^*C' = C$.
\end{proof}

Let us define the action of $\Aut_{\HRiem}(M,g,G,\sigma)$ on the space $\Stat^G(M,g)$ by putting 
\[
    (\phi, f) \cdot \nabla := (f^{-1})^\ast \nabla
\]
for $\nabla \in \Stat^G(M,g)$ and $(\phi,f) \in \Aut_{\HRiem}(M,g,G,\sigma)$.
Then by Proposition \ref{proposition:sim_AutHRiem}, 
we have 
\[
\MHSStat^G(M,g) := \Stat^G(M,g)/{\sim_{\HS}} = \Aut_{\HRiem}(M,g,G,\sigma) \backslash \Stat^G(M,g).
\]

Recall that $\Aut_{\HRiem}(M,g,G,\sigma)$ is studied in Section~\ref{subsec_Autgrp}.
Especially, by fixing a base point $p_0 \in M$, 
we have a surjective group homomorphism 
\[
\Aut(G,K,g_0) \ltimes \theta(G) \longrightarrow \Aut_{\HRiem}(M,g,G,\sigma),
\qquad
(\phi,\theta_l) \longmapsto (\Inn_l \circ \phi, \sigma_l \circ f_\phi)
\]
as in Theorem \ref{theorem:AutGKthetaontoAutHRiem}.
Via the surjection, the semi-direct product group $\Aut(G,K,g_0) \ltimes \theta(G)$ acts on $\Stat^G(M,g)$.
One can easily see that the normal subgroup $\theta(G)$ acts on $\Stat^G(M,g)$ trivially.
Therefore 
we have the following.

\begin{Thm}\label{thm:MHS-Stat_Aut}
The equality
\begin{equation}
    \MHSStat^G(M,g) = \Aut(G,K,g_0) \backslash \Stat^G(M,g)
\end{equation}
holds, where $\Aut(G,K,g_0)$ acts on $\Stat^G(M,g)$ by
\[
\phi \cdot \nabla := (f_\phi^{-1})^\ast \nabla
\]
for $\nabla \in \Stat^G(M,g)$ and $\phi \in \Aut(G,K,g_0)$
(see~\eqref{eq:def-f_phi} for the definition of $f_\phi$).
\end{Thm}

Recall that $\Stat^G(M,g)$ can be identified with $S^3((\mathfrak{g}/\mathfrak{k})^\ast)^K$.
For each $\phi \in \Aut(G,K,g_0)$ and $C \in S^3((\mathfrak{g}/\mathfrak{k})^\ast)^K$, 
we define $(\phi^{-1})^\ast C \in S^3((\mathfrak{g}/\mathfrak{k})^\ast)$ by putting 
\[
((\phi^{-1})^*C)([X_1]_{\mathfrak{k}},[X_2]_{\mathfrak{k}},[X_3]_{\mathfrak{k}}) := C(\phi_{\mathfrak{g}/{\mathfrak{k}}}^{-1}[X_1]_{\mathfrak{k}},\phi_{\mathfrak{g}/{\mathfrak{k}}}^{-1}[X_2]_{\mathfrak{k}},\phi_{\mathfrak{g}/{\mathfrak{k}}}^{-1}[X_3]_{\mathfrak{k}})
\]
(see~\eqref{eq:def-phi_gk} for the definition of $\phi_{\mathfrak{g}/\mathfrak{k}}$).
Since 
\[
\Ad(k) \circ \phi_{\mathfrak{g}/{\mathfrak{k}}}^{-1} = \phi_{\mathfrak{g}/{\mathfrak{k}}}^{-1} \circ \Ad(\phi^{-1}(k))
\]
holds for each $k \in K$,
we see that $(\phi^{-1})^*C$ is $K$-invariant.
Thus we have a linear representation
of $\Aut(G,K,g_0)$ on $S^3((\mathfrak{g}/\mathfrak{k})^\ast)^K$ by defining 
\[
\phi \cdot C := (\phi^{-1})^*C
\]
for each $\phi \in \Aut(G,K,g_0)$ and $C \in S^3((\mathfrak{g}/\mathfrak{k})^\ast)^K$.

By Proposition \ref{proposition:fphistarphigk}, the linear isomorphism
$(\phi^{-1})^*$ on 
$S^3((\mathfrak{g}/\mathfrak{k})^\ast)^K$ corresponds to the linear isomorphism
$(f_{\phi}^{-1})^*$ on $\Stat^G(M,g)$.
Thus, we obtain the following.
\begin{Thm}\label{theorem:orbitdescription}
The homeomorphisms $\Stat^G(M,g) \simeq S^3( T^\ast_{p_0}M )^K \simeq S^3((\mathfrak{g}/\mathfrak{k})^\ast)^K$ stated above induce homeomorphisms 
$\MHSStat^G(M,g) \simeq \Aut(G,K,g_0) \backslash S^3( T^\ast_{p_0}M )^K \simeq \Aut(G,K,g_0) \backslash S^3((\mathfrak{g}/\mathfrak{k})^\ast)^K$.
\end{Thm}

\subsection{$\HS$-moduli spaces for classes of statistical manifolds}\label{subsec:class-Stat^G}
Let $\mathcal{C}$ be a class of statistical manifolds in the sense of
Definition~\ref{Definition:class}.

\begin{Def}
For each homogeneous Riemannian manifold $(M,g,G,\sigma)$, 
We define
\[
\Stat_{\mathcal{C}}^G(M,g)
:= \bigl\{
\nabla \in \Stat^G(M,g)
\;\big|\;
(M,g,\nabla) \text{ is of class } \mathcal{C}
\bigr\},
\]
which is a subspace of $\Stat^G(M,g)$.    
\end{Def}
    
Via the identifications in Proposition~\ref{proposition:StatG-S3-isom},
the subspace $\Stat_{\mathcal{C}}^G(M,g)$ corresponds to subspaces
\[
\Gamma(S^3T^\ast M)^G_{g\mathchar`-\mathcal{C}},
\qquad
S^3(T^\ast_{p_0} M)^K_{g\mathchar`-\mathcal{C}},
\qquad
S^3((\mathfrak{g}/\mathfrak{k})^\ast)^K_{g\mathchar`-\mathcal{C}},
\]
of
$\Gamma(S^3T^\ast M)^G$,
$S^3(T^\ast_{p_0} M)^K$, and
$S^3((\mathfrak{g}/\mathfrak{k})^\ast)^K$, respectively.

We note that by the definition of 
classes of statistical manifold in the sense of Definition \ref{Definition:class}, the following holds:

\begin{Prop}\label{prop:sim-close_for-C}
For each class $\mathcal{C}$ of statistical manifolds, 
the subset $\Stat_{\mathcal{C}}^G(M,g)$ is closed under the equivalence relation $\sim_{\HS}$.
\end{Prop}

\begin{Def}\label{def:MH-Stat_class-C}
The \emph{$\HS$-moduli space of invariant statistical connections of class $\mathcal C$} on $(M,g,G,\sigma)$ is defined as the quotient space of $\Stat_{\mathcal C}^G(M,g)$ by $\sim_{\HS}$. This space is denoted by $\MHSStat_{\mathcal C}^G(M,g)$.
\end{Def}

By Proposition~\ref{prop:sim-close_for-C}, we have the following.
\begin{Thm}\label{theorem:orbitdescription_classC}
The following homeomorphisms hold.
\begin{align*}
    \MHSStat_{\mathcal{C}}^G(M, g) 
    &\simeq \Aut(G,K,g_0) \backslash \Stat_{\mathcal{C}}^G(M,g)
    \simeq \Aut(G,K,g_0) \backslash \Gamma(S^3T^\ast M)^G_{g\mathchar`-\mathcal{C}}
    \\ &\simeq \Aut(G,K,g_0) \backslash S^3((\mathfrak{g}/\mathfrak{k})^\ast)^K_{g\mathchar`-\mathcal{C}}
    \simeq \Aut(G,K,g_0) \backslash S^3(T^\ast_{p_0} M)^K_{g\mathchar`-\mathcal{C}}
\end{align*}
\end{Thm}

\subsection{A sufficient condition for the $\STAT$-moduli space and the $\HS$-moduli space to coincide}\label{subsec:Stat-HS-moduli_coincide}
As in the previous subsection, let $(M,g,G,\sigma)$ be a homogeneous Riemannian manifold.
In general, the $\STAT$-moduli space $\MStatStat^G(M,g)$ can be regarded as a quotient of the $\HS$-moduli space $\MHSStat^G(M,g)$.
The purpose of this subsection is to provide sufficient conditions under which these two moduli spaces coincide.

\begin{Prop}\label{proposition:SurjStatmodHSmod}
Let us assume that  
\[
\Aut_{\HRiem}(M,g,G,\sigma) \longrightarrow \Isom(M,g), \qquad (\phi,f) \longmapsto f
\]
is surjective.
Then $\MHSStat^G(M,g) = \MStatStat^G(M,g)$.    
\end{Prop}

\begin{proof}
Note that the $\Aut_{\HRiem}(M,g,G,\sigma)$-action on $\Stat^G(M,g)$ factors the map $\Aut_{\HRiem}(M,g,G,\sigma) \rightarrow \Isom(M,g)$.
Thus the subset $\Stat^G(M,g)$ of $\Stat(M,g)$ is stable by $\Isom(M,g)$-action  
since the map $\Aut_{\HRiem}(M,g,G,\sigma) \rightarrow \Isom(M,g)$ is surjective.
Therefore, we have 
\begin{align*}
\MHSStat(M,g)
    &= \Aut_{\HRiem}(M,g,G,\sigma) \backslash \Stat^G(M,g) \\
    &= \Isom(M,g) \backslash \Stat^G(M,g) \\
    &= \MStatStat^G(M,g).
\end{align*}
\end{proof}

\begin{Prop}\label{prop:MHSStat-MStatStat_coincides_2}
Let $p_0 \in M$, and let $K$ be the isotropy subgroup of $G$ at $p_0$.
Moreover, denote by $\Isom(M,g)_{p_0}$ the isotropy subgroup of $\Isom(M,g)$ at $p_0$, and let $g_0$ be the inner product induced by $g$ at $p_0$.
If the group homomorphism
\[
    \Aut(G, K, g_0) \longrightarrow \Isom(M ,g)_{p_0},
    \qquad \phi \longmapsto f_{\phi}
\]
is surjective (see Section~\ref{subsec_Autgrp} for the definition of $\Aut(G,K,g_0)$), then
\[
    \MHSStat^G (M, g) = \MStatStat^G(M, g).
\]
\end{Prop}

\begin{proof}
By Theorem~\ref{thm:MHS-Stat_Aut}, we have
\[
    \MHSStat^G(M ,g) = \Aut(G,K,g_0) \backslash \Stat^G(M,g).
\]
Moreover, by assumption, we also have
\[
    \Aut(G,K,g_0) \backslash \Stat^G(M,g)
    = \Isom(M,g)_{p_0} \backslash \Stat^G(M,g).
\]
Since $(M,g)$ is a $G$-homogeneous Riemannian manifold and the statistical connections are $G$-invariant, we obtain
\[
    \Isom(M,g)_{p_0} \backslash \Stat^G(M,g)
    = \MStatStat^G(M,g).
\]
Therefore, we conclude that $\MHSStat^G (M, g) = \MStatStat^G(M, g)$.
\end{proof}

The following provides a sufficient condition under which Proposition~\ref{proposition:SurjStatmodHSmod} applies:

\begin{Prop}\label{prop:two-moduli-coincide_sufficient}
Assume that the isometric $G$-action $\sigma$ on $(M,g)$ is effective and $M$ has only finitely many connected components.
Let us consider $G$ as a subgroup of $\Isom(M,g)$.
If $G$ is a normal subgroup of $\Isom(M,g)$, 
the map 
\[
\Aut_{\HRiem}(M,g,G,\sigma) \longrightarrow \Isom(M,g), \qquad (\phi,f) \longmapsto f
\]
gives a group isomorphism, and in particular, $\MHSStat^G(M,g) = \MStatStat^G(M,g)$.    
\end{Prop}

\begin{proof}
The first claim comes immediately from Proposition \ref{Proposition:EffectiveNormalizer}.
The latter is the claim of Proposition \ref{proposition:SurjStatmodHSmod}.
\end{proof}

We next present another proposition (cf. Proposition~\ref{prop:two-moduli-coincide_sufficient_2}), which can be regarded as a generalization of Proposition~\ref{prop:two-moduli-coincide_sufficient}.
Let $p_0 \in M$, and let $K$ be the isotropy subgroup of $G$ at $p_0$.
Define
\[
    \overline{G} := G / \ker \bigl(G \to \Isom(M,g)\bigr), 
    \qquad
    \overline{K} := K / \ker \bigl(G \to \Isom(M,g)\bigr).
\]
Then $(M,g)$ admits the structure of a $\overline{G}$-homogeneous Riemannian manifold with effective $\overline{G}$-action.
In particular, we have
\[
    M \simeq \overline{G} / \overline{K}.
\]
Moreover, the kernel $\ker (G \to \Isom(M,g))$ coincides with $\ker (G \to \Diff(M))$ and, furthermore, with
\[
    \Core_{G}(K) := \bigcap_{g \in G} g K g^{-1}.
\]
In particular, $\Core_{G}(K)$ is a closed normal subgroup of $G$ contained in $K$.
In this setting, the following holds.

\begin{Prop}
The map
\[
    \Aut(G,K) \longrightarrow \Aut(\overline{G}, \overline{K}),
    \qquad
    \phi \longmapsto \overline{\phi},
\]
defined by
\[
    \overline{\phi}(h \Core_{G}(K)) := \phi(h)\Core_{G}(K),
\]
is well-defined and is a group homomorphism.
\end{Prop}

\begin{proof}
Let $\phi \in \Aut(G,K)$. 
First, we show that $\phi(\Core_{G}(K)) = \Core_{G}(K)$.
Recall that
\[
    \Core_{G}(K) = \bigcap_{g \in G} gKg^{-1}.
\]
Since $\phi(K)=K$,
\[
    \phi(hKh^{-1}) = \phi(h)K\phi(h)^{-1}
\]
for any $h\in G$.
Therefore, we have $\phi(\Core_{G}(K)) = \Core_{G}(K)$.
It follows that $\phi$ induces a map
\[
    \overline{\phi} : \overline{G} \longrightarrow \overline{G}, \qquad
    h \Core_{G}(K) \longmapsto \phi(h)\Core_{G}(K).
\]
Since $\phi$ is a group automorphism, it follows immediately that $\overline{\phi}$ is a group automorphism of $\overline{G}$.
Moreover, since $\phi(K)=K$, we have
\[
    \overline{\phi}(\overline{K})
    = \overline{\phi}(K / \Core_{G}(K))
    = \phi(K) / \Core_{G}(K)
    = K / \Core_{G}(K)
    = \overline{K},
\]
and hence $\overline{\phi} \in \Aut(\overline{G}, \overline{K})$.
Finally, for $\phi_1, \phi_2 \in \Aut(G,K)$, one can verify that
\[
    \overline{\phi_1 \circ \phi_2}
    =
    \overline{\phi_1}\circ\overline{\phi_2},
\]
and hence the assignment $\phi \mapsto \overline{\phi}$ defines a group homomorphism.
\end{proof}

\begin{Prop}\label{prop:two-moduli-coincide_sufficient_2}
Assume that $M$ has finitely many connected components.
Suppose that $\Isom(M,g)$ normalizes $\sigma(G)$, that is, for any
$f\in\Isom(M,g)$ and $h\in G$,
\[
    f\circ\sigma(h)\circ f^{-1}
    \in
    \sigma(G).
\]
Then the following equality holds:
\[
    \MHSStat^{\overline G}(M,g)
    =
    \MStatStat^G(M,g).
\]
Furthermore, if the actions of
$\Aut(G,K,g_0)$ and
$\Aut(\overline G,\overline K,g_0)$
on $\Stat^G(M,g)$ induce the same orbit decomposition, then
\[
    \MHSStat^G(M,g)
    =
    \MStatStat^G(M,g).
\]
\end{Prop}

\begin{Prop}
If the natural homomorphism
\[
    \Aut(G, K, g_0) \longrightarrow \Aut(\overline{G}, \overline{K}, g_0), \qquad \phi \longmapsto \overline{\phi}
\]
is surjective, then the actions of\/ $\Aut(G, K, g_0)$ and
$\Aut(\overline{G}, \overline{K}, g_0)$ on $\Stat^G(M,g)$
induce the same orbit decomposition.
In particular, if the natural homomorphism
\[
    \Aut(G, K) \longrightarrow \Aut(\overline{G}, \overline{K})
\]
is surjective, then so is
\[
    \Aut(G, K, g_0) \longrightarrow \Aut(\overline{G}, \overline{K}, g_0).
\]
\end{Prop}

\subsection{Simply transitive cases}\label{subsection:simplytransitive}
For homogeneous Riemannian manifolds with simply transitive actions, studying invariant statistical connections and their moduli spaces is equivalent to studying the moduli spaces of left-invariant statistical connections on the corresponding metric Lie groups. 
In this subsection, we explain this in detail.

Let $(M,G,\sigma)$ be a homogeneous manifold in the sense of Section~\ref {section:EM} such that the $G$-action $\sigma$ on $M$ is simply transitive.
We fix a base point $p_0$.
Then one has the diffeomorphism 
\[
    \pi : G \longrightarrow M, \qquad h \longmapsto h \cdot_{\sigma} p_0. 
\]
We denote by $\sigma^L$ the action of a Lie group $G$ on the manifold $G$ by left multiplication, that is,
\[
    h \cdot_{\sigma^L} x := hx \qquad (h \in G,\; x \in G).
\]
Then $(G,G,\sigma^L)$ is a homogeneous manifold in the sense of Section~\ref{section:EM}, and $(\mathrm{id}_G,\pi)$ gives an isomorphism in the category $\HMFD$ between $(M,G,\sigma)$ and $(G,G,\sigma^L)$. 
In particular, the set of $G$-invariant Riemannian metrics on $(M,G,\sigma)$ is in one-to-one correspondence with the set of left-invariant Riemannian metrics on $G$.

Fix a $G$-invariant Riemannian metric $g$ on $M$, and use the same symbol $g$ to denote the corresponding left-invariant Riemannian metric on $G$. 
Then $(M,g,G,\sigma)$ and $(G,g,G,\sigma^L)$ are homogeneous Riemannian manifolds in the sense of Section~\ref{section:EM}, 
and are isomorphic in the category $\HRiem$.
In particular, $\Stat^G(M,g)$ and $\Stat^G(G,g)$ are homeomorphic to each other. 
Moreover, we have
\[
    \MHSStat^G(M,g) \simeq \MHSStat^G(G,g), \quad 
    \MStatStat^G(M,g) \simeq \MStatStat^G(G,g)
\]
(see Section~\ref{subsection:DefStatModuli} for $\MStatStat^G(M,g)$ and Definition~\ref{def:MH-Stat^G} for $\MHSStat^G(M,g)$).
Furthermore, if we fix a class $\mathcal{C}$ of statistical manifolds (see Definition~\ref{Definition:class}), then
\[
    \MHSStat^G_{\mathcal{C}}(M,g) \simeq \MHSStat^G_{\mathcal{C}}(G,g), \quad 
    \MStatStat^G_{\mathcal{C}}(M,g) \simeq \MStatStat^G_{\mathcal{C}}(G,g)
\]
(see Section~\ref{subsection:DefStatModuli} for $\MStatStat^G_{\mathcal{C}}(M,g)$ and Definition~\ref{def:MH-Stat_class-C} for $\MHSStat^G_{\mathcal{C}}(M,g)$).

For the action of a Lie group $G$ on itself by left multiplication, the isotropy subgroup $K$ at the identity is trivial. 
In particular, we have $\Aut(G,K) = \Aut(G)$. 
Moreover, $\mathfrak{g}/\mathfrak{k} \simeq \mathfrak{g} \simeq T_e G$, and a left-invariant Riemannian metric $g$ on $G$ induces an inner product on $\mathfrak{g}$, which we denote by $g_0$. 
For each $\phi \in \Aut(G)$, we denote by $\phi_{\mathfrak{g}}$ the induced automorphism on $\mathfrak{g}$, and define
\[
    \Aut(G,g_0) := \{ \phi \in \Aut(G) \mid \phi \text{ preserves } g \} = \{ \phi \in \Aut(G) \mid \phi_{\mathfrak{g}} \text{ preserves } g_0 \}.
\]
Then $\Aut(G,K,g_0) = \Aut(G,g_0)$. 
It follows that
\[
\begin{aligned}
\MHSStat^G_{\mathcal{C}}(M,g)
&\simeq
\MHSStat^G_{\mathcal{C}}(G,g)
\\
&=
\Aut(G,g_0)
\backslash
\Stat^G_{\mathcal{C}}(G,g)
\\
&\simeq
\Aut(G,g_0)
\backslash
S^3(\mathfrak{g}^{\ast})_{g\mathchar`-\mathcal{C}}.
\end{aligned}
\]
(see Section~\ref{subsec:class-Stat^G} for $S^3(\mathfrak{g}^\ast)_{g\mathchar`-\mathcal{C}}$).
This observation shows that the $\HS$-moduli space introduced in this paper follows the same philosophy as the moduli spaces of left-invariant geometric structures on Lie groups established by Kodama--Takahara--Tamaru \cite{KTT} and others.
These previous studies are reviewed in Section~\ref{subsec:moduli-left-inv-geom-str}.

In \cite[Section 5,6,7]{KOOT_2025}, for several Lie groups $G$, left-invariant Riemannian metrics $g$ on them, and for the classes of statistical manifolds $\mathcal{C}$ given by conjugate symmetry~$\mathrm{CS}$, constant curvature~$\mathrm{CC}$ and dually flat~$\mathrm{DF}$, the explicit forms of
\[
    \Aut(G,g_0) \backslash S^3(\mathfrak{g}^\ast)_{g\mathchar`-\mathcal{C}}
\]
are computed.
For the definition of statistical manifolds of constant curvature, see~\cite{Kurose-1990}.
In particular, dually flat structures have constant curvature, and constant curvature implies conjugate symmetry. 
Combining those results with the above discussion, we obtain the following.

\begin{Thm}\label{thm:KOOT_results}
Let $(M,g,G,\sigma)$ be a homogeneous Riemannian manifold with simply transitive action $\sigma$. 
We denote by $\Homog_k(x_1, \dots, x_n)$ the real vector space of homogeneous polynomials of degree $k$ in $n$ variables over $\RR$.
\begin{enumerate}[label=(\arabic*)]
    \item If $G \simeq \RR^n$ (the $n$-dimensional simply connected abelian Lie group), then
    \begin{align*}
        \MHSStat^G_{\mathrm{CS}}(M,g) 
        &\simeq O(n)\backslash \Homog_3(x_1, \dots, x_n), \\
        \MHSStat^G_{\mathrm{DF}}(M,g)  
        &\simeq \{ (\lambda_1, \dots, \lambda_n) \in \RR^n \mid \lambda_1 \geq \dots \geq \lambda_n \}.
    \end{align*}
    In particular,
    \[
        \MHSStat^G(M,g) = \MHSStat^G_{\mathrm{CS}}(M,g).
    \]

    \item If $G \simeq G_{\RR \mathrm{H}^n}$ $(n \geq 2)$ (the \emph{Lie group of $n$-dimensional real hyperbolic space $\RR \mathrm{H}^n$}, namely, the solvable part of the Iwasawa decomposition of the identity component of $SO(n,1)$), then
    \begin{align*}
        \MHSStat^G_{\mathrm{CS}}(M,g) = \MHSStat^G_{\mathrm{CC}}(M,g)
        &\simeq \RR, \\
        \MHSStat^G_{\mathrm{DF}}(M,g)  
        &\simeq \{\pm 1\}.
    \end{align*}

    \item If $G \simeq H^3 \times \RR^{n-3}$ $(n \geq 4)$ (where $H^3$ is the $3$-dimensional Heisenberg group and $H^3 \times \RR^{n-3}$ denotes the direct product Lie group of $H^3$ and $\RR^{n-3}$), then
    \begin{align*}
        \MHSStat^G_{\mathrm{CS}}(M,g)  
        &\simeq O(n-3) \backslash \bigl( \Homog_1(x_{4}, \dots, x_n) \oplus \Homog_3(x_{4}, \dots, x_n) \bigr), \\
        \MHSStat^G_{\mathrm{DF}}(M,g)   
        &= \emptyset.
    \end{align*}
    If $G \simeq H^3$, then
    \begin{align*}
        \MHSStat^G_{\mathrm{CS}}(M,g)  
        &\simeq \{\ast\}, \\
        \MHSStat^G_{\mathrm{CC}}(M,g)  
        &= \emptyset, \\
        \MHSStat^G_{\mathrm{DF}}(M,g)  
        &= \emptyset.
    \end{align*}
    See also~\cite{IO-2024}.
\end{enumerate}
\end{Thm}

\begin{Rem}
In the case $G \simeq \RR^n$, left-invariant statistical connections with constant curvature (in the sense of Kurose~\cite{Kurose-1990}) can be read off from the results of Opozda~\cite{O-2016}.
It would be interesting to investigate the structure of the corresponding moduli spaces, which appear not to have been explicitly described in the literature.
\end{Rem}

In particular, if \(G \simeq G_{\RRH^n}\), then the results in~\cite{KOOT_2025} imply the following.

\begin{Prop}
Let \((M,g,G,\sigma)\) be a homogeneous Riemannian manifold with simply transitive action \(\sigma\).
If \(G \simeq G_{\RRH^n}\), then
\begin{align}
    \MStatStat^G_{\mathrm{CS}}(M,g)
    &=
    \MHSStat^G_{\mathrm{CS}}(M,g),
    \label{eq:CS-HS-coincide}
    \\
    \MStatStat^G_{\mathrm{DF}}(M,g)
    &=
    \MHSStat^G_{\mathrm{DF}}(M,g).
    \label{eq:DF-HS-coincide}
\end{align}
\end{Prop}

\begin{proof}
It suffices to prove Equation~\eqref{eq:CS-HS-coincide}. 
Moreover, without loss of generality, we may fix a left-invariant Riemannian metric on \(G_{\RRH^n}\) and work with it (cf.~\cite{KOOT_2025}). 
Let \(g\) denote the left-invariant Riemannian metric on \(G_{\RRH^n}\) defined in \cite[Section~6]{KOOT_2025}. 
For each \(\alpha \in \RR\), let $ \nabla^\alpha := \nabla^{(g,C^\alpha)}$,
where \(C^\alpha\) is defined in \cite[Theorem~6.1]{KOOT_2025}.
Then, by the same theorem,
\[
    \{ \nabla^\alpha \mid \alpha \in \RR \}
\]
forms a complete set of representatives of \(\MHSStat^G_{\mathrm{CS}}(M,g)\).
Therefore, it suffices to show that $\nabla^\alpha \nsim_{\STAT} \nabla^{\alpha'}$ for any distinct \(\alpha, \alpha' \in \RR\).

First, it is clear that $\nabla^0 \nsim_{\STAT} \nabla^\alpha$ for every \(\alpha \neq 0\), since \(\nabla^0 = \nabla^g\).
Moreover, by \cite[Theorem 4.2]{Takano-2006}, among the connections \(\nabla^\alpha\), only \(\nabla^{\pm1}\) are dually flat. 
Hence, $\nabla^{\pm1} \nsim_{\STAT} \nabla^\alpha$ for all \(\alpha \neq \pm1\).
Furthermore, we have $\nabla^1 \nsim_{\STAT} \nabla^{-1}$ (cf.~\cite{Shima} and \cite[Remark 6.2]{KOOT_2025}).
On the other hand, one easily verifies that, in general, if statistical manifolds \((N,h,C)\) and \((N',h',C')\) are isomorphic in \(\STAT\), then
\((N,h,rC)\) and \((N',h',rC')\) are also isomorphic in \(\STAT\) for every \(r \neq 0\).
Combining these facts, it follows that the statistical connections
\(\nabla^\alpha\) are mutually non-isomorphic with respect to \(\sim_{\STAT}\).
This completes the proof.
\end{proof}

\section{Main results}\label{sec:main-results}
In this section, we present the main theorems (Theorems~\ref{thm:Main-thm-G=GL+} and \ref{thm:Main-thm-G=Isom}) and provide their proofs.

\subsection{Preliminaries for the main theorems}\label{subsec:Pre-for-mainresults}
In this subsection, we introduce some notation that will be used to state the main theorems.
Let $D$ denote the $n$-dimensional subspace of $\Sym(n,\RR)$ consisting of all diagonal matrices.
We write $S^3(D^\ast)$ for the space of third-order symmetric tensors on the dual space $D^\ast$ of $D$, and $\mathcal{SP}_n^3$ for the space of all $n$-variable homogeneous cubic symmetric polynomials over $\RR$, including the zero polynomial.

The symmetric group $\mathfrak{S}_n$ acts naturally on $D$ by permuting the components. 
We denote by $S^3(D^\ast)^{\mathfrak{S}_n}$ the subspace of
$S^3(D^\ast)$ consisting of all elements fixed under this action.
Both linear spaces $S^3(D^\ast)^{\mathfrak{S}_n}$ and $\mathcal{SP}^3_n$ can be identified with the linear space $\mathcal{P}^3(D)^{\mathfrak{S}_n}$ (the space of
$\mathfrak{S}_n$-invariant homogeneous cubic polynomial functions on $D$)
via the following correspondences:
\[
    S^3(D^\ast)^{\mathfrak{S}_n} \ni C
    \longmapsto q(C) \in \mathcal{P}^3(D)^{\mathfrak{S}_n},
    \quad
    q(C)(d) := C(d,d,d),
\]
\[
    \mathcal{SP}^3_n \ni P(x_1, \dots, x_n)
    \longmapsto f_P \in \mathcal{P}^3(D)^{\mathfrak{S}_n},
    \quad
    f_P(\mathrm{diag}(\lambda_1, \dots, \lambda_n))
        := P(\lambda_1, \dots, \lambda_n).
\]
Throughout this paper, we identify $S^3(D^\ast)^{\mathfrak{S}_n}$ with $\mathcal{SP}^3_n$ via the linear isomorphisms described above.
By the theory of symmetric polynomials, we have the following.
\begin{Lem}\label{lem:SP3n-generator}
$\mathcal{SP}^3_n$ is spanned by the following three homogeneous cubic symmetric polynomials:
\begin{equation}\label{eq:generating-set_SP^3_n}
    f_1 := \sum_{i = 1}^n x_i^3, \qquad f_2 := \sum_{i, j = 1}^n x_i^2 x_j, \qquad f_3 := \sum_{i, j, k = 1}^n x_i x_j x_k. 
\end{equation}
\end{Lem}

Recall that $\mathcal{N}_0$ denotes the space of zero-mean multivariate normal distributions and can be identified with the space $\Sym^+(n,\RR)$
(see Section~\ref{subsec:N0} for details).
Let $C_1$, $C_2$, and $C_3$ be the $(0,3)$-tensor fields on $\mathcal{N}_0$ defined by
\begin{align}
    C_1|_{\Sigma}(X,Y,Z)
        &:=
        \tr(\Sigma^{-1}X\Sigma^{-1}Y\Sigma^{-1}Z),
        \label{eq:def-C_1}
        \\
    C_2|_{\Sigma}(X,Y,Z)
        &:=
        \frac13
        \begin{aligned}[t]
        \bigl(
            &\tr(\Sigma^{-1}X\Sigma^{-1}Y)\tr(\Sigma^{-1}Z)
            \\
            &+\tr(\Sigma^{-1}Y\Sigma^{-1}Z)\tr(\Sigma^{-1}X)
            \\
            &+\tr(\Sigma^{-1}Z\Sigma^{-1}X)\tr(\Sigma^{-1}Y)
        \bigr),
        \end{aligned}
        \label{eq:def-C_2}
        \\
    C_3|_{\Sigma}(X,Y,Z)
        &:=
        \tr(\Sigma^{-1}X)\,
        \tr(\Sigma^{-1}Y)\,
        \tr(\Sigma^{-1}Z).
        \label{eq:def-C_3}
\end{align}
Here we use the natural identification
\[
    \Sym(n,\RR) \ni X
    \longmapsto 
    \left.\frac{d}{dt}\right|_{t=0}(\Sigma + tX)
    \in T_{\Sigma}\mathcal{N}_0 .
\]
It is easy to verify that $C_1$, $C_2$, and $C_3$ are $GL(n,\RR)$-invariant symmetric $(0,3)$-tensor fields on $\mathcal{N}_0$ with respect to the transitive $GL(n,\RR)$-action~\eqref{eq:def-GL-action-N0}.

\begin{Rem}\label{rem:C_1-C^A(+1)_coincide}
The tensor field \(C_1\) coincides with the Amari--Chentsov \((+1)\)-tensor field \(C^{A(+1)}\) on \(\mathcal{N}_0\), and \(\alpha \cdot C_1\) coincides with \(C^{A(\alpha)}\) (see~\cite{Mitchell_1989}). 
In particular, it follows that \(\nabla^{(g^F, \alpha \cdot C_1)} = \nabla^{A(\alpha)}\).
\end{Rem}

\subsection{Main theorems}\label{subsec:main-results}
In this subsection, we present the main theorems (Theorems~\ref{thm:Main-thm-G=GL+} and \ref{thm:Main-thm-G=Isom}).

Recall that the quadruple
$(\mathcal{N}_0, g^F, GL(n, \RR), \Gamma)$
is a homogeneous Riemannian manifold, and that for each $\alpha \in \RR$
the quintuple $(\mathcal{N}_0, g^F, \nabla^{A(\alpha)}, GL(n, \RR), \Gamma)$ is a homogeneous statistical manifold
(see Examples~\ref{ex:N0_homog-Riem} and~\ref{ex:N0_homog-stat}).

Our first main theorem is as follows.

\begin{Thm}\label{thm:Main-thm-G=GL+}
Let $n$ be an integer with $n \ge 1$.
Then the following statements hold.
\begin{enumerate}[label=(\arabic*)]
    \item \label{thm:Main-thm-G=GL_Stat}
    The map
    \[
        \Stat^{GL(n,\RR)}(\mathcal{N}_0^n, g^F)
        \longrightarrow \Gamma (S^3 T^\ast \mathcal{N}_0^n)^{GL(n,\RR)},
        \qquad
        \nabla \longmapsto C^{\nabla},
    \]
    is a homeomorphism.
    Furthermore, there exists a linear isomorphism
    \[
        \Psi : \Gamma( S^3 T^\ast \mathcal{N}_0^n)^{GL(n,\RR)} \longrightarrow \mathcal{SP}_n^3
    \]
    satisfying
    \begin{equation}\label{eq:Psi_C-f}
        \Psi(C^{A(+1)}) = \Psi(C_1) = f_1,
        \qquad
        \Psi(C_2) = f_2,
        \qquad
        \Psi(C_3) = f_3,
    \end{equation}
    where $f_1, f_2, f_3 \in \mathcal{SP}_n^3$ are defined in
    \eqref{eq:generating-set_SP^3_n}.
    In addition, we have
    \begin{equation}\label{eq:S3-GL_classify}
        \Gamma( S^3 T^\ast \mathcal{N}^n_0)^{GL(n,\RR)}
        =
        \begin{cases}
        \mathrm{span}_{\RR}\{C_1, C_2, C_3\} & (n \ge 3),\\[2mm]
        \mathrm{span}_{\RR}\{C_1, C_3\} & (n = 2),\\[2mm]
        \mathrm{span}_{\RR}\{C_1\}
            = \{\, C^{A(\alpha)} \mid \alpha \in \RR \,\}
            & (n = 1).
        \end{cases}
    \end{equation}

    \item \label{thm:Main-thm-G=GL_Hesse}
    \[
    \Stat^{GL(n, \RR)}_{\mathrm{DF}}(\mathcal{N}_0^n, g^F)
        =
        \begin{cases}
        \{\nabla^{A(+1)},\ \nabla^{A(-1)},\ 
          \nabla',\ (\nabla')^\ast \} & (n \ge 3),\\[2mm]
        \{\nabla^{A(+1)},\ \nabla^{A(-1)}\} & (n = 2),\\[2mm]
        \{\nabla^{A(\alpha)} \mid \alpha \in \RR\} & (n = 1).
        \end{cases}
    \]
    Here, $\nabla' := \nabla^{(g^F, C')}$, where $C'$ is a $GL(n, \RR)$-invariant symmetric $(0,3)$-tensor field on $\mathcal{N}_0^n$ defined by
    \begin{equation}\label{eq:def-C'}
        C' = C_1 - \frac{6}{n} C_2 + \frac{4}{n^2} C_3.
    \end{equation}
    See Equations~\eqref{eq:def-C_1}--\eqref{eq:def-C_3} for the definitions of $C_1$, $C_2$, and $C_3$, and Section~\ref{subsec:stat-mfd} for the notation $\nabla^{(g,C)}$.
    Note that $(\nabla')^\ast = \nabla^{(g^F, -C')}$.

    \item \label{thm:Main-thm-G=GL_MStat}
    Let $\eta_1, \eta_2 \in GL(D)$ be defined by
    \begin{align*}
        \eta_1
        &:=
        -\,\mathrm{id}_{D},
        \\
        \eta_2(X)
        &:=
        X
        -
        \frac{2}{n}\tr(X)\,I_n
        \qquad (X\in D).
    \end{align*}
    If $n \ge 3$, then 
    \[
        \MHSStat^{GL(n, \RR)}(\mathcal{N}_0^n, g^F)
        \simeq
        \langle \eta_1, \eta_2 \rangle \backslash S^3(D^\ast)^{\mathfrak{S}_n}
        \simeq
        \langle \eta_1, \eta_2 \rangle \backslash \mathcal{SP}_n^3.
    \]
    If $n = 1$ or $n = 2$, then
    \[
        \MHSStat^{GL(n, \RR)}(\mathcal{N}_0^n, g^F)
        \simeq
        \langle \eta_1 \rangle \backslash S^3(D^\ast)^{\mathfrak{S}_n}
        \simeq
        \langle \eta_1 \rangle \backslash \mathcal{SP}_n^3.
    \]

    \item \label{thm:Main-thm-G=GL_MHesse}
    The following identities hold:
    \begin{gather}
        \sigma_1^\ast \nabla^{A(+1)} = \nabla^{A(-1)}, \qquad 
        \sigma_1^\ast \nabla^{(g^F, C')} = \nabla^{(g^F, -C')}, \label{eq:sigma_1-pullback} \\
        \sigma_2^\ast \nabla^{A(+1)} = \nabla^{(g^F, C')}, \qquad 
        \sigma_2^\ast \nabla^{A(-1)} = \nabla^{(g^F, -C')}. \label{eq:sigma_2-pullback}
    \end{gather}
    Here, $\sigma_1$ and $\sigma_2$ are the isometries of $(\mathcal{N}_0^n, g^F)$ introduced in Proposition~\ref{prop:full-isometry_N0}.
    Moreover, we have
    \begin{equation}\label{eq:Main-thm-G=GL+_MHesse}
    \MHSStat_{\mathrm{DF}}^{GL(n, \RR)}(\mathcal{N}_0^n, g^F)
        \simeq
        \begin{cases}
        \{*\} & (n \ge 2),\\[2mm]
        \RR_{\ge 0} & (n = 1).
        \end{cases}
    \end{equation}
\end{enumerate}
\end{Thm}

We present several supplementary propositions and remarks concerning Theorem~\ref{thm:Main-thm-G=GL+}.
First, the following holds.

\begin{Prop}\label{prop:N0_HS-Stat-moduli-coincides}
The following equality holds.
\[
    \MHSStat^{GL(n, \RR)}(\No, g^F) = \MStatStat^{GL(n, \RR)}(\No, g^F).
\]
\end{Prop}

The proof of Proposition~\ref{prop:N0_HS-Stat-moduli-coincides} follows immediately from Lemma~\ref{lem:surj_Aut-Isom_In}, which will be stated and proved later, together with Proposition~\ref{prop:MHSStat-MStatStat_coincides_2}.
From Proposition~\ref{prop:N0_HS-Stat-moduli-coincides}, the following corollary follows.

\begin{Cor}\label{cor:HS_Stat-moduli_coincides_N0}
The equality
\[
    \MHSStat^{GL(n, \RR)}_{\mathcal{C}}(\No, g^F) = \MStatStat^{GL(n, \RR)}_{\mathcal{C}}(\No, g^F)   
\]
holds for each class $\mathcal{C}$ of statistical manifolds (see Section~\ref{subsec:class-Stat^G} for the definition of the left-hand side and Section~\ref{subsection:DefStatModuli} for that of the right-hand side).
\end{Cor}

\begin{Rem}
The Riemannian manifold $(\No, g^F)$ is a $\overline{GL(n,\RR)}$-homogeneous Riemannian manifold.
In particular, $\overline{GL(n, \RR)}$ acts effectively on $\No$ (see Section~\ref{subsec:Stat-HS-moduli_coincide} for the definition of $\overline{G}$).
Theorem~\ref{thm:Main-thm-G=GL+} remains valid if $GL(n,\RR)$ is replaced by $\overline{GL(n,\RR)}$ throughout.
\end{Rem}

Recall that $GL(n,\RR)$-homogeneous Riemannian manifold $(\No, g^F)$ is also a $GL^+(n,\RR)$-homogeneous Riemannian manifold (see Section~\ref{subsec:N0}).
Theorem~\ref{thm:Main-thm-G=GL+} remains valid if $GL(n,\RR)$ is replaced by $GL^+(n,\RR)$ throughout.
That is, the following holds.

\begin{Prop}
For each class $\mathcal{C}$ of statistical manifolds, the following equalities hold:
\begin{align}
    \label{eq:Stat_GL+-GL_coincide}
    \Stat^{GL^+(n, \RR)}(\mathcal{N}_0^n, g^F)
    &= \Stat^{GL(n,\RR)}(\mathcal{N}_0^n, g^F),\\
    \label{eq:StatDF_GL+-GL_coincide}
    \Stat_{\mathcal{C}}^{GL^+(n, \RR)}(\mathcal{N}_0^n, g^F)
    &= \Stat_{\mathcal{C}}^{GL(n,\RR)}(\mathcal{N}_0^n, g^F),\\
    \label{eq:MStat_GL+-GL_coincide}
    \MHSStat^{GL^+(n, \RR)}(\mathcal{N}_0^n, g^F)
    &= \MHSStat^{GL(n,\RR)}(\mathcal{N}_0^n, g^F),\\
    \label{eq:MStatDF_GL+-GL_coincide}
    \MHSStat_{\mathcal{C}}^{GL^+(n, \RR)}(\mathcal{N}_0^n, g^F)
    &= \MHSStat_{\mathcal{C}}^{GL(n,\RR)}(\mathcal{N}_0^n, g^F).
\end{align}
\end{Prop}

\begin{proof}
We first prove Equation~\eqref{eq:Stat_GL+-GL_coincide}.
By Proposition~\ref{proposition:StatG-S3-isom}, it suffices to show that
\begin{equation*}
    S^3(T^{\ast}_{I_n}\mathcal{N}_0)^{SO(n)}
    \;=\;
    S^3(T^{\ast}_{I_n}\mathcal{N}_0)^{O(n)}.
\end{equation*}
By the $O(n)$-equivariant linear isomorphism
\[
    \Sym(n,\RR) \longrightarrow T_{I_n}\mathcal{N}_0,
    \qquad
    X \longmapsto \left.\frac{d}{dt}\right|_{t=0}(I_n + tX),
\]
it further suffices to show that
\[
    S^3(V^\ast)^{SO(n)} = S^3(V^\ast)^{O(n)}, 
    \qquad (V := \Sym(n, \RR)),
\]
where the action of $A \in O(n)$ on $X \in \Sym(n,\RR)$ is given by $A \cdot X := A X {}^{t}A$.
Let $C(V)^{SO(n)}$ (resp.~$C(V)^{O(n)}$) denote the space of $SO(n)$- (resp.~$O(n)$-) invariant $\RR$-valued continuous functions on $V$.
Then the map
\[
    S^3(V^\ast)^{SO(n)} \longrightarrow C(V)^{SO(n)}, 
    \qquad C \longmapsto q_C
\]
is linear and injective, where $q_C(X) := C(X,X,X)$.
Since any symmetric matrix is diagonalizable by an element of $SO(n)$, we have
\[
    C(V)^{SO(n)} = C(V)^{O(n)}.
\]
Therefore,
\[
    S^3(V^\ast)^{SO(n)} = S^3(V^\ast)^{O(n)}.
\]
This proves Equation~\eqref{eq:Stat_GL+-GL_coincide}.
Equation~\eqref{eq:StatDF_GL+-GL_coincide} follows immediately from Equation~\eqref{eq:Stat_GL+-GL_coincide}.

Next, we prove Equation~\eqref{eq:MStat_GL+-GL_coincide}.
By Proposition~\ref{prop:Aut-GL+_Isom_surj}, which will be stated later, and Proposition~\ref{prop:MHSStat-MStatStat_coincides_2}, we obtain
\[
    \MHSStat^{GL^+(n, \RR)}(\No, g^F)
    = \MStatStat^{GL^+(n, \RR)}(\No, g^F).
\]
Moreover, by Equation~\eqref{eq:Stat_GL+-GL_coincide}, we have
\[
    \MHSStat^{GL^+(n, \RR)}(\No, g^F)
    = \MStatStat^{GL(n, \RR)}(\No, g^F).
\]
Hence, Equation~\eqref{eq:MStat_GL+-GL_coincide} follows from Proposition~\ref{prop:N0_HS-Stat-moduli-coincides}. Equation~\eqref{eq:MStatDF_GL+-GL_coincide} follows from Equation~\eqref{eq:MStat_GL+-GL_coincide}.
\end{proof}

The canonical divergences of the Hessian structures described in Theorem~\ref{thm:Main-thm-G=GL+}~\ref{thm:Main-thm-G=GL_Hesse} can be written explicitly as follows.

\begin{Prop}\label{prop:Amari-open-prob}
By Theorem~\ref{thm:Main-thm-G=GL+}~\ref{thm:Main-thm-G=GL_MHesse}, the canonical divergence in the sense of~\cite{Shima} of the Hessian structure $(g^F, \nabla^{(g^F,-C')})$ coincides with $\sigma_2^\ast D_{\mathrm{KL}}$. 
Here, $D_{\mathrm{KL}}$ denotes the Kullback--Leibler divergence (KL-divergence)
\[
    D_{\mathrm{KL}}(\Sigma_1, \Sigma_2) = \frac{1}{2} \left( \tr(\Sigma_1 \Sigma_2^{-1})
    -
    \log\det( \Sigma_1 \Sigma_2^{-1})
    -
    n \right),
\]
which is the canonical divergence of $(g^F, \nabla^{A(-1)})$. 
Writing $D'$ explicitly, we obtain
\[
    D'(\Sigma_1,\Sigma_2) = \frac{1}{2} \left( \det(\Sigma_1 \Sigma_2^{-1})^{-2/n}\,
        \tr(\Sigma_1 \Sigma_2^{-1})
        - \log \det(\Sigma_1^{-1} \Sigma_2)
        - n \right).
\]
In particular, $\sigma_2$ induces an isomorphism in the category $\HS$ between the two homogeneous Hessian manifolds
\[
(\No, g^F, \nabla^{A(-1)}, GL(n,\RR), \Gamma)
\quad \text{and} \quad
(\No, g^F, \nabla^{(g^F,-C')}, GL(n,\RR), \Gamma)
\]
Indeed, this follows from Corollary~\ref{cor:full-isometry_N0_In}, Lemma~\ref{lem:surj_Aut-Isom_In}, and the arguments in Section~\ref{subsec:moduli-inv-stat-conn}.
Therefore, it can be said that the divergences $D'$ and $D_{\mathrm{KL}}$ are related by an isometry of $(\mathcal{N}^n_0, g^F)$ that is compatible with the $GL(n,\mathbb{R})$-action.
\end{Prop}

\begin{Rem}
Every \(GL(n,\RR)\)-invariant dually flat statistical connection on \((\No^n,g^F)\) induces a Jordan algebra structure on the tangent space \(T_{I_n}\No^n \simeq \Sym(n,\RR)\) (see~\cite[Section 9.3]{Shima} for details).
In particular, as a consequence of
Theorem~\ref{thm:Main-thm-G=GL+}~\ref{thm:Main-thm-G=GL_MHesse},
all such Jordan algebra structures are mutually isomorphic
when \(n \geq 2\).
Moreover, the isomorphisms can be chosen to preserve the inner product
\(g^F_{I_n}\).
\end{Rem}

We now present the second main theorem of this paper.

\begin{Thm}\label{thm:Main-thm-G=Isom}
Let $n$ be an integer with $n \ge 1$, and set $G := \mathrm{Isom}(\mathcal{N}_0^n, g^F)$.
Then the following equalities hold:
\begin{align}
    \Stat^{G}(\mathcal{N}_0^n, g^F)
        &= \{\, \nabla^{g^F} \,\},
        \label{eq:Main-thm-G=Isom_Stat}
    \\
    \Stat_{\mathrm{DF}}^{G}(\mathcal{N}_0^n, g^F)
        &=
        \begin{cases}
        \emptyset & (n \ge 2),\\[2mm]
        \{\, \nabla^{g^F} \,\} & (n = 1),
        \end{cases}
        \label{eq:Main-thm-G=Isom_Hesse}
    \\
    \MHSStat^{G}(\mathcal{N}_0^n, g^F)
        &= \{\, [\nabla^{g^F}] \,\},
        \label{eq:Main-thm-G=Isom_MStat}
    \\
    \MHSStat_{\mathrm{DF}}^{G}(\mathcal{N}_0^n, g^F)
        &=
        \begin{cases}
        \emptyset & (n \ge 2),\\[2mm]
        \{\, [\nabla^{g^F}] \,\} & (n = 1),
        \end{cases}
        \label{eq:Main-thm-G=Isom_MHesse}
\end{align}
Here, $\nabla^{g^F}$ denotes the Levi-Civita connection with respect to $g^F$.
\end{Thm}

\begin{Rem}\label{Remark:strictequivariant}
As mentioned in the Introduction (Section~\ref{sec:Intro}), one may attempt to define a moduli space of $G$-invariant statistical connections using only strictly $G$-equivariant isometries. We explain here that such a definition leads to a trivial equivalence relation.

In this paper, as one of the formulations of the moduli space, we adopt the equivalence relation induced by the action of the group $\Aut_{\HRiem}(M,g,G,\sigma)$ (see Section~\ref{subsec:moduli-inv-stat-conn}).
On the other hand, consider the subgroup of $G$-equivariant isometries of $(M,g)$ defined by
\[
\Isom^G(M, g) := \{ f \in \Isom(M,g) \mid f(h \cdot x) = h \cdot f(x) \text{ for all } h \in G, x \in M \}.
\]
Via the correspondence $f \mapsto (\mathrm{id}_G, f)$, this group can be regarded as a subgroup of $\Aut_{\HRiem}(M,g,G,\sigma)$.
We investigate what happens if one defines the moduli space using the action of $\Isom^G(M,g)$. 

Consider the homogeneous Riemannian manifold
\[
(M,g,G,\sigma) = (\mathcal{N}_0^n, g^F, GL(n,\RR), \Gamma), \qquad n \geq 1
\]
(see Example~\ref{ex:N0_homog-Riem}).
For each $r > 0$, define a map
\[
f_r : \mathcal{N}_0^n \longrightarrow \mathcal{N}_0^n, \qquad \Sigma \longmapsto r\Sigma.
\]
Then $f_r \in \Isom^G(M,g)$. In fact, one can show that
\[
\Isom^G(M,g) = \{ f_r \mid r > 0 \}.
\]
Moreover, since $\sqrt{r} I_n \in GL^+(n,\RR) \subset GL(n,\RR)$ and $f_r = \sigma_{\sqrt{r} I_n}$, it follows that each $f_r$ acts trivially on any $G$-invariant tensor field on $M$. In particular, the induced action of $\Isom^G(M,g)$ on $\Stat^G(M,g)$ is trivial.
Consequently, the equivalence relation induced by $\Isom^G(M,g)$ is trivial, and no nontrivial identifications occur in the corresponding moduli space.  
For example, in the case $n \geq 3$, the four elements of $\Stat^G_{\mathrm{DF}}(M,g)$ that appear in Theorem~\ref{thm:Main-thm-G=GL+}~\ref{thm:Main-thm-G=GL_Hesse} remain distinct under this equivalence.
Similarly, when $n = 2$, the two elements of $\Stat^G_{\mathrm{DF}}(M,g)$ are not identified.

This shows that restricting to strictly $G$-equivariant isometries is too rigid to yield a meaningful moduli space.
\end{Rem}

\subsection{Proof of Theorem~\ref{thm:Main-thm-G=GL+}}\label{subsec:proof-of-main-thm}
In this subsection, we give a proof of Theorem~\ref{thm:Main-thm-G=GL+}.

Throughout this subsection, let $\Xi$ denote the linear isomorphism defined in \eqref{eq:O(n)-equiv-linear-isom}. Namely,
\begin{equation}\label{eq:def-of-Xi}
    \Xi \colon \Sym(n,\RR) \longrightarrow T_{I_n}\mathcal{N}_0,
    \qquad
    X \longmapsto \left.\frac{d}{dt}\right|_{t=0}(I_n + tX).
\end{equation}
Note that the linear isomorphism $\Xi$ is $O(n)$-equivariant, where the action of $A \in O(n)$ on $X \in \Sym(n,\RR)$ is given by
\[
    A \cdot X = A  X {}^{t}A.
\]

We begin with \ref{thm:Main-thm-G=GL_Stat} in Theorem~\ref{thm:Main-thm-G=GL+}.
For its proof, we prepare the following lemma.

\begin{Lem}\label{lem:Chevally-thorem_1}
Let $\iota : D \hookrightarrow \Sym(n,\mathbb{R})$ denote the inclusion of the
subspace of diagonal matrices.
The restriction map
\[
     S^3(\Sym(n,\mathbb{R})^\ast) \longrightarrow S^3(D^\ast),
     \qquad C \longmapsto \iota^{\ast} C,
\]
induces the following linear isomorphism:
\[
     S^3(\Sym(n,\mathbb{R})^\ast)^{O(n)}
     \xrightarrow{\;\sim\;}
     S^3(D^\ast)^{\mathfrak{S}_n}.
\]
\end{Lem}

Lemma~\ref{lem:Chevally-thorem_1} may be regarded as a version of what is often called the
Chevalley restriction theorem (see, for example, \cite[Theorem 2.1.5.1]{Warner_1972}). However, we were unable to find a statement in the literature that matches the exact form of Lemma~\ref{lem:Chevally-thorem_1}.
For the reader's convenience, we remark that an essentially complete proof was provided in our earlier paper by the second and third authors~\cite{Kobayashi-Okuda_2025}, to which we refer the interested reader.

We now give the proof of \ref{thm:Main-thm-G=GL_Stat} in
Theorem~\ref{thm:Main-thm-G=GL+}.

\begin{proof}[Proof of \ref{thm:Main-thm-G=GL_Stat} in Theorem~\ref{thm:Main-thm-G=GL+}]
The map
\[
    \Stat^{GL(n,\RR)}(\mathcal{N}_0, g^F)
        \longrightarrow \Gamma( S^3 T^\ast \mathcal{N}_0)^{GL(n,\RR)}, 
        \qquad
        \nabla \longmapsto C^{\nabla},
\]
is a homeomorphism by Proposition~\ref{proposition:StatG-S3-isom}.
Next, we construct a linear isomorphism
\[
    \Psi : \Gamma( S^3 T^\ast \mathcal{N}_0)^{GL(n,\RR)}
        \longrightarrow \mathcal{SP}_n^3
\]
that satisfies
\[
    \Psi(C^{A(+1)}) = \Psi(C_1) = f_1, \qquad
    \Psi(C_2) = f_2, \qquad
    \Psi(C_3) = f_3.
\]
By Proposition~\ref{proposition:StatG-S3-isom}, the restriction map
\[
    |_{I_n} :
    \Gamma( S^3 T^\ast \mathcal{N}_0)^{GL(n,\RR)}
        \longrightarrow
    S^3(T^\ast_{I_n}\mathcal{N}_0)^{O(n)},
    \qquad
    C \longmapsto C|_{I_n},
\]
is a linear isomorphism.
Moreover, by using the $O(n)$-equivariant linear isomorphism $\Xi$
(see~\eqref{eq:def-of-Xi}), we obtain a linear isomorphism
\[
    \gamma :
    S^3(T^\ast_{I_n}\mathcal{N}_0)^{O(n)}
        \longrightarrow
    S^3(\Sym(n,\RR)^\ast)^{O(n)}.
\]
In particular, from the definitions of $C_1$, $C_2$, and $C_3$
(see Equations~\eqref{eq:def-C_1}--\eqref{eq:def-C_3}), we obtain
\begin{align*}
    \gamma(C_1|_{I_n}) (X,Y,Z) &= \tr(X Y Z), \\
    \gamma(C_2|_{I_n}) (X,Y,Z) &= \tfrac{1}{3} \left(\tr( X Y) \tr(Z) + \tr(Y  Z) \tr( X) + \tr( Z  X) \tr( Y) \right), ~ \text{and} \\
    \gamma(C_3|_{I_n}) (X,Y,Z) &= \tr(X) \tr( Y) \tr( Z).
\end{align*}
By Lemma~\ref{lem:Chevally-thorem_1}, the restriction map
\[
    \rest_D :
    S^3(\Sym(n,\RR)^\ast)^{O(n)}
        \longrightarrow
    S^3(D^\ast)^{\mathfrak{S}_n}
\]
is a linear isomorphism.
For $C \in S^3(D^\ast)$, define
\[
    q(C) : \Sym(n,\RR) \longrightarrow \RR,
    \qquad
    q(C)(X) := C(X,X,X).
\]
Then, for any $\mathrm{diag}(\lambda_1,\dots,\lambda_n) \in D$, we have
\begin{align*}
    (q \circ \rest_D \circ \gamma)(C_1|_{I_n})
        (\mathrm{diag}(\lambda_1,\dots,\lambda_n))
        &= \sum_{i=1}^n \lambda_i^3
            = q(f_1)(\mathrm{diag}(\lambda_1,\dots,\lambda_n)), \\
    (q \circ \rest_D \circ \gamma)(C_2|_{I_n})
        (\mathrm{diag}(\lambda_1,\dots,\lambda_n))
        &= \sum_{i,j=1}^n \lambda_i^2 \lambda_j
            = q(f_2)(\mathrm{diag}(\lambda_1,\dots,\lambda_n)), ~ \text{and} \\
    (q \circ \rest_D \circ \gamma)(C_3|_{I_n})
        (\mathrm{diag}(\lambda_1,\dots,\lambda_n))
        &= \sum_{i,j,k=1}^n \lambda_i \lambda_j \lambda_k
            = q(f_3)(\mathrm{diag}(\lambda_1,\dots,\lambda_n)).
\end{align*}
Here we note that $f_1, f_2, f_3 \in \mathcal{SP}_n^3 \simeq S^3(D^\ast)^{\mathfrak{S}_n}$ are the polynomials defined in~\eqref{eq:generating-set_SP^3_n} of Lemma~\ref{lem:SP3n-generator}.
In general, a symmetric $(0, k)$-tensor $T$ on a vector space is uniquely determined by its diagonal values $T(X, \dots, X)$.
Therefore, we obtain 
\[
    \rest_D( \gamma( {C_1}|_{I_n}) ) = f_1, \quad \rest_D( \gamma( {C_2}|_{I_n}) ) = f_2, \quad \rest_D( \gamma( {C_3}|_{I_n}) ) = f_3.
\]
Thus, if we define $\Psi := \rest_D \circ \gamma \circ |_{I_n}$,
then $\Psi$ is a linear isomorphism and satisfies \eqref{eq:Psi_C-f}.
Finally, we show Equation~\eqref{eq:S3-GL_classify}.
By Lemma~\ref{lem:SP3n-generator}, the set
$\{f_1, f_2, f_3\}$ forms a basis of $\mathcal{SP}_n^3$ for $n \ge 3$, the set $\{f_1, f_3\}$ forms a basis for $n = 2$, and the set $\{f_1\}$ forms a basis for $n = 1$.
Since the linear isomorphism
\[
    \Psi : \Gamma( S^3 T^\ast \mathcal{N}_0)^{GL(n,\RR)}
        \longrightarrow \mathcal{SP}_n^3
\]
satisfies \eqref{eq:Psi_C-f}, the Equation~\eqref{eq:S3-GL_classify} follows.
\end{proof}

We now turn to \ref{thm:Main-thm-G=GL_Hesse} in
Theorem~\ref{thm:Main-thm-G=GL+}.
By the definition of the Amari--Chentsov $\alpha$-connection (see Section~\ref{subsec:N0}), one can see that \ref{thm:Main-thm-G=GL_Hesse} in Theorem~\ref{thm:Main-thm-G=GL+} follows from the following proposition.

\begin{Prop}\label{prop:Hesse-GL+}
Let \(C\) be a \(GL(n,\mathbb{R})\)-invariant symmetric \((0,3)\)-tensor field on \(\mathcal{N}_0^n\).
Then the following are equivalent.
\begin{enumerate}[label=(\roman*)]
    \item\label{thm:hessian-complete-classify_1}
    \((g^F, C)\) is dually flat; that is, \(\nabla^{(g, C)}\) is flat.

    \item\label{thm:hessian-complete-classify_2}
    When \(n \geq 3\), \(C\) is equal to one of
    \[
        C^{A(+1)}, \qquad C^{A(-1)} = - C^{A(+1)}, \qquad C', \qquad -C'.
    \]
    When \(n = 2\), \(C\) is equal to one of
    \[
        C^{A(+1)}, \qquad C^{A(-1)}.
    \]
    When \(n = 1\), \(C\) is equal to one of the Amari--Chentsov \(\alpha\)-tensor fields \(C^{A(\alpha)}\).
\end{enumerate}
\end{Prop}

To prove Proposition~\ref{prop:Hesse-GL+}, we prepare two lemmas
(Lemmas~\ref{lem:sec-curve_C_n=2} and \ref{lem:sec-curve_C_n-geq3}).
Let $\mathcal{B}_F$ denote the basis of $\Sym(n,\RR)$ defined by
\[
    \mathcal{B}_F :=
    \left\{
        \sqrt{2}\,E_{ii}, \;
        E_{jk} + E_{kj}
        \;\middle|\;
        1 \le i \le n,\;
        1 \le j < k \le n
    \right\}.
\]
Here, \( E_{ij} \) denotes the $n \times n$ matrix whose \((i,j)\)-entry is \(1\) and all other entries are \(0\).
When the inner product $g^F_{I_n}$ on $T_{I_n}\mathcal{N}_0$ is transferred to $\Sym(n,\RR)$ via the linear isomorphism $\Xi$ (see~\eqref{eq:def-of-Xi}),
it is straightforward to verify from Proposition~\ref{prop:Fisher-on-N0}
that $\mathcal{B}_F$ forms an orthonormal basis with respect to $g^F_{I_n}$.
In the following two lemmas, we identify $\Sym(n,\RR)$ and 
$T_{I_n}\mathcal{N}_0$ through the isomorphism~$\Xi$.
The following lemmas follow from a direct computation using Definition~\ref{def:CS-sect_basis} and Proposition~\ref{prop:Fisher-curvature}.

\begin{Lem}\label{lem:sec-curve_C_n=2}
Let $a_1, a_3 \in \RR$, and define
\[
    C := a_1 C_1 + a_3 C_3 \in \Gamma( S^3 T^\ast \mathcal{N}_0^2)^{GL(n,\RR)}.
\]
Then the $GL(n,\RR)$-invariant statistical structure $(g^F, C)$ is
conjugate symmetric by Proposition~\ref{prop:GL+-stat_CS}.
Moreover, we have
\begin{align}
    \Sect_{g^F}^{C}(\sqrt{2}\,E_{11}, \sqrt{2}\,E_{22})_{\mathcal{B}_F}
        &= 4 a_1 a_3 \text{ and} \\
    \Sect_{g^F}^{C}(\sqrt{2}\,E_{11}, E_{12} + E_{21})_{\mathcal{B}_F}
        &= -\tfrac{1}{2} + \tfrac{1}{2} a_1^2 + 2 a_1 a_3.
\end{align}
\end{Lem}

\begin{Lem}\label{lem:sec-curve_C_n-geq3}
Assume $n \ge 3$.
Let $a_1, a_2, a_3 \in \RR$, and define
\[
    C := a_1 C_1 + a_2 C_2 + a_3 C_3
        \in \Gamma (S^3 T^\ast \mathcal{N}_0)^{GL(n,\RR)}.
\]
Then the $GL(n,\RR)$-invariant statistical structure $(g^F, C)$ is
conjugate symmetric by Proposition~\ref{prop:GL+-stat_CS}.
Moreover, we have
\begin{align}
4\,\Sect_{g^F}^C(\sqrt{2}\,E_{11},\,E_{12}+E_{21})_{\mathcal{B}_F}
&=
\begin{aligned}[t]
    &-2+2a_1^2
    +\frac{8}{3}a_1(2a_2+3a_3)
    \\
    &\quad
    +\frac{8}{9}a_2
    \bigl(a_2(n+1)+3a_3n\bigr),
\end{aligned}
\\
4\,\Sect_{g^F}^C(\sqrt{2}\,E_{11},\,E_{23}+E_{32})_{\mathcal{B}_F}
&=
\frac{8}{9}
\bigl(
6a_1a_2+a_2^2+9a_1a_3
+a_2(a_2+3a_3)n
\bigr), \text{ and}
\\
4\,\Sect_{g^F}^C(E_{12}+E_{21},\,E_{23}+E_{32})_{\mathcal{B}_F}
&=
-1+a_1^2
+\frac{16}{3}a_1a_2
+\frac{8}{9}a_2^2n.
\end{align}
\end{Lem}

Let us give the proof of Proposition~\ref{prop:Hesse-GL+}.

\begin{proof}[Proof of Proposition~\ref{prop:Hesse-GL+}]
First, we show that \ref{thm:hessian-complete-classify_2} implies \ref{thm:hessian-complete-classify_1}.
When $n = 1$, the manifold $\mathcal{N}_0$ is $1$-dimensional. Hence, the statistical structure $(g^F, C) = (g^F, C^{A(\alpha)})$ is clearly dually flat.
Let $n \geq 2$.
By Proposition~\ref{prop:Fisher-Amari_CS-DF}, the statistical structures \((g^F, C^{A(+1)})\) and \((g^F, C^{A(-1)})\) are dually flat. 
Therefore, it suffices to show that the statistical structure \((g^F, C')\) is dually flat.
Let us recall the diffeomorphism
\[
\sigma_2 : \mathcal{N}_0 \longrightarrow \mathcal{N}_0, 
\qquad
\Sigma \longmapsto (\det\Sigma)^{-2/n}\, \Sigma.
\]
The map $\sigma_2$ is an isometry with respect to $g^F$ (see Proposition~\ref{prop:full-isometry_N0}). Moreover, by Lemma~\ref{lem:involutive-isometry} (stated later), we have
\[
    \sigma_2^\ast C^{A(+1)} = C',
    \qquad
    \sigma_2^\ast C^{A(-1)} = - C'.
\]
Since $(g^F, C^{A(+1)})$ and $(g^F, C^{A(-1)})$ are dually flat, it follows that $(g^F, C')$ and $(g^F, -C')$ are also dually flat. 
This completes the proof of \ref{thm:hessian-complete-classify_2} \(\Rightarrow\) \ref{thm:hessian-complete-classify_1}.

Conversely, we show that \ref{thm:hessian-complete-classify_1} implies \ref{thm:hessian-complete-classify_2}.
Let $C$ be a $GL(n, \RR)$-invariant symmetric $(0,3)$-tensor field on $\mathcal{N}_0$ such that the pair $(g^F, C)$ is dually flat.
When $n = 1$, it follows from Equation~\eqref{eq:S3-GL_classify} in Theorem~\ref{thm:Main-thm-G=GL+}, that $C = C^{A(\alpha)}$ for some $\alpha \in \RR$.
We now prove the statement by considering the cases $n = 2$ and $n \geq 3$ separately.

\textbf{Case \(n=2\):}
By Equation~\eqref{eq:S3-GL_classify} in Theorem~\ref{thm:Main-thm-G=GL+}, there exist $a_1, a_3 \in \RR$ such that
\[
    C = a_1 C_1 + a_3 C_3.
\]
Let $R$ be the curvature tensor field of $\nabla^{(g^F, C)}$.  
Since $(g^F, C)$ is dually flat, we have $R \equiv 0$.  
Therefore, by Proposition~\ref{prop:curvature-formula} and Lemma~\ref{lem:sec-curve_C_n=2}, we obtain the following.
\begin{align}
0 &= 4 a_1 a_3, \label{eq:seccurv=0-1} \\
0 &= - \tfrac{1}{2} +  \tfrac{1}{2} a_1^2 + 2 a_1 a_3. \label{eq:seccurv=0-2}
\end{align} 
The solution to the above system of equations is
\[
    (a_1, a_3) = (1, 0) \text{ or } (-1, 0).
\]
Therefore, $C$ coincides with $C_1 = C^{A(+1)}$ or with $-C_1 = C^{A(-1)}$.
This completes the proof for the case $n=2$.

\textbf{Case \(n \ge 3\):}
By Equation~\eqref{eq:S3-GL_classify} in Theorem~\ref{thm:Main-thm-G=GL+}, there exist $a_1, a_2, a_3 \in \RR$ such that
\[
    C = a_1 C_1 + a_2 C_2 + a_3 C_3.
\]
As in the case \(n=2\), Proposition~\ref{prop:curvature-formula} together with Lemma~\ref{lem:sec-curve_C_n-geq3} yields the following equalities:
\begin{align}
0 &= -2 + 2 a_1^2 + \frac{8}{3} a_1 (2 a_2 + 3 a_3)
      + \frac{8}{9} a_2 \bigl(a_2 (n+1) + 3 a_3 n \bigr),
      \label{eq:seccurv=0-1} \\
0 &= \frac{8}{9} \bigl( 6 a_1 a_2 + a_2^2 + 9 a_1 a_3
      + a_2 (a_2 + 3 a_3) n \bigr),
      \label{eq:seccurv=0-2} \\
0 &= -1 + a_1^2 + \frac{16}{3} a_1 a_2 + \frac{8}{9} a_2^2 n.
      \label{eq:seccurv=0-3}
\end{align}
We now solve this system of equations, since these equalities are necessary for the cubic tensor \(C\) to define a dually flat structure.
By subtracting Equation~\eqref{eq:seccurv=0-2} from
Equation~\eqref{eq:seccurv=0-1}, we obtain
\[
   2(-1 + a_1^2) = 0.
\]
Hence \(a_1 = \pm 1\).
Substituting these values into Equation~\eqref{eq:seccurv=0-3} yields
\[
\begin{cases}
a_1 = 1: & a_2 (6 + a_2 n) = 0 \ \Longleftrightarrow\ a_2 = 0,\ -\frac{6}{n}, \\[2mm]
a_1 = -1: & a_2 (-6 + a_2 n) = 0 \ \Longleftrightarrow\ a_2 = 0,\ \frac{6}{n}.
\end{cases}
\]
Furthermore, substituting each resulting pair into
Equation~\eqref{eq:seccurv=0-2} yields the following possibilities:
\begin{enumerate}[label=(\arabic*)]
\item \( (a_1, a_2, a_3) = (1, 0, 0), \)
\item \( (a_1, a_2, a_3) = (-1, 0, 0), \)
\item \( (a_1, a_2, a_3) = \left(1, -\frac{6}{n}, \frac{4}{n^2}\right), \)
\item \( (a_1, a_2, a_3) = \left(-1, \frac{6}{n}, -\frac{4}{n^2}\right). \)
\end{enumerate}
This completes the proof for the case \(n \ge 3\).
\end{proof}

Next, we proceed to prove \ref{thm:Main-thm-G=GL_MStat} in Theorem~\ref{thm:Main-thm-G=GL+}.  
For the proof, we prepare the following two lemmas.

\begin{Lem}\label{lem:surj_Aut-Isom_In}
The group homomorphism
\[
    \Aut (GL(n, \RR), O(n), g^F_{I_n}) \longrightarrow \mathrm{Isom}(\mathcal{N}_0, g^F)_{I_n}, \qquad \phi \longmapsto f_\phi
\]
is surjective.
\end{Lem}

\begin{Lem}\label{lem:rest-intertwiner}
Let $L$ be a finite group, $V$ a finite-dimensional vector space over $\RR$, $\rho : L \to GL(V)$ a linear representation, and $V_0 \subset V$ a subrepresentation, that is, $V_0$ is a subspace of $V$ such that $(\rho(L))(V_0) \subset V_0$.
We also denote by $\iota : V_0 \to V$ the inclusion map.
Then, the restriction map
\[
    \rest : S^k(V^\ast) \longrightarrow S^k(V_0^\ast), \qquad T \longmapsto \iota^\ast T 
\]
is $L$-intertwiner, that is, $\rest (l \cdot C) = l \cdot \rest(C)$ holds for any $C \in S^k(V^\ast)$ and $l \in L$.
\end{Lem}

The proof of Lemma~\ref{lem:rest-intertwiner} is straightforward.  
We now give the proof of Lemma~\ref{lem:surj_Aut-Isom_In}.

\begin{proof}[Proof of Lemma~\ref{lem:surj_Aut-Isom_In}]
By Corollary~\ref{cor:full-isometry_N0_In}, the following holds:
\begin{align*}
    \mathrm{Isom}(\mathcal{N}^n_0, g^F)_{I_n} &=
    \left\{
    \begin{array}{ll}
    \left\{ \Gamma(g) \circ h ~\middle|~ g \in O(n), ~ h \in \langle \sigma_1, \sigma_2 \rangle \right\}  & (n \geq 3),\\
    \left\{ \Gamma(g) \circ h ~\middle|~ g \in O(n), ~ h \in \langle \sigma_1 \rangle \right\}  & (n = 1 \text{ or } 2).
    \end{array}\right.
\end{align*}
Let us put $G := GL(n, \RR)$.
Recall that $(\No^n, g^F, G, \Gamma)$ is a homogeneous Riemannian manifold.
Here, 
\[ 
    \Gamma(A, \Sigma) = A \Sigma {}^{t}A.
\]
Let us fix a point $I_n \in \No^n$ and put $K := GL(n, \RR)_{I_n} = O(n)$.
It is clear that for each $A \in O(n)$, $\Inn(A) \in \Aut(G, K)$ holds, and one can see that $f_{\Inn(A)} = \Gamma(A)$.
Therefore, for each $A \in O(n)$, $\Inn(A) \in \Aut(G, K, g^F_{I_n})$ and $f_{\Inn(A)} = \Gamma(A)$.
Therefore, it remains only to show that, for each $i = 1,2$, there exists $\phi \in \Aut(G, K, g^F_{I_n})$ such that $f_{\phi} = \sigma_i$.
Let 
\[
    \Tilde{\sigma}_1 (A) := {}^{t}A^{-1}, \qquad \Tilde{\sigma}_2 (A) := \left( (\det A \right)^2 )^{-1/n} A.
\]
Then, one can check that $\Tilde{\sigma}_1, \Tilde{\sigma}_2 \in \Aut(G, K)$,
\[
    f_{\tilde{\sigma}_1} = \sigma_1, \quad \text{and} \quad f_{\tilde{\sigma}_2} = \sigma_2.
\]
Thus, $\tilde{\sigma}_1, \tilde{\sigma}_2 \in \Aut(G, K, g^F_{I_n})$.
Therefore, the group homomorphism in Lemma~\ref{lem:surj_Aut-Isom_In} is surjective.
\end{proof}

By a similar argument, we obtain the following.

\begin{Prop}\label{prop:Aut-GL+_Isom_surj}
The following homomorphism is surjective.
\[
    \Aut (GL^+(n, \RR), SO(n), g^F_{I_n}) \longrightarrow \mathrm{Isom}(\mathcal{N}_0, g^F)_{I_n}, \qquad \phi \longmapsto f_\phi
\]
\end{Prop}

We now give the proof of \ref{thm:Main-thm-G=GL_MStat} in Theorem~\ref{thm:Main-thm-G=GL+}.

\begin{proof}[Proof of \ref{thm:Main-thm-G=GL_MStat} in Theorem~\ref{thm:Main-thm-G=GL+}]
Let $G := GL(n, \RR)$ and $K := O(n)$.
By Theorem~\ref{theorem:orbitdescription}, we have
\[
    \MHSStat^G(\No, g^F) \simeq \Aut(G, K, g^F_{I_n}) \backslash S^3(T^\ast_{I_n} \No)^K.
\]
By Lemma~\ref{lem:surj_Aut-Isom_In}, the group homomorphism
\[
    \Aut(G, K, g^F_{I_n}) \longrightarrow \Isom(\No, g^F)_{I_n}, \qquad \phi \longmapsto f_\phi
\]
is surjective.
Therefore, 
\[
    \MHSStat^G(\No, g^F) \simeq \Isom(\No, g^F)_{I_n} \backslash S^3(T^\ast_{I_n} \No)^{K}.
\]
By Lemma~\ref{cor:full-isometry_N0_In}, we have
\begin{align*}
    \MHSStat^G(\No, g^F) &\simeq
    \left\{
    \begin{array}{ll}
    (\Gamma(O(n)) \rtimes \langle \sigma_1, \sigma_2 \rangle) \backslash S^3(T^\ast_{I_n} \No)^{K}  & (n \geq 3),\\
    (\Gamma(O(n)) \rtimes \langle \sigma_1 \rangle) \backslash S^3(T^\ast_{I_n} \No)^{K}  & (n = 1 \text{ or } 2).
    \end{array}\right.
\end{align*}
Thus,
\begin{align*}
    \MHSStat^G(\No, g^F) &\simeq
    \left\{
    \begin{array}{ll}
    \langle \sigma_1, \sigma_2 \rangle \backslash S^3(T^\ast_{I_n} \No)^{K}  & (n \geq 3),\\
    \langle \sigma_1 \rangle \backslash S^3(T^\ast_{I_n} \No)^{K}  & (n = 1 \text{ or } 2).
    \end{array}\right.
\end{align*}
Let us put $\tilde{\eta}_1, \tilde{\eta}_2 \in GL(\Sym(n, \RR))$ as $\tilde{\eta}_1 (X) := -X$ and $\tilde{\eta}_2 (X) := X - \tfrac{2}{n} \tr(X) I_n$.
Since $\Xi$ is $K$-equivariant, by Proposition~\ref{prop:diff_sigma_1-2},  we have
\begin{align*}
    \MHSStat^G(\No, g^F) &\simeq
    \left\{
    \begin{array}{ll}
    \langle \tilde{\eta}_1, \tilde{\eta}_2 \rangle \backslash S^3(\Sym(n, \RR)^\ast)^{K}  & (n \geq 3),\\
    \langle \tilde{\eta}_1 \rangle \backslash S^3(\Sym(n, \RR)^\ast)^{K}  & (n = 1 \text{ or } 2).
    \end{array}\right.
\end{align*}
It is clear that $\tilde{\eta}_1 (D) \subset D$ and $\tilde{\eta}_2 (D) \subset D$.
Thus, by Lemmas~\ref{lem:Chevally-thorem_1} and \ref{lem:rest-intertwiner}, we have
\begin{align*}
    \MHSStat^G(\No, g^F) &\simeq
    \left\{
    \begin{array}{ll}
    \langle \tilde{\eta}_1|_{D}, \tilde{\eta}_2|_{D} \rangle \backslash S^3(D^\ast)^{\mathfrak{S}_n}  & (n \geq 3),\\
    \langle \tilde{\eta}_1|_{D} \rangle \backslash S^3(D^\ast)^{\mathfrak{S}_n}  & (n = 1 \text{ or } 2).
    \end{array}\right.
\end{align*}
Since $\tilde{\eta}_1 = \eta_1$ and $\tilde{\eta}_2 = \eta_2$ on $D$, we also have
\begin{align*}
    \MHSStat^G(\No, g^F) &\simeq
    \left\{
    \begin{array}{ll}
    \langle \eta_1, \eta_2 \rangle \backslash S^3(D^\ast)^{\mathfrak{S}_n}  & (n \geq 3),\\
    \langle \eta_1 \rangle \backslash S^3(D^\ast)^{\mathfrak{S}_n}  & (n = 1 \text{ or } 2).
    \end{array}\right.
\end{align*}
This completes the proof.
\end{proof}

Finally, we turn to \ref{thm:Main-thm-G=GL_MHesse} in Theorem~\ref{thm:Main-thm-G=GL+}.  
The proof follows from the following lemma.

\begin{Lem}\label{lem:involutive-isometry}
Let 
\[
\begin{aligned}
\sigma_1:\mathcal{N}_0 &\longrightarrow \mathcal{N}_0,
&
\Sigma &\longmapsto \Sigma^{-1},
\\
\sigma_2:\mathcal{N}_0 &\longrightarrow \mathcal{N}_0,
&
\Sigma &\longmapsto (\det\Sigma)^{-2/n}\Sigma.
\end{aligned}
\]
Then, both $\sigma_1$ and $\sigma_2$ are isometries with respect to $g^F$ (see Proposition~\ref{prop:full-isometry_N0}). 
Moreover, the following holds.
\begin{gather*}
    \sigma_1^\ast C^{A(+1)} = -C^{A(+1)} = C^{A(-1)}, \qquad \sigma_1^\ast C' = -C',\\
    \sigma_2^\ast C^{A(+1)} = C'.
\end{gather*}
\end{Lem}

First, assuming Lemma~\ref{lem:involutive-isometry} above, we prove 
\ref{thm:Main-thm-G=GL_MHesse} in Theorem~\ref{thm:Main-thm-G=GL+}.

\begin{proof}[Proof of \ref{thm:Main-thm-G=GL_MHesse} in Theorem \ref{thm:Main-thm-G=GL+}]
Equations \eqref{eq:sigma_1-pullback} and \eqref{eq:sigma_2-pullback} follow from Lemma~\ref{lem:involutive-isometry}.

Let us prove Equation~\eqref{eq:Main-thm-G=GL+_MHesse}.
Let $G = GL(n, \RR)$ and $K := O(n)$.
By Theorem~\ref{theorem:orbitdescription_classC}, we have
\[
    \MHSStat_{\mathrm{DF}}^G(\No, g^F) \simeq \Aut(G, K, g^F_{I_n}) \backslash S^3(T^\ast_{I_n} \No)^{K}_{g^F\mathchar`-\mathrm{DF}}.
\]
By Lemma~\ref{lem:surj_Aut-Isom_In}, 
\[
    \MHSStat_{\mathrm{DF}}^G(\No, g^F) \simeq \Isom(\No, g^F)_{I_n} \backslash S^3(T^\ast_{I_n} \No)^{K}_{g^F\mathchar`-\mathrm{DF}}.
\]
Let $n \geq 2$.
By Corollary~\ref{cor:full-isometry_N0_In}, $\sigma_1, \sigma_2 \in \Isom(\No^n, g^F)_{I_n}$.
Thus, by Proposition~\ref{prop:Hesse-GL+} and Lemma~\ref{lem:involutive-isometry}, the action of $\Isom(\No^n, g^F)_{I_n}$ on $S^3(T^\ast_{I_n} \No^n)^{K}_{g^F\mathchar`-\mathrm{DF}}$ is transitive. 
Therefore, $\MHSStat_{\mathrm{DF}}^{G}(\No^n, g^F)$ is a singleton when $n \geq 2$.
Let $n = 1$.
Then, by Equation~\eqref{eq:S3-GL_classify}, 
\[
    S^3(T^\ast_{I_n} \No)^{K}_{g^F\mathchar`-\mathrm{DF}} = \left\{ \alpha \cdot C_1|_{I_n} ~\middle|~ \alpha \in \RR \right\}  \simeq \RR.
\]
Since the differential of $\sigma_1$ at $I_n$ satisfies $(d\sigma_1)_{I_n} = -\mathrm{id}$, 
the quotient 
\[
    \Isom(\mathcal{N}_0^n, g^F)_{I_n} 
    \backslash 
    S^3(T^\ast_{I_n}\mathcal{N}_0^n)^{K}_{g^F\mathchar`-\mathrm{DF}}
\]
is isomorphic to $\mathbb{R}_{\ge 0}$ when $n = 1$.
\end{proof}

Let us give the proof of Lemma~\ref{lem:involutive-isometry}.

\begin{proof}[Proof of Lemma~\ref{lem:involutive-isometry}]
Recall that $C^{A(+1)}, C' \in \Gamma( S^3 T^\ast \No)^{GL(n, \RR)}$.
By Proposition~\ref{proposition:StatG-S3-isom}, it suffices to carry out the discussion only on $T_{I_n}\mathcal{N}_0 \simeq \Sym(n, \RR)$.
By Proposition~\ref{prop:diff_sigma_1-2}, one can see that $((d \sigma_1)_{I_n})^\ast C = -C$ holds for every $C \in S^3(T^\ast_{I_n} \No)$.
Thus, it is enough to show that 
\begin{equation*}
    ((d \sigma_2)_{I_n})^\ast ((C_1)_{I_n}) = (C')_{I_n}.
\end{equation*}
By Equations~\eqref{eq:def-C_1}--\eqref{eq:def-C_3}
and~\eqref{eq:def-C'},
the tensors $(C_1)_{I_n}$ and $(C')_{I_n}$ are given on $\Sym(n,\RR)$ by
\begin{align*}
(C_1)_{I_n}(X,Y,Z)
&=
\tr(XYZ),
\\
(C')_{I_n}(X,Y,Z)
&=
(C_1)_{I_n}(X,Y,Z)
\\
&\quad
-\frac{2}{n}
\bigl(
\tr(XY)\tr(Z)
+\tr(YZ)\tr(X)
+\tr(ZX)\tr(Y)
\bigr)
\\
&\quad
+\frac{4}{n^2}
\tr(X)\tr(Y)\tr(Z).
\end{align*}
Moreover, by Proposition~\ref{prop:diff_sigma_1-2}, we have
\[
(d\sigma_2)_{I_n}(X)
=
X-\frac{2}{n}\tr(X)I_n.
\]
Therefore, writing $C_1$ and $C'$ for $(C_1)_{I_n}$ and $(C')_{I_n}$, respectively, we obtain
\begin{align*}
    ((d\sigma_2)_{I_n})^\ast C_1(X,Y,Z)
    &=
    C_1\bigl(
    (d\sigma_2)_{I_n}(X),
    (d\sigma_2)_{I_n}(Y),
    (d\sigma_2)_{I_n}(Z)
    \bigr)
    \\
    &=
    C_1\Bigl(
    X-\tfrac{2}{n}\tr(X)I_n,\,
    Y-\tfrac{2}{n}\tr(Y)I_n,\,
    Z-\tfrac{2}{n}\tr(Z)I_n
    \Bigr)
    \\
    &=
    C_1(X,Y,Z)
    \\
    &\quad
    -\frac{2}{n}
    \begin{aligned}[t]
    \Bigl(
    &\tr(X)C_1(I_n,Y,Z)
    \\
    &+\tr(Y)C_1(X,I_n,Z)
    \\
    &+\tr(Z)C_1(X,Y,I_n)
    \Bigr)
    \end{aligned}
    \\
    &\quad
    +\frac{4}{n^2}
    \begin{aligned}[t]
    \Bigl(
    &\tr(X)\tr(Y)C_1(I_n,I_n,Z)
    \\
    &+\tr(Y)\tr(Z)C_1(X,I_n,I_n)
    \\
    &+\tr(Z)\tr(X)C_1(I_n,Y,I_n)
    \Bigr)
    \end{aligned}
    \\
    &\quad
    -\frac{8}{n^3}
    \tr(X)\tr(Y)\tr(Z)C_1(I_n,I_n,I_n)
    \\
    &=
    C_1(X,Y,Z)
    \\
    &\quad
    -\frac{2}{n}
    \bigl(
    \tr(X)\tr(YZ)
    +\tr(Y)\tr(XZ)
    +\tr(Z)\tr(XY)
    \bigr)
    \\
    &\quad
    +\frac{4}{n^2}
    \tr(X)\tr(Y)\tr(Z)
    \\
    &=
    C'(X,Y,Z).
\end{align*}
Therefore, $((d \sigma_2)_{I_n})^\ast ((C_1)_{I_n}) = (C')_{I_n}$ holds.
This completes the proof.
\end{proof}

\subsection{Proof of Theorem~\ref{thm:Main-thm-G=Isom}}\label{subsec:proof-of-main-thm-Isom}
In this subsection, we provide the proof of Theorem~\ref{thm:Main-thm-G=Isom}.

For the proof, we prepare the following two lemmas.

\begin{Lem}\label{lem:Isom-inv-cubic}
Let $(M,g)$ be a Riemannian (globally) symmetric space in the sense of Kobayashi--Nomizu~\cite[Chapter XI, Section 6]{Kobayashi-Nomizu_II}, and let $\nabla$ be a statistical connection on $(M, g)$.
If the statistical connection $\nabla$ is invariant under the action of $\mathrm{Isom}(M,g)$, then $\nabla = \nabla^{g}$ holds.
\end{Lem}

\begin{proof}
Let $C$ be the cubic form associated with $(M,g,\nabla)$.
It suffices to show that $C = 0$.
For each point $p \in M$, denote by $s_p$ the symmetry at $p$ of the Riemannian symmetric space $(M,g)$.
Then $s_p \in \Isom(M,g)$.
At each point $p \in M$, the differential $(ds_p)_p$ coincides with $-\mathrm{id}_{T_p M}$.
Hence,
\[
((ds_p)_p)^\ast C_p = - C_p .
\]
Since $\nabla$ is invariant under the action of $\Isom(M,g)$, the cubic form $C$ is invariant as well.
Therefore, for every $p \in M$ we have
\[
C_p = ((ds_p)_p)^\ast C_p = - C_p,
\]
which implies $C_p = 0$.
Thus $C = 0$.
This completes the proof.
\end{proof}

The following lemma follows immediately from Proposition~\ref{prop:Fisher-curvature}.

\begin{Lem}
Let $n \geq 2$. 
Then, $\nabla^{g^F}$ is not flat.
\end{Lem}

Note that $(\No, g^F)$ is a Riemannian (globally) symmetric space. 
Let us prove Theorem~\ref{thm:Main-thm-G=Isom}.

\begin{proof}[Proof of Theorem~\ref{thm:Main-thm-G=Isom}]
Equation~\eqref{eq:Main-thm-G=Isom_Stat} follows immediately from Lemma~\ref{lem:Isom-inv-cubic}.
Equation~\eqref{eq:Main-thm-G=Isom_Hesse} follows from Equation~\eqref{eq:Main-thm-G=Isom_Stat} and Corollary~\ref{cor:Fisher-flatness}.
Equations~\eqref{eq:Main-thm-G=Isom_MStat} and \eqref{eq:Main-thm-G=Isom_MHesse} are clear from
Equations~\eqref{eq:Main-thm-G=Isom_Stat} and \eqref{eq:Main-thm-G=Isom_Hesse}, respectively.
\end{proof}

\section*{Acknowledgements}
The authors are grateful to Michel Nguiffo Boyom for his valuable questions following the second author's presentation at the 7th International Conference on Geometric Science of Information (GSI 2025), which served as one of the motivations for the present work.
We also thank Ting-Kam Leonard Wong for his insightful comments and questions during GSI 2025.
The authors would also like to acknowledge Hajime Fujita, Hitoshi Furuhata, Hiroto Inoue, Kyosuke Ishimoto, Yu Ohno, Hiroshi Tamaru, Koichi Tojo, and Ryu Ueno for their valuable comments and suggestions.


\begin{thebibliography}{10}

\bibitem{Amari_1985}
S.-i. Amari.
\newblock {\em Differential-Geometrical Methods in Statistics}, volume~28 of {\em Lecture Notes in Statistics}.
\newblock Springer-Verlag, New York, 1985.
\newblock \doi{10.1007/978-1-4612-5056-2}

\bibitem{Amari_2014}
S.-i. Amari.
\newblock Information Geometry of Positive Measures and Positive-Definite Matrices: Decomposable Dually Flat Structure.
\newblock {\em Entropy}, 16(4):2131--2145, 2014.
\newblock \doi{10.3390/e16042131}

\bibitem{Amari-Nagaoka_2000}
S.-i. Amari and H.~Nagaoka.
\newblock {\em Methods of Information Geometry}, volume 191 of {\em Translations of Mathematical Monographs}.
\newblock American Mathematical Society, Providence, RI; Oxford University Press, Oxford, 2000.
\newblock \doi{10.1090/mmono/191}

\bibitem{Ay_2015}
N.~Ay, J.~Jost, H.~V. L\^e, and L.~Schwachh\"ofer.
\newblock Information geometry and sufficient statistics.
\newblock {\em Probab. Theory Related Fields}, 162(1-2):327--364, 2015.
\newblock \doi{10.1007/s00440-014-0574-8}

\bibitem{Ay_2017}
N.~Ay, J.~Jost, H.~V. L\^e, and L.~Schwachh\"ofer.
\newblock {\em Information Geometry}, volume~64 of {\em Ergebnisse der Mathematik und ihrer Grenzgebiete. 3. Folge. A Series of Modern Surveys in Mathematics}.
\newblock Springer, Cham, 2017.
\newblock \doi{10.1007/978-3-319-56478-4}

\bibitem{Ayadi_2023}
I.~Ayadi, F.~Bouchard, and F.~Pascal.
\newblock Elliptical Wishart Distribution: Maximum Likelihood Estimator from Information Geometry.
\newblock In {\em IEEE International Conference on Acoustics, Speech and Signal Processing (ICASSP)}, pages 1--5, 2023.
\newblock \doi{10.1109/ICASSP49357.2023.10096222}

\bibitem{Burbea_1986}
J.~Burbea.
\newblock Informative Geometry of Probability Spaces.
\newblock {\em Exposition. Math.}, 4(4):347--378, 1986.

\bibitem{CCR_2025}
I.~Cardoso, A.~Cosgaya, and S.~Reggiani.
\newblock The moduli space of left-invariant metrics on six-dimensional characteristically solvable nilmanifolds.
\newblock {\em Math. Nachr.}, 298(5):1496--1520, 2025.
\newblock \doi{10.1002/mana.202400213}

\bibitem{Cosgaya-Reggiani_2022}
A.~Cosgaya and S.~Reggiani.
\newblock Isometry groups of three-dimensional Lie groups.
\newblock {\em Ann. Global Anal. Geom.}, 61(4):831--845, 2022.
\newblock \doi{10.1007/s10455-022-09835-3}

\bibitem{Dekic-Babic-Vukmirovic_2022}
A.~Deki\'c, M.~Babi\'c, and S.~Vukmirovi\'c.
\newblock Classification of Left Invariant Riemannian Metrics on Complex Hyperbolic Space.
\newblock {\em Mediterr. J. Math.}, 19(5):Paper No. 232, 18, 2022.
\newblock \doi{10.1007/s00009-022-02152-w}

\bibitem{Dolcetti-Pertici_2019}
A.~Dolcetti and D.~Pertici.
\newblock Differential properties of spaces of symmetric real matrices.
\newblock {\em Rend. Semin. Mat. Univ. Politec. Torino}, 77(1):25--43, 2019.

\bibitem{Furuhata_2009}
H.~Furuhata.
\newblock Hypersurfaces in statistical manifolds.
\newblock {\em Differential Geom. Appl.}, 27(3):420--429, 2009.
\newblock \doi{10.1016/j.difgeo.2008.10.019}

\bibitem{Furuhata-Hasegawa_2016}
H.~Furuhata and I.~Hasegawa.
\newblock Submanifold Theory in Holomorphic Statistical Manifolds.
\newblock In {\em Geometry of {C}auchy-{R}iemann submanifolds}, pages 179--215. Springer, Singapore, 2016.
\newblock \doi{10.1007/978-981-10-0916-7_7}

\bibitem{FIK}
H.~Furuhata, J.-i.~Inoguchi, and S.-P.~Kobayashi.
\newblock A characterization of the alpha-connections on the statistical manifold of normal distributions.
\newblock {\em Inf. Geom.}, 4(1):177--188, 2021.
\newblock \doi{10.1007/s41884-020-00037-z}

\bibitem{Hashinaga_2014}
T.~Hashinaga.
\newblock On the minimality of the corresponding submanifolds to four-dimensional solvsolitons.
\newblock {\em Hiroshima Math. J.}, 44(2):173--191, 2014.
\newblock \doi{10.32917/hmj/1408972906}

\bibitem{Hashinaga-Tamaru_2017}
T.~Hashinaga and H.~Tamaru.
\newblock Three-dimensional solvsolitons and the minimality of the corresponding submanifolds.
\newblock {\em Internat. J. Math.}, 28(6):1750048, 31, 2017.
\newblock \doi{10.1142/S0129167X17500483}

\bibitem{HTT_2016}
T.~Hashinaga, H.~Tamaru, and K.~Terada.
\newblock Milnor-type theorems for left-invariant {R}iemannian metrics on {L}ie groups.
\newblock {\em J. Math. Soc. Japan}, 68(2):669--684, 2016.
\newblock \doi{10.2969/jmsj/06820669}

\bibitem{Inoguchi_2026}
J.-i. Inoguchi.
\newblock Holomorphic statistical 2-manifolds and almost contact statistical 3-manifolds.
\newblock {\em Differential Geom. Appl.}, 103:Paper No. 102366, 2026.
\newblock \doi{10.1016/j.difgeo.2026.102366}

\bibitem{IO-2024}
J.-i. Inoguchi and Y.~Ohno.
\newblock Homogeneous statistical manifolds.
\newblock {\em Inf. Geom.}, 8(2):285--341, 2025.
\newblock \doi{10.1007/s41884-025-00172-5}

\bibitem{Ivanov_1995}
S.~Ivanov.
\newblock On dual-projectively flat affine connections.
\newblock {\em J. Geom.}, 53(1-2):89--99, 1995.
\newblock \doi{10.1007/BF01224043}

\bibitem{KOOT_2025}
H.~Kobayashi, Y.~Ohno, T.~Okuda, and H.~Tamaru.
\newblock The moduli spaces of left-invariant statistical structures on {L}ie groups.
\newblock {\em Inf. Geom.}, 9(1):235--278, 2026.
\newblock \doi{10.1007/s41884-026-00200-y}

\bibitem{Kobayashi-Okuda_2025}
H.~Kobayashi and T.~Okuda.
\newblock On Invariant Conjugate Symmetric Statistical Structures on the Space of Zero-Mean Multivariate Normal Distributions.
\newblock In {\em Geometric Science of Information. Part I}, volume~16033 of {\em Lecture Notes in Comput. Sci.}, pages 65--72. Springer, Cham, 2026.
\newblock \doi{10.1007/978-3-032-03918-7\_7}

\bibitem{Kobayashi-Nomizu_I}
S.~Kobayashi and K.~Nomizu.
\newblock {\em Foundations of Differential Geometry. {V}ol {I}}.
\newblock Interscience Publishers (a division of John Wiley \& Sons, Inc.), New York-London, 1963.

\bibitem{Kobayashi-Nomizu_II}
S.~Kobayashi and K.~Nomizu.
\newblock {\em Foundations of Differential Geometry. {V}ol. {II}}, volume No. 15 of {\em Interscience Tracts in Pure and Applied Mathematics}.
\newblock Interscience Publishers John Wiley \& Sons, Inc., New York-London-Sydney, 1969.

\bibitem{KO}
S.-P.~Kobayashi and Y.~Ohno.
\newblock A characterization of the alpha-connections on the statistical manifold of multivariate normal distributions.
\newblock {\em Osaka J. Math.}, 62(2):329--349, 2025.
\newblock \doi{10.18910/101132}

\bibitem{KTT}
H.~Kodama, A.~Takahara, and H.~Tamaru.
\newblock The space of left-invariant metrics on a {L}ie group up to isometry and scaling.
\newblock {\em Manuscripta Math.}, 135(1-2):229--243, 2011.
\newblock \doi{10.1007/s00229-010-0419-4}

\bibitem{KONDO}
Y.~Kondo.
\newblock A classification of left-invariant pseudo-{R}iemannian metrics on some nilpotent {L}ie groups.
\newblock {\em Hiroshima Math. J.}, 52(3):333--356, 2022.
\newblock \doi{10.32917/h2021054}

\bibitem{KT_2023}
Y.~Kondo and H.~Tamaru.
\newblock A classification of left-invariant {L}orentzian metrics on some nilpotent {L}ie groups.
\newblock {\em Tohoku Math. J. (2)}, 75(1):89--117, 2023.
\newblock \doi{10.2748/tmj.20211122}

\bibitem{KOTT}
A.~Kubo, K.~Onda, Y.~Taketomi, and H.~Tamaru.
\newblock On the moduli spaces of left-invariant pseudo-{R}iemannian metrics on {L}ie groups.
\newblock {\em Hiroshima Math. J.}, 46(3):357--374, 2016.
\newblock \doi{10.32917/hmj/1487991627}

\bibitem{Kurose-1990}
T.~Kurose.
\newblock Dual connections and affine geometry.
\newblock {\em Math. Z.}, 203(1):115--121, 1990.
\newblock \doi{10.1007/BF02570725}

\bibitem{Kurose_1994}
T.~Kurose.
\newblock On the divergences of {$1$}-conformally flat statistical manifolds.
\newblock {\em Tohoku Math. J. (2)}, 46(3):427--433, 1994.
\newblock \doi{10.2748/tmj/1178225722}

\bibitem{L-SM}
S.~L. Lauritzen.
\newblock Statistical Manifolds.
\newblock In {\em Differential geometry in statistical inference}, volume~10, pages 163--216. Institute of Mathematical Statistics, Hayward, 1987.
\newblock \doi{10.1214/lnms/1215467061}

\bibitem{Lee_Manifold_2013}
J.~M. Lee.
\newblock {\em Introduction to Smooth Manifolds}, volume 218 of {\em Graduate Texts in Mathematics}.
\newblock Springer, New York, second edition, 2013.
\newblock \doi{10.1007/978-1-4419-9982-5}

\bibitem{Matsuzoe_2010}
H.~Matsuzoe.
\newblock Statistical manifolds and affine differential geometry.
\newblock In {\em Probabilistic approach to geometry}, volume~57 of {\em Adv. Stud. Pure Math.}, pages 303--321. Math. Soc. Japan, Tokyo, 2010.
\newblock \doi{10.2969/aspm/05710303}

\bibitem{Mitchell_1989}
A.~F.~S. Mitchell.
\newblock The information matrix, skewness tensor and {$\alpha$}-connections for the general multivariate elliptic distribution.
\newblock {\em Ann. Inst. Statist. Math.}, 41(2):289--304, 1989.
\newblock \doi{10.1007/BF00049397}

\bibitem{Opozda_2015}
B.~Opozda.
\newblock Bochner's technique for statistical structures.
\newblock {\em Ann. Global Anal. Geom.}, 48(4):357--395, 2015.
\newblock \doi{10.1007/s10455-015-9475-z}

\bibitem{O-2016}
B.~Opozda.
\newblock A sectional curvature for statistical structures.
\newblock {\em Linear Algebra Appl.}, 497:134--161, 2016.
\newblock \doi{10.1016/j.laa.2016.02.021}

\bibitem{Opozda_2019}
B.~Opozda.
\newblock Curvature bounded conjugate symmetric statistical structures with complete metric.
\newblock {\em Ann. Global Anal. Geom.}, 55(4):687--702, 2019.
\newblock \doi{10.1007/s10455-019-09647-y}

\bibitem{Luis_2022}
L.~Pedro Castellanos~Moscoso.
\newblock Left-invariant symplectic structures on diagonal almost abelian {L}ie groups.
\newblock {\em Hiroshima Math. J.}, 52(3):357--378, 2022.
\newblock \doi{10.32917/h2021055}

\bibitem{Luis-T}
L.~Pedro Castellanos~Moscoso and H.~Tamaru.
\newblock A classification of left-invariant symplectic structures on some {L}ie groups.
\newblock {\em Beitr. Algebra Geom.}, 64(2):471--491, 2023.
\newblock \doi{10.1007/s13366-022-00643-1}

\bibitem{Salamon_2001}
S.~M. Salamon.
\newblock Complex structures on nilpotent {L}ie algebras.
\newblock {\em J. Pure Appl. Algebra}, 157(2-3):311--333, 2001.
\newblock \doi{10.1016/S0022-4049(00)00033-5}

\bibitem{Shima}
H.~Shima.
\newblock {\em The Geometry of Hessian Structures}.
\newblock World Scientific Publishing Co. Pte. Ltd., Hackensack, NJ, 2007.
\newblock \doi{10.1142/9789812707536}

\bibitem{Shima-Hao_2000}
H.~Shima and J.~H. Hao.
\newblock Geometry associated with normal distributions.
\newblock {\em Osaka J. Math.}, 37(2):509--517, 2000.

\bibitem{Skovgaard_1984}
L.~T. Skovgaard.
\newblock A Riemannian Geometry of the Multivariate Normal Model.
\newblock {\em Scand. J. Statist.}, 11(4):211--223, 1984.

\bibitem{Sukilovic-Vukmirovic_2023}
T.~\v Sukilovi\'c and S.~Vukmirovi\'c.
\newblock Geometry of cotangent bundle of {H}eisenberg group.
\newblock {\em Differential Geom. Appl.}, 88:Paper No. 101997, 11, 2023.
\newblock \doi{10.1016/j.difgeo.2023.101997}

\bibitem{Sukilovic-et.al_2025}
T.~\v Sukilovi\'c, S.~Vukmirovi\'c, and N.~Bokan.
\newblock On the moduli space of left-invariant metrics on the cotangent bundle of the {H}eisenberg group.
\newblock {\em Rev. Un. Mat. Argentina}, 68(2):485--518, 2025.
\newblock \doi{10.33044/revuma.3426}

\bibitem{Takano-2006}
K.~Takano.
\newblock Geodesics on statistical models of the multivariate normal distribution.
\newblock {\em Tensor (N.S.)}, 67(2):162--169, 2006.

\bibitem{Taketomi_2015}
Y.~Taketomi.
\newblock Examples of hyperpolar actions of the automorphism groups of {L}ie algebras.
\newblock {\em Topology Appl.}, 196:904--910, 2015.
\newblock \doi{10.1016/j.topol.2015.05.057}

\bibitem{Taketomi_2018}
Y.~Taketomi.
\newblock On a {R}iemannian submanifold whose slice representation has no nonzero fixed points.
\newblock {\em Hiroshima Math. J.}, 48(1):1--20, 2018.
\newblock \doi{10.32917/hmj/1520478020}

\bibitem{Taketomi-Tamaru_2018}
Y.~Taketomi and H.~Tamaru.
\newblock On the nonexistence of left-invariant {R}icci solitons---a conjecture and examples.
\newblock {\em Transform. Groups}, 23(1):257--270, 2018.
\newblock \doi{10.1007/s00031-017-9439-4}

\bibitem{Scott_2017}
S.~V. Thuong.
\newblock Metrics on 4-dimensional unimodular {L}ie groups.
\newblock {\em Ann. Global Anal. Geom.}, 51(2):109--128, 2017.
\newblock \doi{10.1007/s10455-016-9527-z}

\bibitem{Torres_2024}
J.~F. Torres.
\newblock Moduli space of bi-invariant metrics.
\newblock {\em Bol. Soc. Mat. Mex. (3)}, 30(2):Paper No. 49, 17, 2024.
\newblock \doi{10.1007/s40590-024-00622-7}

\bibitem{Warner_1983}
F.~W. Warner.
\newblock {\em Foundations of Differentiable Manifolds and Lie Groups}, volume~94 of {\em Graduate Texts in Mathematics}.
\newblock Springer-Verlag, New York-Berlin, 1983.
\newblock \doi{10.1007/978-1-4757-1799-0}

\bibitem{Warner_1972}
G.~Warner.
\newblock {\em Harmonic Analysis on Semi-Simple Lie Groups I}, volume~188 of {\em Die Grundlehren der mathematischen Wissenschaften}.
\newblock Springer-Verlag, New York-Heidelberg, 1972.
\newblock \doi{10.1007/978-3-642-50275-0}

\end{thebibliography}
\end{document}